\documentclass{article}

\usepackage{hyperref}
\usepackage{amsmath,amssymb, amsmath,amsthm, amsopn}
\usepackage{cite}
\usepackage{graphicx}
\usepackage{hyperref}
\usepackage{adjustbox}
\usepackage[margin=1cm]{caption}
\usepackage{fancyhdr}
\usepackage[utf8]{inputenc}
\usepackage[T1]{fontenc}
\usepackage{etex}
\usepackage{caption}
\usepackage{authblk}
\usepackage{geometry}
\usepackage{lipsum}
\usepackage{amsfonts}
\usepackage{graphicx}
\usepackage{epstopdf}
\usepackage{algorithmic}
\usepackage{enumitem}
\usepackage{color}
\ifpdf
  \DeclareGraphicsExtensions{.eps,.pdf,.png,.jpg}
\else
  \DeclareGraphicsExtensions{.eps}
\fi

\theoremstyle{remark}
\newtheorem{proposition}{Proposition}

\newtheorem{remark}{Remark}
\newtheorem{assumption}{Assumption}

\newcommand{\R}{\mathbb{R}}
\newcommand{\N}{\mathbb{N}}
\renewcommand{\bar}{\overline}

\newcommand{\calI}{\mathcal{I}}
\newcommand{\calJ}{\mathcal{J}}

\newcommand{\calD}{\mathcal{D}}

\newcommand{\calR}{\mathcal{R}}

\newcommand{\calS}{\mathcal{S}}
\newcommand{\calT}{\mathcal{T}}

\newcommand{\xdag}{x^\dagger}
\newcommand{\Phidag}{\Phi^\dagger}
\newcommand{\ydel}{y^\delta}
\newcommand{\xad}{x_\alpha^\delta}
\newcommand{\Phiad}{\Phi_\alpha^\delta}

\newcommand{\ul}[1]{\underline{#1}}
\newcommand{\ol}[1]{\overline{#1}}

\newcommand{\setof}[2]{\left\{#1:#2\right\}}
\newcommand{\argmin}{\mbox{argmin}}
\newcommand{\Mad}{M_{ad}}
\newcommand{\Tikh}{T}

\newcommand{\spc}{ \mbox{~~~} }

\begin{document}

\title{Regularization of inverse problems \\ via box constrained minimization
\thanks{This work was partially supported by the Austrian Science Fund FWF
 under grants I2271 and P30054.}
}
\author{Philipp Hungerl\"ander}
\author{Barbara Kaltenbacher}
\author{Franz Rendl}
\affil{Alpen-Adria-Universit\"at Klagenfurt, Austria\\
 \href{mailto:firstname.surname@aau.at}{firstname.surname@aau.at}
}

\maketitle

\begin{abstract}

In the present paper we consider minimization based formulations of
inverse problems
$(x,\Phi)\in\argmin\setof{\mathcal{J}(x,\Phi;y)}{(x,\Phi)\in M_{ad}(y)}$ for
the specific but highly relevant case that the admissible set $M_{ad}^\delta(y^\delta)$ is
defined by pointwise bounds, which is the case, e.g., if $L^\infty$
constraints on the parameter are imposed in the sense of Ivanov
regularization, and the $L^\infty$ noise level in the observations is
prescribed in the sense of Morozov regularization. As application
examples for this setting we consider three coefficient identification
problems in elliptic boundary value problems.

Discretization of $(x,\Phi)$ with piecewise constant and piecewise linear finite elements,
respectively, leads to finite dimensional nonlinear box constrained
minimization problems that can numerically be solved via Gauss-Newton
type SQP methods. In our computational experiments we revisit the
suggested application examples. In order to speed up the computations
and obtain exact numerical solutions we use recently developed active set methods
for solving strictly convex quadratic programs with box
constraints as subroutines within our Gauss-Newton type SQP approach.
\end{abstract}


\maketitle

\section{Introduction}

Recently, as alternatives to the classical reduced formulation of inverse problems as operator equations with a forward operator $F$
\begin{equation}\label{Fxy}
F(x)=y\ ,
\end{equation}
all-at once methods based on the more original formulation as a system of model and observation equation
\begin{align}
&A(x,\Phi)=0\ , \label{Axu}\\
&C(\Phi)=y\ , \label{Cuy}
\end{align}
and beyond that, minimization based formulations
\begin{equation}\label{minJ}
(x,\Phi)\in\argmin\setof{\calJ(x,\Phi;y)}{(x,\Phi)\in \Mad(y)}\ ,
\end{equation}
have been put forward, see, e.g., \cite{minIP,Kindermann}.
In
\eqref{Fxy} -- \eqref{minJ} $x$ is the searched for quantity (e.g., a coefficient in a PDE), $\Phi$ the corresponding state (e.g., the solution of this PDE) and $y$ the observed data. The forward operator $F$ in \eqref{Fxy} is related to the model and observation operators $A:X\times U\to W$, $C:U\to Y$ via the parameter-to-state map $S:X\to U$ determined by the identity
\begin{equation}\label{AxSx0}
A(x,S(x))=0\ , \quad \forall x\in\calD\,,
\end{equation}
which, if well-defined, allows to eliminate the state and define $F=C\circ S:\calD(\subseteq X)\to Y$ in the classical reduced formulation \eqref{Fxy}.

The need for a parameter-to-state map $S$ often leads to restrictions in the domain $\calD$ when working with the classical formulation \eqref{Fxy}. Moreover, numerical evaluation of $F$ in, e.g., iterative methods for solving \eqref{Fxy}, requires solution of the underlying PDE model in each step. Our intention is to avoid these drawbacks by using regularization strategies based on the formulations \eqref{Axu}, \eqref{Cuy} or more generally \eqref{minJ}, thus avoiding the use of $S$.

There are several ways of writing the reduced \eqref{Fxy} and the all-at-once \eqref{Axu}, \eqref{Cuy} formulations as special cases of the minimization based form \eqref{minJ}, e.g. by setting
\[
\calJ(x,\Phi;y)=\calS(F(x),y)\,,\qquad
\Mad(y)=\calD\times \{\Phi_0\}\,,
\]
with some fixed dummy state $\Phi_0\in U$ and some positive definite functional
$\calS:Y\times Y\to\ol{\R}$, i.e., such that
\begin{equation}\label{Sdefinite}
\forall y_1,y_2\in Y\, : \quad \calS(y_1,y_2)\geq0 \quad \mbox{ and }\quad \Bigl(y_1=y_2 \ \Leftrightarrow \ \calS(y_1,y_2)=0\Bigr)\,,
\end{equation}
or
\[
\calJ(x,\Phi;y)=\calS(C(\Phi),y)+\calI_{\{0\}}(A(x,\Phi))\,,
\qquad\Mad(y)=\calD\times U\,,
\]
with the indicator function $\calI_M:W\to\ol{\R}$ defined by
$\calI_M(w)=\begin{cases}0\mbox{ if }w\in M\\ +\infty\mbox{
    else}\end{cases}$\hspace*{-0.4cm} , or
\[
\calJ(x,\Phi;y)=\calS(C(\Phi),y) \,,
\qquad\Mad(y)=\setof{(x,\Phi)\in\calD\times U}{A(x,\Phi)=0}\,,
\]
or
\begin{equation}\label{minJ_Morozov}
\calJ(x,\Phi;y)=\calJ_A(x,\Phi)\,,
\qquad \Mad(y)=\setof{(x,\Phi)\in\calD\times U}{C(\Phi)=y}\,,
\end{equation}
with $\calJ_A:X\times U\to\R$ such that
\begin{equation}\label{Qdefinite}
\forall x,\Phi\in X\times U\, : \quad \calJ_A(x,\Phi)\geq0 \quad\mbox{ and }\quad\Bigl(A(x,\Phi)=0 \ \Leftrightarrow \ \calJ_A(x,\Phi)=0\Bigr)\,,
\end{equation}
e.g., $\calJ_A(x,\Phi)=\frac{1}{2}\|A(x,\Phi)\|_W^2$ or $\calJ_A(x,\Phi)=\calI_{\{0\}}(A(x,\Phi))$.

Additionally, there are relevant minimization based formulations that cannot be cast into the framework of \eqref{Fxy} nor \eqref{Axu}, \eqref{Cuy}, such as the variational formulation of the electrical impedance tomography problem EIT cf., e.g., \cite{Knowles1998,KohnMcKenny90,KohnVogelius87}.

In \cite{minIP} we provide an analysis of regularization methods
\begin{align}\label{minJR}
(\xad,\Phiad)\in\argmin\{\Tikh_\alpha(x,\Phi;\ydel)=\calJ(x,\Phi;\ydel)+\alpha\cdot\calR(x,\Phi)
  : (x,\Phi)\in \Mad^\delta(y^\delta)\}\,,
\end{align}
based on the formulation \eqref{minJ} and investigate its applicability to several concrete choices of the cost function and the admissible set, based on \eqref{Axu}, \eqref{Cuy} as well as on the mentioned variational formulation of EIT. Here $\ydel$ are the actually given noisy data satisfying
\begin{equation}\label{delta}
\calS(y,\ydel)\leq\delta\,,
\end{equation}
and $\calR$, $\alpha$ are regularization functional and parameter, respectively.

The present paper is supposed to provide computational results in the specific but highly relevant case that  $\Mad^\delta(y^\delta)$ is defined by pointwise bounds, which is the case, e.g., if $L^\infty$ constraints on the parameter are imposed in the sense of Ivanov regularization, and the $L^\infty$ noise level in the observations is prescribed in the sense of Morozov regularization, i.e., starting from \eqref{minJ}, \eqref{minJ_Morozov}, a regularizer is defined via the minimization problem
\begin{equation}\label{minJRaaoM}
\left\{
\begin{aligned}
&\min_{(x,\Phi)\in X\times U}\Tikh_\alpha(x,\Phi;\ydel)=\calJ_A(x,\Phi)+\alpha\cdot\calR(x,\Phi)
\\
&\hspace*{0.5cm} \mbox{s.t. }\ul{x}\leq x \leq \ol{x}\ , \mbox{ and }
\ydel-\tau\delta\leq C(\Phi)\leq \ydel+\tau\delta\,,
\end{aligned}
\right.
\end{equation}
where $\tau>1$ is a fixed safety factor for the error level in the observation residual, $\ul{x}$, $\ol{x}$ are given (pointwise) bounds on $x$, and $\calS(y_1,y_2)=\|y_1-y_2\|_{L^\infty}$. Here we think of $x$ and $\ydel$ as functions on some domains $\Omega$, $\tilde{\Omega}$, and of $L^\infty$ as the corresponding Lebesgue spaces.
Accordingly, the inequality constraints in \eqref{minJRaaoM} are to be understood pointwise almost everywhere in $\Omega$ and $\tilde{\Omega}$, respectively.
Moreover, $C$ is supposed to be a linear operator, more precisely, a restriction of the state, e.g. to some subdomain or some part of the boundary of $\Omega$.

Regularization here mainly relies on the upper and lower bounds on $x$ as well as on relaxation of the data misfit constraint from \eqref{minJ_Morozov} to an interval $[-\tau\delta,\tau\delta]$. The term $\alpha\cdot\calR$ may as well be skipped, as we will see in some of the examples in Section \ref{sec:appex}, where we consider just
\[
\min_{(x,\Phi)\in X\times U}\calJ_A(x,\Phi)
\hspace*{0.5cm} \mbox{s.t. }\ul{x}\leq x \leq \ol{x}\ , \mbox{ and }
\ydel-\tau\delta\leq C(\Phi)\leq \ydel+\tau\delta\,.
\]

Therewith, simple discretizations of \eqref{minJRaaoM} with, e.g.,
piecewise linear and/or piecewise constant finite elements, lead to
finite dimensional nonlinear box constrained minimization problems. Thus, in
their numerical solution via Gauss-Newton type SQP methods, we can
take advantage of recently developed methods for the efficient
solution of large strictly convex quadratic programs with box
constraints \cite{HungerlaenderRendl15,HungerlaenderRendl17}.

The remainder of this paper is organized as follows. In the following
section we provide a result on convergence of the regularizer to an
exact solution of \eqref{Axu}, \eqref{Cuy} for the particular setting
\eqref{minJRaaoM}. In Section \ref{sec:appex} we discuss three coefficient
identification problems in elliptic boundary value problems as application
examples of our setting \eqref{minJRaaoM}. In Section \ref{sec:GNSQP}
we show how discretization of  $x$ and $\Phi$ in our application examples
leads to finite
dimensional non-linear box constrained minimization problems that can
be solved by Gauss-Newton type SQP approaches. Section \ref{sec:QPbox}
is concerned with the description of methods for
the efficient solution of large strictly convex quadratic programs
with box constraints that are used as subroutines of Gauss-Newton type
SQP approaches. In Section \ref{sec:num} we conduct several numerical
experiments for the suggested application examples. Section
\ref{sec:con} concludes the paper.

\section{Convergence}
For the sake of self-containedness we provide a result on convergence of the regularizer to an exact solution $(\xdag,\Phidag)$ of \eqref{Axu}, \eqref{Cuy} for the particular setting \eqref{minJRaaoM} of \eqref{minJR}, along with its (very short) proof.
The main assumption is existence of an appropriate topology $\tau_U$ on the state space, which will be constructed appropriately in the application examples below, and which also determines the topology in which the regularized solutions converge.
As parameter and data spaces, in view of the pointwise bounds on $x$ and $C(\Phi)$, we will always use
\[
X=L^\infty(\Omega)\,, \quad Y=L^\infty(\tilde{\Omega})\,,
\]
respectively, and $\calT_X$ will be defined by the weak * topology on $L^\infty(\Omega)$.

\begin{assumption}\label{ass1}\textcolor{white}{.}
\begin{itemize}
\item $A(\xdag,\Phidag)=0$, $C(\Phidag)=y$, $\|y-\ydel\|\leq\delta$, $\tau>1$, $\ul{x}\leq \xdag\le\ol{x}$.
\item The topology $\calT_U$ on $U$ is chosen such that for any $c>0$, the level sets
\[
\begin{aligned}
L_c=\{&(x,\Phi)\in L^\infty(\Omega)\times U \, : \, \calJ_A(x,\Phi)+\calR(x,\Phi)\leq c
\\&\mbox{ and } \ul{x}\leq x \leq \ol{x} \mbox{ and }
\ydel-\tau\delta\leq C(\Phi)\leq \ydel+\tau\delta \}
\end{aligned}\]
are $\calT$ compact. Here $\calT$ is the topology defined by weak-* $L^\infty$ convergence on $X=L^\infty(\Omega)$ and by $\calT_U$ on $U$.
\item $\calR:L^\infty(\Omega)\times U\to[0,\infty]$ is a proper convex $\calT$ lower semicontinuous functional.
\item $\calJ_A: L^\infty(\Omega)\times U\to\R$ is $\calT$ lower semicontinuous.
\item For any sequence $(x_n,\Phi_n)_{n\in\N}\subseteq L^\infty(\Omega)\times U$ and any element $(\bar{x}, \bar{\Phi})\in L^\infty(\Omega)\times U$ the following implication holds:
\begin{equation}\label{xbarubar}
\Bigl((x_n,\Phi_n)\stackrel{\calT}{\rightharpoonup}(\bar{x},\bar{\Phi})\, \wedge \,
\|C(\Phi_n)-y\|_{L^\infty(\tilde{\Omega})}\to0 \, \wedge \,
A(\bar{x}, \bar{\Phi})=0\Bigr)
\ \Rightarrow \ C(\bar{\Phi})=y.
\end{equation}
\end{itemize}
\end{assumption}

\begin{proposition}\label{prop_conv}
Under Assumption \ref{ass1}, for any $\alpha>0$, $\ydel\in Y$, a solution $(x_\alpha^\delta,\Phi_\alpha^\delta)$ to \eqref{minJRaaoM} exists.\\
As $\delta\to0$, $\alpha=\alpha(\delta)\to0$, the family of regularized solutions $(x_{\alpha(\delta)}^\delta,\Phi_{\alpha(\delta)}^\delta)_{\delta>0}$ has a $\calT$ convergent subsequence and the limit of every $\calT$ convergent subsequence solves \eqref{Axu}, \eqref{Cuy}. If the solution $(\xdag,\Phidag)$ to \eqref{Axu}, \eqref{Cuy} is unique, then $(x_\alpha^\delta,\Phi_\alpha^\delta)$ converges to $(\xdag,\Phidag)$ in $\calT$.
\end{proposition}
\begin{proof}
For fixed $\alpha>0$, minimality and admissibility of $(\xdag,\Phidag)$ for \eqref{minJRaaoM}, together with the fact that $A(\xdag,\Phidag)=0$, yields the estimate
\[
\calJ_A(x_\alpha^\delta,\Phi_\alpha^\delta)+\alpha\cdot\calR(x_\alpha^\delta,\Phi_\alpha^\delta)=\Tikh_\alpha(x,\Phi;\ydel)\leq \Tikh_\alpha(\xdag,\Phidag;\ydel)=\alpha\cdot\calR(\xdag,\Phidag)\,,
\]
hence
\begin{equation}\label{estQR}
\calJ_A(x_\alpha^\delta,\Phi_\alpha^\delta)\leq \alpha\cdot\calR(\xdag,\Phidag) \,,\qquad\calR(x_\alpha^\delta,\Phi_\alpha^\delta)\leq\calR(\xdag,\Phidag)\,,
\end{equation}
so that it suffices to restrict the search for a minimizer to the level set $L_c$ with $c=(1+\alpha)\calR(\xdag,\Phidag)$. Thus, w.l.o.g., a minimizing sequence $(x_n,\Phi_n)_{n\in\N}$ such that $\lim_{n\to\infty}\Tikh_\alpha(x_n,\Phi_n)=\inf\{\Tikh_\alpha(x,\Phi;\ydel)\, : \, \ul{x}\leq x \leq \ol{x} \mbox{ and } \ydel-\tau\delta\leq C(\Phi)\leq \ydel+\tau\delta \}$ is contained in $L_c$ and thus has a $\calT$ convergent subsequence with limit $(\bar{x},\bar{\Phi})\in L_c$. Thus, $(\bar{x},\bar{\Phi})$ satisfies  the box constraints and $\calT$ lower semicontinuity of  $\Tikh_\alpha$ yields minimality of $(\bar{x},\bar{\Phi})$.

To prove convergence as $\delta\to0$, $\alpha=\alpha(\delta)\to0$, we again invoke the minimality estimate
which yields \eqref{estQR}. As a consequence, since w.l.o.g. $\alpha\leq1$, $\calT$ compactness of $L_{2\calR(\xdag,\Phidag)}$ implies that
$(x_{\alpha(\delta)}^\delta,\Phi_{\alpha(\delta)}^\delta)_{\delta>0}$ has a $\calT$ convergent subsequence.
For the limit $(\bar{x},\bar{\Phi})$ of any $\calT$ convergent subsequence, by $\alpha(\delta)\to0$, the first estimate in \eqref{estQR}, and $\calT$ lower semicontinuity of $\calJ_A$, we get $\calJ_A(\bar{x},\bar{\Phi})=0$, hence by \eqref{Qdefinite}, $(\bar{x},\bar{\Phi})$ satisfies \eqref{Axu}. Moreover, due to the box constraints on $C(\Phi)$, the limit  $\delta\to0$, and \eqref{xbarubar}, this limit $(\bar{x},\bar{\Phi})$ also  satisfies \eqref{Cuy}. Convergence of the whole family in case of uniqueness follows from a subsequence-subsequence argument.
\end{proof}

\begin{remark}\label{rem:R}
In view of estimate \eqref{estQR}, which in case $\calR=0$ implies that $(x_\alpha^\delta,\Phi_\alpha^\delta)$ satisfies the model equation \eqref{Axu} exactly, the presence of a strictly positive regularization term can be viewed as a relaxation of the model equation.
Also note that the regularization term is here not necessarily needed for stability of $x$, since this is already achieved by the pointwise bounds. Still, the presence of the term $\alpha\cdot\calR$ may help to enable existence of minimizers of \eqref{minJRaaoM} and in this sense, well-posedness of the regularized problem.
\end{remark}

\section{Application examples}\label{sec:appex}
In this section we provide some examples of coefficient identification problems in elliptic boundary value problems.

\subsection{An inverse source problem}\label{sec_bprob}
Consider identification of a spatially varying source term $f$ (e.g., a heat source) in the elliptic boundary value problems
\begin{equation}\label{PDEs_bprob}
\begin{array}{rcl}
-\Delta \phi_i &=f&\mbox{ in }\Omega\ ,\\
\frac{\partial \phi_i}{\partial \nu}&=j_i&\mbox{ on }\Gamma\subseteq\partial\Omega\ ,\\
\phi_i&=0&\mbox{ on }\partial\Omega\setminus\Gamma\,,
\end{array} \quad i = 1,\ldots,I\ ,
\end{equation}
from observations $y_i=C_i(\phi_i)$ of $\phi_i$, $i = 1,\ldots,I$ (e.g., temperatures) in the interior and/or on the boundary of the domain $\Omega$.
This is a linear inverse problem which can be formulated as a linear operator equation or as a quadratic minimization problem.

We use the function spaces
\begin{equation}\label{V}
V=H^1_\diamondsuit(\Omega)=\begin{cases}
\setof{v\in H^1(\Omega)}{\mathrm{tr}_{\partial\Omega\setminus\Gamma}v=0}\ , \mbox{ if } \mbox{meas}(\partial\Omega\setminus\Gamma)>0\ ,
\\
\setof{v\in H^1(\Omega)}{\int_{\partial\Omega}\mathrm{tr}_{\partial\Omega}v\, ds =0}\ , \mbox{ if } \mbox{meas}(\partial\Omega\setminus\Gamma)=0\,,
\end{cases}
\end{equation}
(where in the latter case we assume that $\int_\Omega f \, dx=0=\int_{\partial\Omega} j_i\, ds$)
and the negative Laplace operator
\[
D:V\to V^*\,, \quad \langle D v,w\rangle_{V^*,V}= \int_\Omega \nabla v\cdot \nabla w \, dx\,,
\]
which is an isomorphism between $V$ and its dual $V^*$, thus we can use
\[
\|v\|_V:= \sqrt{\langle D v,v\rangle_{V^*,V}}= \sqrt{\int_\Omega|\nabla v|^2\, dx}\ ,
\]
as a norm on $V$.
With this function space setting and the functionals
$b_f,g_i\in V^*$ defined by
\[
\langle b_f, v\rangle_{V^*,V} =\int_\Omega f v\, dx\,, \qquad
\langle g_i, v\rangle_{V^*,V} =\int_{\Gamma} j_i \mathrm{tr}_\Gamma v\, ds\,,
\]
the weak form of \eqref{PDEs_bprob} reads as
\[
D \phi_i =b_f+g_i \mbox{ in } V^*\,, \quad i = 1,\ldots,I\,.
\]
Here $b_f$ plays the role of the parameter $x$ in the previous section, and the state consists of $\Phi:=(\phi_1,\ldots,\phi_I)$.

We will particularly concentrate on the practically relevant observations
\begin{equation}\label{C}
C_iv=\mathrm{tr}_{\partial\Omega}v\mbox{ or } C_iv=v\vert_{\omega_o}\ ,
\end{equation}
where in the latter case $\omega_o\subseteq\Omega$ is a measurable subset with positive measure, and thus use the data space
\begin{equation}\label{Y}
Y=L^\infty(\tilde{\Omega})^I\mbox{ for }\tilde{\Omega}=\partial\Omega\mbox{ or }\tilde{\Omega}=\omega_o, \mbox{ respectively.}
\end{equation}

From the point of view of elliptic PDEs, a natural norm for measuring the deviation from the model, i.e., the residual of the PDE, is the $H^{-1}$ norm, which can be implemented by using the inverse of the negative Dirichlet Laplacian. This results in the cost function
\[
\hspace*{-0.3cm}\begin{aligned}
&\calJ_A(b_f,\Phi)=\tfrac12 \sum_{i=1}^{I} \|D \phi_i-b_f-g_i\|_{V^*}^2
=\tfrac12 \sum_{i=1}^{I} \|\phi_i-D^{-1}(b_f+g_i)\|_{V}^2\\
&=\tfrac12 \sum_{i=1}^{I} \langle D\phi_i-(b_f+g_i),\phi_i-D^{-1}(b+g_i)\rangle_{V^*,V}\\
&=\tfrac12 \sum_{i=1}^{I} \Bigl(\int_\Omega|\nabla\phi_i|^2\, dx
-2\langle b_f,\phi_i\rangle_{V^*,V}
-2 \int_{\Gamma} j_i \mathrm{tr}_\Gamma \phi_i\, ds
+\langle b_f+g_i,D^{-1}(b_f+g_i)\rangle_{V^*,V},
\end{aligned}
\]
and leads to a formulation of the
inverse problem as a constrained minimization problem
\begin{equation}\label{bprob_min}
\min_{b_f,\Phi} \calJ_A(b_f,\Phi)\quad
\mbox{s.t. } C\phi_i=y_i\,, \ i=1,\ldots,I\,.
\end{equation}

This directly corresponds to the reformulation \eqref{minJ_Morozov} of the all-at-once version \eqref{Axu},\eqref{Cuy} with $\calJ_A(b_f,\Phi)=\frac{1}{2}\|A(b_f,\Phi)\|_W^2$,
$U=V^I$, $W=(V^*)^I$
\begin{eqnarray}
&A:L^\infty(\Omega)\times V^I\to (V^*)^I\,, \quad A(b_f,\Phi)=(D\phi_i-b_f-g_i)_{i=1}^I\ , \nonumber\\
&C:V^I\to Y=(Y_1,\ldots,Y_I)\,, \quad C\Phi=(C_i\phi_i)_{i=1}\,, \quad C_i:V\to Y_i\,,
\label{Ci}
\end{eqnarray}
while a reduced one \eqref{Fxy} can be defined via the linear forward operator
\[
F:X=L^\infty(\Omega)\to Y\,, \quad Fb_f=(C_i(D^{-1}(b_f+g_i)))_{i=1}^I\,.
\]

A priori information on pointwise lower and upper bounds $\underline{b}, \overline{b}$ of the source term, as often available in practice, and a relaxation of the observation equation according to the discrepancy principle leads to the regularized problem
\begin{equation}\label{bprob_min_reg}
\min_{b_f,\Phi} \calJ_A(b_f,\Phi)\quad
\mbox{s.t. }\ul{b}\leq b_f\leq \ol{b} \mbox{ a.e. in }\Omega\,, \quad
\|C_i\phi_i-y_i^\delta\|_{Y_i}\leq\tau\delta\,, \ i=1,\ldots,I\,.
\end{equation}
cf. \eqref{minJRaaoM}. Note that we set $\calR$ to zero here, which is feasible in the setting of Proposition \ref{prop_conv}, as long as Assumption \ref{ass1} can be verified.

To do so, we define the topology $\calT_U$ on $U=V^I$ by
\begin{equation}\label{topo}
\Phi_n\stackrel{\calT_U}{\to} \Phi\ \Leftrightarrow \
\begin{cases}
\Phi_n \to \Phi \mbox{ in }L^2(\Omega)^{I}\ ,\\
\Phi_n \rightharpoonup \Phi \mbox{ in }H^1(\Omega)^{I}\,, \\
C\Phi_n\to C\Phi \mbox{ in }L^\infty(\tilde{\Omega})^{I}\ .
\end{cases}
\end{equation}
Therewith, $\calT$ compactness of level sets $L_c$ obviously holds.
Indeed, boundedness of $\calJ_A(b_n,\Phi_n)$ and $L^\infty(\Omega)$ (hence $V^*$) boundedness of $b_n$ implies boundedness of $\Phi_n$ in $H^1(\Omega)^{I}$, which implies existence of subsequences $(b_{n_k},\Phi_{n_k})$ converging weakly * in $L^\infty(\Omega)$, as well as according to the first two limits in \eqref{topo}, to some $(\bar{b}_f,\bar{\Phi})$. By boundedness of $C\Phi_{n_k}$ in $L^\infty(\tilde{\Omega})^{I}$, we can extract another subsequence (without relabelling) such that $C\Phi_{n_k}\stackrel{*}{\rightharpoonup}\bar{y}$ in $L^\infty(\tilde{\Omega})$.
It remains to show that $\bar{y}=C\bar{\Phi}$.
In both cases of \eqref{C}, the operator $C$ is continuous as a mapping from $V$ to $L^p(\tilde{\Omega})$ for some $p\in(1,\infty)$ (in the first case, due to the Trace Theorem, in the second case, due to continuity of the embedding $V\to L^p(\Omega)$ and of the restriction operator $L^p(\Omega)\to L^p(\omega_o)$), hence, as a linear operator it is also weakly continuous. Thus, $C\Phi_{n_k}\rightharpoonup C(\bar{\Phi})$ in $L^p(\tilde{\Omega})$, which by uniqueness of weak limits implies
$\bar{y}=C\bar{\Phi}$.
By weak lower semicontinuity of the norms, we have $(\bar{b}_f,\bar{\Phi})\in L_c$.

Also $\calT$ lower semicontinutiy of $\calJ_A$ is a direct consequence of weak lower semicontinuity of the $V^*$ norm and the fact that $\calT$ convergence of $(b_n,\Phi_n)$ to $(\bar{b}_f,\bar{\Phi})$ implies weak convergence of $D \phi_{n,i}-b_n-g_i$ to $D \bar{\phi}_i-\bar{b}_f-g_i$  in $V^*$.

Moroever the last line in \eqref{topo} immediately implies \eqref{xbarubar}.

We mention in passing that Assumption \ref{ass1} remains valid if we add a regularization term $\alpha\cdot\calR$ with $\calR$ defined, e.g, by some positive power of a norm.

\subsection{Identification of a potential}\label{sec_cprob}
We now consider a nonlinear inverse problem for an elliptic PDE, namely recovery of the distributed coefficient $x=c$ in the boundary value problem
\begin{equation}\label{PDEs_cprob}
\begin{array}{rcl}
-\Delta \phi_i+ c \phi_i&=f_i&\mbox{ in }\Omega\ ,\\
\frac{\partial \phi_i}{\partial \nu}&=j_i&\mbox{ on }\Gamma\subseteq\partial\Omega\ ,\\
\phi_i&=0&\mbox{ on }\partial\Omega\setminus\Gamma\,,
\end{array} \quad i = 1,\ldots,I\ ,
\end{equation}
from interior or boundary observations $y_i=C_i(\phi_i)$ of the state $\Phi=(\phi_1,\ldots,\phi_I)$, where $\Omega\subseteq\R^d$, $d\in\{2,3\}$ is a Lipschitz domain and the excitation is done via the sources $f_i$ and the Neumann data $j_i$, which are assumed to be known.

In case $c$ is nonnegative almost everywhere in $\Omega$, the above PDE is elliptic and thus the forward problem of computing $\phi_i$ in \eqref{PDEs_cprob} is well-posed. The situation is to some extent similar (but technically more challenging, cf., e.g., \cite{Faucheretal}) when replacing the PDE in \eqref{PDEs_cprob} by the Helmholtz equation
\begin{equation}\label{Helmholtz}
-\Delta\phi_i - \frac{\omega^2}{c_0^2} \phi_i=f_i\quad \mbox{ in }\Omega\ ,
\end{equation}
as long as one can guarantee that $\omega^2$ stays away from the eigenfrequencies of $-c_0^2\Delta$.
This model, together with boundary observations of $\phi_i$ is the frequency domain version of the seismic inverse problem of recovering the spatially varying wave speed $c_0$ in the subsurface from surface measurements of the acoustic pressure $\phi_i$, often referred to as full waveform inversion (FWI) cf., e.g., \cite{Faucheretal} and the references therein.

The approach we are following here does not require well-definedness of the parameter-to-state map $S$ and therefore allows to consider \eqref{PDEs_cprob} with arbitrary $c\in L^\infty(\Omega)$, in particularly also \eqref{Helmholtz} without restriction on the frequency $\omega$. We will demonstrate this by means of some numerical experiments with negative and mixed sign coefficients $c$ in \eqref{PDEs_cprob} in Section \ref{sec:num}.

The weak form of \eqref{PDEs_cprob} can be written as
\[
D_c \phi_i =g_i \mbox{ in } V^*\,, \quad i = 1,\ldots,I\ ,
\]
where
\[
D_c:V\to V^*\,, \quad \langle D_c v,w\rangle_{V^*,V}= \int_\Omega
(\nabla v\cdot \nabla w +c v w)\, dx \ ,
\]
and $g_i\in V^*$ is defined by
\[
\langle g_i, v\rangle =\int_\Omega f_i v\, dx + \int_{\Gamma} j_i
\mathrm{tr}_\Gamma v\, ds \ ,
\]
with $V=H^1_\diamondsuit(\Omega)$ according to \eqref{V}, where in the pure Neumann case $\mbox{meas}(\partial\Omega\setminus\Gamma)=0$ we assume that $\int_\Omega f_i \, dx+\int_{\partial\Omega} j_i\, ds =0$.

Analogously to the inverse source problem above, we can formulate a cost functional by using the $V^*$ norm of the residual in the state equation.
\begin{equation}\label{Jc}
\begin{aligned}
 \calJ_A(c,\Phi) =\tfrac12 \sum_{i=1}^{I} \|D_c
   \phi_i-g_i\|_{V^*}^2& \\
=\tfrac12 \sum_{i=1}^{I} \Bigl(\int_\Omega(|\nabla\phi_i|^2 + 2c\phi_i^2)\, dx
-& 2 \langle g_i,\phi_i\rangle_{V^*,V} +\langle
c\phi_i-g_i,D^{-1}(c\phi_i-g_i)\rangle_{V^*,V}\,.
\end{aligned}
\end{equation}

Starting from the mimimization based formulation \eqref{bprob_min} (with $b_f$ replaced by $c$ and $\calJ_A$ according to \eqref{Jc}) of the inverse problem, again we can define the regularized problem without using $\calR$
\begin{equation}\label{cprob_min_reg}
\min_{c,\Phi} \calJ_A(c,\Phi)\quad
\mbox{s.t. }0\leq\ul{c}\leq c\leq \ol{c}\ \mbox{ a.e. in }\Omega\,, \
\ \
\|C_i\phi_i-y_i^\delta\|_{Y_i}\leq\tau\delta\,, \ i=1,\ldots,I\,,
\end{equation}
where this time we also guarantee nonnegativity of $c$ by means of the lower bound.

The topology to verify Assumption \ref{ass1} is the same as in the previous example \eqref{topo} and the verfication of $\calT$ compactness of level sets and of \eqref{xbarubar} goes exactly like in Section \ref{sec_bprob}.
To see $\calT$ lower semicontinuity of $\calJ_A$, consider a sequence $(c_n,\Phi_n)\stackrel{\calT}{\to} (\bar{c},\bar{\Phi})$, which implies $V^*$ convergence of
$D_{c_n}\phi_{n,i}-g_i$ to $D_{\bar{c}}\bar{\phi}_i-g_i$ as follows. For any $\psi\in V$, we have
\[
\begin{aligned}
&\left|\langle D_{c_n}\phi_{n,i}-D_{\bar{c}}\bar{\phi}_i, \psi\rangle_{V^*,V}\right|
\leq\underbrace{\left|\int_\Omega \nabla(\phi_{n,i}-\bar{\phi}_i)\cdot\nabla\psi\, dx\right|}_{\to0}\\
&+\underbrace{\left|\int_\Omega c_n(\phi_{n,i}-\bar{\phi}_i)\psi\, dx\right|}_{
\leq \|c_n\|_{L^\infty(\Omega)}\|\phi_{n,i}-\bar{\phi}_i\|_{L^2(\Omega)}\|\psi\|_{L^2(\Omega)}}
+\left|\int_\Omega (c_n-\bar{c})\underbrace{\bar{\phi}_i\psi}_{\in L^1(\Omega)}\Bigr)\, dx\right|
\to0 \mbox{ as }n\to\infty\,.
\end{aligned}
\]
By weak lower semicontinuity of the $V^*$ norm, this implies $\calT$ lower semicontinuity of $\calJ_A$.

\begin{remark}\label{rem:altJ}
In the elliptic case $c\geq0$ in \eqref{PDEs_cprob},
alternatively to $D$, the operator $D_c$can be used as an isomorphism between $V$ and $V^*$, which induces the norms
\[
\begin{aligned}
&\|v\|_{H^1_c(\Omega)}^2:=\langle D_c v, v\rangle_{V^*,V}
=\int_\Omega (|\nabla v|^2 +cv^2)\, dx \,, \quad v\in V=H^1_\diamondsuit(\Omega)\\
&\|v^*\|_{H^{-1}_c(\Omega)}^2:=\langle v^*, D_c^{-1}v^*\rangle_{V^*,V}
 \,, \quad v^*\in V^*\ .
\end{aligned}
\]
Thus for $c\geq0$, we could define the cost function by
\begin{equation}\label{tilJc}
\begin{aligned}
&\tilde{\calJ}_A(c,\Phi)=\tfrac12 \sum_{i=1}^{I} \|D_c \phi_i-g_i\|_{H^{-1}_c(\Omega)}^2
=\tfrac12 \sum_{i=1}^{I} \|\phi_i-D_c^{-1}g_i\|_{H^1_c(\Omega)}^2\\
&=\tfrac12 \sum_{i=1}^{I} \Bigl(\int_\Omega(|\nabla\phi_i|^2 +c\phi_i^2)\, dx
-2 \langle g_i,\phi_i\rangle_{V^*,V} +\langle g_i,D_c^{-1}g_i\rangle_{V^*,V}\,.
\end{aligned}
\end{equation}
Due to  the use of a parameter dependent norm, this is really different from the standard reformulation \eqref{minJ_Morozov} of the all-at-once version \eqref{Axu}, \eqref{Cuy} with $\calJ_A(c,\Phi)=\frac{1}{2}\|A(c,\Phi)\|_W^2$, $U=V^I$, $W=(V^*)^I$, and
\[
\begin{aligned}
&A:L^\infty(\Omega)\times V^I\to (V^*)^I\,, \quad A(c,\Phi)=(D_c\phi_i-g_i)_{i=1}^I\ ,
\end{aligned}
\]
and \eqref{Ci}, as actually given by \eqref{Jc}.

However, the last term $\langle g_i,D_c^{-1}g_i\rangle_{V^*,V}$ in \eqref{tilJc} would spoil $\calT$ semicontinuity of $\calJ_A$ even if we add some higher norm of $\Phi$ as a regularization term and strengthen $\calT$ accordingly (see \eqref{topo_a} below for a different application example). Thus to establish existence of a minimizer of the regularized problem with $\tilde{\calJ}_A$, we would have to additionally regularize $c$ e.g., by a total variation term, which would require more sophisticated discretization than piecewise constant finite elements, though.
Moroever, due to this last term, evaluation of the cost function \eqref{tilJc} for some iterate $c_n$ requires solving boundary value problems with the elliptic operator $D_{c_n}=-\Delta+c_n\cdot$ changing in each iteration, while \eqref{Jc} only involves simple Laplace problems. Thus we furtheron stay with the formulation based on \eqref{Jc}.

\medskip

We mention in passing that a reduced formulation \eqref{Fxy} can be defined by means of the nonlinear forward operator
$F:X=L^\infty(\Omega)\to Y\,, \quad F(c)=(C_i(D_c^{-1}(g_i)))_{i=1}^I$,
$\calD=\setof{c\in L^\infty(\Omega)}{0\leq c\leq \ol{c}\ \mbox{ a.e. in }\Omega}$.
\end{remark}

\subsection{Identification of a diffusion coefficient}\label{sec_aprob}
Another nonlinear problem is the identification of the distributed diffusion coefficient $x=a$ in the elliptic boundary value problem
\begin{equation}\label{PDEs_aprob}
\begin{array}{rcl}
-\nabla\cdot(a\nabla\phi_i)&=f_i&\mbox{ in }\Omega\ ,\\
a\frac{\partial \phi_i}{\partial \nu}&=j_i&\mbox{ on }\Gamma\subseteq\partial\Omega\,,\\
\phi_i&=0&\mbox{ on }\partial\Omega\setminus\Gamma\,,
\end{array} \quad i = 1,\ldots,I\ ,
\end{equation}
from observations $y_i=C_i(\phi_i)$ of the state $\Phi=(\phi_1,\ldots,\phi_I)$.
Again, $\Omega\subseteq\R^d$, $d\in\{2,3\}$ is a Lipschitz domain, the excitations $f_i$, $j_i$ are assumed to be known, and we focus on the observation setting \eqref{C}, \eqref{Y}.

With
\[
D_a:V\to V^*\,, \quad \langle D_a v,w\rangle_{V^*,V}= \int_\Omega
a\nabla v\cdot \nabla w\, dx\ ,
\]
and the function space $V=H^1_\diamondsuit(\Omega)$ according to \eqref{V} (where again in case $\mbox{meas}(\partial\Omega\setminus\Gamma)=0$ we assume that $\int_\Omega f_i \, dx+\int_{\partial\Omega} j_i\, ds =0$), as well as $g_i\in V^*$ defined by
\[
\langle g_i, v\rangle =\int_\Omega f_i v\, dx + \int_{\Gamma} j_i \mathrm{tr}_\Gamma v\, ds\ ,
\]
the weak form of \eqref{PDEs_aprob} can be written as
\[
D_a \phi_i =g_i \mbox{ in } V^*\,, \quad i = 1,\ldots,I\ .
\]

Note that with $f_i=0$, $\Gamma=\partial\Omega$, $C_i=\mathrm{tr}_{\partial\Omega}$,  this setting comprises the electrical impedance tomography (EIT) problem  of recovering the conductivity $a$ from several current-voltage measurements $(j_i,\mathrm{tr}_{\partial\Omega}\phi_i)$, $i = 1,\ldots,I$, cf., e.g., \cite{BorceaEIT} and the references therein.

To derive a minimization based formulation of this inverse problem we define
\[
\begin{aligned}
&\calJ_A(a,\Phi)=
\tfrac12 \sum_{i=1}^{I} \|D_a\phi_i-g_i\|_{V^*}^2
=\tfrac12 \sum_{i=1}^{I} \langle
D_a\phi_i-g_i),D^{-1}(D_a\phi_i-g_i)\rangle_{V^*,V}\ ,\end{aligned}
\]
in
\begin{equation}\label{aprob_min}
\min_{a,\Phi} \calJ_A(a,\Phi)\quad \mbox{s.t. }C\phi_i=y_i\,, \ i=1,\ldots,I\,,
\end{equation}
which corresponds to the reformulation \eqref{minJ_Morozov} of the all-at-once version \eqref{Axu},\eqref{Cuy} with $\calJ_A(a,\Phi)=\frac{1}{2}\|A(a,\Phi)\|_W^2$,
$U=V^I$, $W=(V^*)^I$,
\[
\begin{aligned}
&A:L^\infty(\Omega)\times V^I\to (V^*)^I\,, \quad A(a,\Phi_1)=(D_a\phi_i-g_i)_{i=1}^I\,,
\end{aligned}
\]
and \eqref{Ci}.

As a regularized version of \eqref{aprob_min} we consider
\begin{equation}\label{aprob_min_reg}
\left\{\hspace*{-0.1cm}
\begin{aligned}
&\min_{a,\Phi} \calJ_A(a,\Phi) + \tfrac{\alpha}{2} \|\phi_i\|_{H^{3/2-\epsilon}(\Omega)}^2\\
&\hspace*{0.15cm}\mbox{s.t. }0<\ul{a}\leq a\leq \ol{a}\ \mbox{ a.e. in
}\Omega\,, \mbox{ and }
\|C_i\phi_i-y_i^\delta\|_{Y_i}\leq\tau\delta\,, \ i=1,\ldots,I\,.
\end{aligned}
\right.
\end{equation}
The regularization term $\|\phi_i\|_{H^{3/2-\epsilon}(\Omega)}^2$ with $\epsilon\in (0,\frac12)$ is required to guarantee existence of a minimizer of \eqref{aprob_min_reg}, while still admitting jumps in the gradient of $\phi_i$, hence jumps in the diffusivity/conductivity $a$, as often occuring in practice.

The $\calT_U$ topolgy is here defined by
\begin{equation}\label{topo_a}
\Phi_n\stackrel{\calT_U}{\to} \Phi\ \Leftrightarrow \
\begin{cases}
\Phi_n \to \Phi \mbox{ in }H^1(\Omega)^{I}\ ,\\
\Phi_n \rightharpoonup \Phi \mbox{ in }H^{3/2-\epsilon}(\Omega)^{I}\,, \\
C\Phi_n\to C\Phi \mbox{ in }L^\infty(\tilde{\Omega})^{I}\,.
\end{cases}
\end{equation}

Therewith, Assumption \ref{ass1} can be established as follows.
Verfication of $\calT$ compactness of level sets and of \eqref{xbarubar} is the same as in Section \ref{sec_bprob}, additionally taking into account weak lower semicontinuity of the $H^{3/2-\epsilon}(\Omega)$ norm.
Moreover, $\calT$ convergence of $(a_n,\Phi_n)$ to $(\bar{a},\bar{\Phi})$ implies weak $V^*$ convergence of $D_{a_n}\phi_{n,i}-g_i$ to $D_{\bar{a}}\bar{\phi}_i-g_i$ as follows. For any $\psi\in V$
\[
\begin{aligned}
&\left|\langle D_{a_n}\phi_{n,i}-D_{\bar{a}}\bar{\phi}_i, \psi\rangle_{V^*,V}\right|\\
&\leq\underbrace{\left|\int_\Omega a_n\nabla(\phi_{n,i}-\bar{\phi}_i)\cdot\nabla\psi\, dx\right|}_{\leq \|a_n\|_{L^\infty(\Omega)}\|\phi_{n,i}-\bar{\phi}_i\|_{H^1(\Omega)}\|\psi\|_{H^1(\Omega)}}
+\left|\int_\Omega (a_n-\bar{a})\underbrace{\nabla\bar{\phi}_i\cdot\nabla\psi}_{\in L^1(\Omega)}\Bigr)\, dx\right| \to0 \mbox{ as }n\to\infty\,.
\end{aligned}
\]
This implies $\calT$ lower semicontintuity of $\calJ_A(a,u)=\frac{1}{2}\|A(a,u)\|_{(V^*)^I}^2$.

\begin{remark}
Analogously to Remark \ref{rem:altJ} we could use the parameter dependent norm
\[
\|v\|_{H^1_a(\Omega)}:=\sqrt{\langle D_a v,v\rangle_{V^*,V}}=\sqrt{\int_\Omega a|\nabla v|^2\, dx} \,, \quad v\in V=H^1_\diamondsuit(\Omega)\ ,
\]
on $V$ for $a\in L^\infty(\Omega)$ positive and bounded away from zero, to define the cost function
\[
\begin{aligned}
&\tilde{\calJ}_A(a,\Phi)=
\tfrac12 \sum_{i=1}^{I} \|D_a\phi_i-g_i\|_{H^{-1}_a(\Omega)}^2=
\tfrac12 \sum_{i=1}^{I} \|\phi_i-D_a^{-1}g_i\|_{H^1_a(\Omega)}^2\\
&=\tfrac12\sum_{i=1}^{I}\Bigl(
\langle \phi_i,D_a^{-1}\phi_i\rangle_{V^*,V}
-2\langle g_i,\phi_i\rangle_{V^*,V}
+\langle g_i,D_a^{-1}g_i\rangle_{V^*,V}\Bigr)\,.
\end{aligned}
\]
However, again the last term would spoil $\calT$ semicontinuity of $\calJ_A$ and necessitate additional regularizion of $a$, which we wish to avoid.

\medskip

A reduced formulation \eqref{Fxy} could here be defined via
$F:\calD\to Y\,, \quad F(a)=(C_i(D_a^{-1}g_i))_{i=1}^I$,
$\calD=\setof{a\in L^\infty(\Omega)}{0<\ul{a}\leq a\leq \ol{a} \mbox{ a.e. in }\Omega}$.)
\end{remark}

\section{Gauss-Newton SQP method} \label{sec:GNSQP}

Discretization of $x$ and $u$ with piecewise constant and piecewise linear finite elements, respectively, leads to finite dimensional nonlinear box constrained minimization problems.
We solve these iteratively by Gauss-Newton type SQP methods, which means that we approximate the Hessian of the Lagrangian by linearizing under the norm that defines the cost function, i.e., skipping higher than first order derivatives of the operator $A$, and leave the constraints unchanged due to their linearity.

For the potential identification problem \eqref{cprob_min_reg} from Section \ref{sec_cprob}, this means that we successively have to solve a discretized version of the quadratic minimization problem
\[
\left\{\hspace*{-0.1cm}
\begin{aligned}
&\min_{c,\Phi} \calJ_k(c,\Phi)\\
&\hspace*{0.15cm}\mbox{s.t. }0\leq\ul{c}\leq c\leq \ol{c}\ \mbox{ a.e. in }\Omega\,, \mbox{ and }
\ydel_i-\tau\delta\leq C_i \phi_i\leq \ydel_i+\tau\delta\,, \ i=1,\ldots,I\,,
\end{aligned}
\right.
\]
where with $d_{k,i}=D^{-1}(c_k\phi_{k,i}+g_i)$, the cost function can be rewritten as
\[
\hspace*{-0.3cm}\begin{aligned}
&\calJ_k(c,\Phi)=\tfrac12 \sum_{i=1}^{I}
\|D \phi_i+c_k\phi_i+c\phi_{k,i}-(c_k\phi_{k,i}+g_i)\|_{V^*}^2=\\
& \sum_{i=1}^{I}
\Bigl(\int_\Omega (\tfrac12|\nabla\phi_i|^2+c_k\phi_i^2+c\phi_{k,i}\phi_i
-(c_k\phi_{k,i}+c_kd_{k,i}+f_i)\phi_i-c \phi_{k,i}d_{k,i})\, dx\\
&\quad-\int_{\partial\Omega}j_i\phi_i\, ds
+\tfrac12\langle c_k\phi_i+c\phi_{k,i}, D^{-1} (c_k\phi_i+c\phi_{k,i})\rangle_{V^*,V}
+\tfrac12\langle c_k\phi_{k,i}+g_i, d_{k,i} \rangle_{V^*,V}\Bigr).
\end{aligned}
\]

For the diffusion identification problem \eqref{aprob_min} from Section \ref{sec_aprob}, the quadratic minimization problem in each Newton step reads as
\[
\left\{\hspace*{-0.1cm}
\begin{aligned}
&\min_{a,\Phi} \calJ_k(a,\Phi)\\
&\hspace*{0.15cm}\mbox{s.t. }0<\ul{a}\leq a\leq \ol{a}\ \mbox{ a.e. in }\Omega\,,
\mbox{ and }
\ydel_i-\tau\delta\leq C_i \phi_i\leq \ydel_i+\tau\delta\,, \ i=1,\ldots,I\,,
\end{aligned}
\right.
\]
with
\[
\begin{aligned}
&\calJ_k(a,\Phi)=\tfrac12 \sum_{i=1}^{I}
\|D_{a_k}\phi_i+(D_a-D_{a_k})\phi_{k,i}-g_i\|_{V^*}^2=\\
&\sum_{i=1}^{I}
\Bigl(
\tfrac12\langle e_i,D^{-1}e_i\rangle_{V^*,V}
+\tfrac12\langle D_a\phi_{k,i}, D^{-1}D_a\phi_{k,i}\rangle_{V^*,V}
+\tfrac12\langle D_{a_k}\phi_{k,i}+g_i,d_{k,i}\rangle_{V^*,V}\\
&\quad+\langle D_a\phi_{k,i},D^{-1}e_i\rangle_{V^*,V}
-\langle e_i,d_{k,i}\rangle_{V^*,V}
-\langle D_a\phi_{k,i},d_{k,i}\rangle_{V^*,V}
\Bigr)\,,
\end{aligned}
\]
where $e_i=D_{a_k}\phi_i\,,$ $d_{k,i}=D^{-1}(D_{a_k}\phi_{k,i}+g_i)$.
Note that we always just invert the negative Laplacian $D$ and not the parameter dependent operator $D_c$ or $D_a$ as it would be required in a reduced formulation \eqref{Fxy}.

For the inverse source problem \eqref{bprob_min}, the cost functional is already quadratic, so just one Newton step is required and coinides with the original regularized minimization problem \eqref{bprob_min}.
Thus, after discretization, each Gauss Newton step consists of solving
a strictly convex box constrained quadratic program
\begin{equation}\label{pd}
\min J(x) \quad \mbox{ s.t. } \ell \leq x \leq u\,,
\end{equation}
with $J(x) = x^TQx+q^Tx$, $Q \succ 0$, $\ell < u$, and the inequalities to be understood
component-wise.

\section{Solving strictly convex box constrained quadratic programs}\label{sec:QPbox}

In this section
we discuss available methods for
efficiently solving \eqref{pd} and provide some details on the approach
performing best in our numerical tests in Section \ref{sec:num}.

\subsection{Available methods} The minimization of a strictly convex
quadratic function under box constraints
is a fundamental problem on its own and additionally an important
building block for solving more complicated optimization problems.
Interior point, active set and gradient projection methods are the
most prominent approaches for efficiently solving \eqref{pd}.

A special type of active-set method was introduced by Bergounioux et al.\
\cite{BHHK:00,BIK:99} in connection with constrained optimal control
problems and tailored to deal with discretisations of
specially structured elliptic partial differential equations.
Their approach has turned out to be a powerful, fast and
competitive approach for \eqref{pd}. Hinterm{\"u}ller et al.\
 \cite{HIK:03} provide a theoretical explanation
 of its efficiency by interpreting it as a semismooth Newton method.
One major drawback of this method lies in the fact that it is not globally
convergent for all classes of strictly convex quadratic objectives.
Practical computational evidence shows that if the method converges at
all, it typically takes very few iterations
to reach an optimal solution, see e.g.\ \cite{JP:94,kure03}.

Several modifications of this method
were introduced recently. In \cite{HungerlaenderRendl15} we propose a
primal feasible active set method that extends the approach from
\cite{BHHK:00,BIK:99} such that
strict convexity of the quadratic objective function is sufficient for
the algorithm to stop after a finite number of steps with an optimal
solution. In \cite{HungerlaenderRendl17} we introduce yet another
modified, globally convergent version
of this active set method, which allows primal infeasiblity of the iterates and
aims at maintaining the combinatorial flavour of the original approach.

Beside their simplicity (no tuning parameters), our globally
convergent methods offer the favorable features of standard active set
approaches like the ability to find the exact numerical solution and
the possibility to warm start.
Computational experience on a variety of difficult classes of test
problems shows that
our approaches mostly outperform other existing methods for
\eqref{pd}. In the following subsection we discuss the workings of the
approach from \cite{HungerlaenderRendl15} in some detail, as it is the
best performing one for the benchmark instances considered in this paper.

\subsection{Details on a primal feasible active set method \cite{HungerlaenderRendl15}}
  It is well known that $x \in \mathbb R^n$ together with vectors
  $\alpha, \gamma \in \mathbb R^n$
of Lagrange multipliers for the box constraints furnishes the unique
global minimium of \eqref{pd} if and only if the triple
($x,\alpha,\gamma$) satisfies the KKT system
\begin{subequations}\label{KKTia}
\begin{align}
 Qx + q + \alpha - \gamma &= 0\,,\label{LKKTM} \\
 \alpha \circ (u-x) &= 0\,,\quad \gamma \circ (x-\ell) = 0\,,\label{LKKTC2} \\
u-x &\geq 0\,,\quad x-\ell \geq 0\,,\label{primf}\\
  \alpha &\geq 0\,,\quad \gamma \geq 0\,,\label{dualf}
\end{align}
\end{subequations}
where $\circ$ denotes component-wise multiplication.

The crucial step in solving \eqref{pd} is to identify those
inequalities which are active, i.e.\ the active sets $\mathcal A \subseteq
\mathcal N := \{1,\ldots,n\}$ and $\mathcal C \subseteq
\mathcal N$, where the solution to \eqref{pd} satisfies $x_{\mathcal
  A} = u_{\mathcal A}$ and $x_{\mathcal C} = \ell_{\mathcal C}$
respectively. Here, for ${\mathcal B} \subseteq {\mathcal N}$ and some vector $x\in \R^n$, $x_{\mathcal B}$ denotes restriction of the vector to components with indices in ${\mathcal B}$; analogous notation will be used for matrices.
Next let us also introduce the inactive set
$\mathcal I := \mathcal N \setminus \left(\mathcal A \cup \mathcal C\right)$.
Then we additionally require $\alpha_{\mathcal I} =
\alpha_{\mathcal C} = \gamma_{\mathcal I} = \gamma_{\mathcal A}
  = 0$ for \eqref{LKKTC2} to hold.
To simplify notation we denote KKT($\mathcal A, \mathcal C$) as
the following set of equations:
\begin{equation*}
\text{KKT($\mathcal A, \mathcal C$):}\quad Qx+q+\alpha - \gamma =
0\,,\quad x_{\mathcal A} = u_{\mathcal A}\,,\quad x_{\mathcal C} = \ell_{\mathcal
  C}\,,\quad \alpha_{\mathcal I} = \alpha_{\mathcal C} =
\gamma_{\mathcal I} = \gamma_{\mathcal A} = 0\,.
\end{equation*}
The solution of KKT($\mathcal A, \mathcal C$) satisfies stationarity
\eqref{LKKTM}
and complementary \eqref{LKKTC2} conditions and is given by
\begin{align*}
& x_{\mathcal A} = u_{\mathcal A}\,,\quad x_{\mathcal C} = \ell_{\mathcal C}\,,\quad
  Q_{\mathcal I, \mathcal I} x_{\mathcal I} = - \left(q_{\mathcal I} + Q_{\mathcal I,
  \mathcal C} \ell_{\mathcal C} + Q_{\mathcal I, \mathcal A} u_{\mathcal A}\right)\,,\\
& \alpha_{\mathcal I} = 0\,,\quad \ \ \alpha_{\mathcal C} = 0\,,\quad \ \ \alpha_{\mathcal
  A} = q_{\mathcal A} + Q_{\mathcal A, \mathcal I} x_{\mathcal I} +
  Q_{\mathcal A, \mathcal C} \ell_{\mathcal C}\,,\\
& \gamma_{\mathcal I} = 0\,,\quad \ \ \gamma_{\mathcal A} = 0\,,\quad \ \ \gamma_{\mathcal C}
  = q_{\mathcal C} + Q_{\mathcal C, \mathcal I} x_{\mathcal I} +
  Q_{\mathcal C, \mathcal A} u_{\mathcal A}\,.
\end{align*}
We write $[x,\alpha,\gamma]$ = KKT($\mathcal A$,$\mathcal C$) to indicate
that $x$, $\alpha$ and $\gamma$
satisfy KKT($\mathcal A$,$\mathcal C$).
In some cases we also write $x$ = KKT($\mathcal A$,$\mathcal C$) to
emphasize that
we only use $x$ and therefore do not need the backsolve to get
($\alpha$,$\gamma$).
If we only carry out the backsolve to get ($\alpha$,$\gamma$), we write
$[\alpha,\gamma]$ = KKT($\mathcal A$,$\mathcal C$), and it is assumed
that the corresponding $x$ is available.

Let us call the pair $(\mathcal A,\mathcal C)$ with $\mathcal A \cup \mathcal
C  \subseteq \mathcal N$ and ${\mathcal A} \cap {\mathcal C} = \emptyset$ (due to the
definition of the bounds) a primal feasible pair, if
the solution $x$ to KKT$(\mathcal A,\mathcal C)$ is primal feasible,
i.e.\ \eqref{primf} holds.
Furthermore the pair $(\mathcal A,\mathcal C)$ is called optimal if
and only if $[x,\alpha,\gamma]$ =
KKT($\mathcal A$,$\mathcal C$) satisfies both primal \eqref{primf} and
dual \eqref{dualf} feasibility,
since $x$ is then the unique solution of \eqref{pd}.

Given a non-optimal primal-feasible pair $(\mathcal A,\mathcal C)$ with
$[x,\alpha,\gamma] = $KKT$(\mathcal A,\mathcal C)$, our goal is to find a new
primal-feasible pair $(\mathcal B,\mathcal D)$ with $[y, \beta,
\delta] = $KKT$(\mathcal B,\mathcal D)$ and $J(y) < J(x)$.
We start a new iteration by using the feasibility information of the
dual variables
($\alpha$,$\gamma$) to determine the new active sets
$\mathcal{B}_s := \{ i \in \mathcal{A}: \alpha_{i} \geq 0 \}$ and
$\mathcal{D}_s := \{ i \in \mathcal{C}: \gamma_{i}\geq 0 \}$.
However the pair $(\mathcal B_s,\mathcal D_s)$ does not need to be primal
feasible.
To turn it into a primal feasible pair $(\mathcal B,\mathcal D)$,
the variables connected to the primal
infeasibilities are added to the respective active set and new
primal variables $y$ are computed. This process
is iterated until a primal feasible solution is generated, see also
Table \ref{KR-method}. This iterative scheme clearly terminates
because $\mathcal B$ and
$\mathcal D$ are only augmented by adding elements of the corresponding
inactive set $\mathcal J := \mathcal N \setminus \left(\mathcal B \cup
  \mathcal D\right)$.

\begin{table}[ht]
\begin{center}
\begin{tabular}{l}
\hline
$\mathcal B \leftarrow \mathcal B_s,\ \mathcal D \leftarrow \mathcal D_s$.\\
while $(\mathcal B, \mathcal D)$ not primal feasible:\\
\ \  $y= \text{KKT}(\mathcal B,\mathcal D),\ \mathcal B \leftarrow
  \mathcal B \cup \{i \in \mathcal N \setminus \mathcal B : y_{i}
\geq u_{i} \},\ \mathcal D \leftarrow \mathcal D \cup \{i \in \mathcal
  N \setminus \mathcal D : y_{i}
 \leq \ell_{i} \}$. \\\hline
\end{tabular}
\end{center}
\caption{Generation of a primal feasible pair $(\mathcal B,\mathcal
  D)$.}\label{KR-method}
\end{table}

In general we have no means
to ensure convergence of the algorithmic setup described so far.
The key idea to ensure convergence of active set methods consists in
establishing that some merit function strictly
decreases during successive
iterates of the algorithm. This gurantees that no active set is
considered more than once and hence cycling cannot occur. In our case
we use the objective function as merit function.

To avoid cycling,
we therefore suggest
additional measures in case $J(y) \geq J(x)$. Let us assume $(\mathcal
A,\mathcal C)$ not optimal for the following case distinction.
Note that $|\mathcal A| + |\mathcal C| \geq 1$ holds due to
\cite[Lemma 10]{HungerlaenderRendl15}. We
consider the following cases separately.

\smallskip
\noindent {\bf Case 1:} $J(y) < J(x)$.

In this case we can set $\mathcal A \leftarrow \mathcal B,\
\mathcal C \leftarrow \mathcal D$,
and continue with the next iteration.

\smallskip
\noindent {\bf Case 2:} $J(y) \geq J(x)$ and $|\mathcal A| + |\mathcal C| =1$.

The pair $(\mathcal A,\mathcal C)$ with $\mathcal A \cup \mathcal C =\{ j\}$
is primal feasible.
Since $x$ = KKT($\mathcal A$,$\mathcal C$)
is primal feasible, but not optimal, we must have
$\alpha_{j} <0$ if $j \in \mathcal A$ and $\gamma_j < 0$ if $j \in \mathcal C$.
Thus the objective function improves by allowing $x_{j}$ to move away
from the respective boundary. Hence in an optimal solution it is
essential that
$x_{j} < u_{j}$ if $j \in \mathcal A$ and $x_j > \ell_j$ if $j \in \mathcal C$. Thus we have
a problem with only $2n-1$ constraints which we solve by induction
on the number of constraints.

\smallskip
\noindent {\bf Case 3:} $J(y) \geq J(x)$ and $|\mathcal A| + |\mathcal C| > 1$.

In this case we again solve a problem with less than $2n$ constraints
to optimality. Its optimal solution then yields a feasible pair $(\mathcal
B,\mathcal D)$ with $y =$KKT$(\mathcal B,\mathcal D)$ and
$J(y)< J(x)$.
Our strategy is to identify two sets $\mathcal A_{0} \subseteq \mathcal A$ and
$\mathcal C_{0} \subseteq \mathcal C$ with $\mathcal A_0 \cup \mathcal C_0 \not =
\emptyset$ such that $x$ is feasible
but not optimal for
\begin{equation}\label{pd1}
\min J(x) \quad \mbox{ s.t. } \ell \leq x \leq u\,,\ \ x_{\mathcal
  A_{0}} = u_{\mathcal A_{0}}\,,\ \ x_{\mathcal C_{0}} = \ell_{\mathcal C_{0}}\,.
\end{equation}
Note that \eqref{pd1} is a subproblem of \eqref{pd}, where some
variables are fixed to their upper respectively lower bounds. For
details on the
choice of $(\mathcal A_{0},\mathcal C_{0})$ we refer to \cite{HungerlaenderRendl15}.

Now (recursively) solving \eqref{pd1}
yields an optimal pair $(\mathcal B_0,\mathcal D_{0})$.
We set $\mathcal B:= \mathcal A_0 \cup \mathcal B_0,\ \mathcal D := \mathcal
C_{0} \cup \mathcal D_{0}$ and compute $[y,\beta,\delta] =$KKT$(\mathcal
B,\mathcal D)$. By construction,
$(\mathcal B,\mathcal D)$ is primal feasible, and $J(y) < J(x)$
because $y$ is optimal for \eqref{pd1}. In summary we can ensure that
the objective value associated to the primal feasible pair of active
sets reduces in each iteration and therefore the feasible active set method from
\cite{HungerlaenderRendl15} is globally convergent. A corresponding
algorithmic description of the method is given in Table \ref{modified-alg_ts}.

\begin{table}[ht]
\hrule
\vskip 2mm
\centerline{Feasible active set method for solving \eqref{pd}}
\vskip 2mm
\hrule
\vskip 2mm
\noindent
{\bf Input:}
$Q \succ 0$, $\ell,\ u,\ q \in \mathbb R^n,\ \ell < u$.
$\mathcal A,\ \mathcal C \subseteq \mathcal N,\ \mathcal A \cap \mathcal C
= \emptyset$, $(\mathcal A,\mathcal C)$ primal feasible. \\
{\bf Output:}
$(\mathcal A,\mathcal C)$ optimal for \eqref{pd}.
\noindent
\vskip 2mm
\hrule
\vskip 2mm
\noindent
 $[x,\alpha,\gamma] = $KKT$(\mathcal A,\mathcal C)$\\
\noindent
{\bf while } $(\mathcal A,\mathcal C)$ not optimal for \eqref{pd}\\
\spc $\mathcal B_{s} \leftarrow \{i \in \mathcal A: \alpha_{i} \geq 0 \};
\mathcal B \leftarrow
\mathcal B_{s}$.\\
\spc $\mathcal D_{s} \leftarrow \{i \in \mathcal C: \gamma_{i} \geq
0 \};
\mathcal D \leftarrow \mathcal D_{s}$.\\
\spc $y= $KKT$(\mathcal B,\mathcal D)$.\\
\spc {\bf while} $(\mathcal B,\mathcal D)$ not primal feasible\\
\spc \spc  $\mathcal B \leftarrow \mathcal B \cup \{ i \in \overline{\mathcal B}: y_{i}
\geq u_{i} \}$.\\
\spc \spc $\mathcal D \leftarrow \mathcal D \cup \{ i \in \overline{\mathcal D}: y_{i}
\leq \ell_{i} \}$. \\
\spc \spc $y= $KKT$(\mathcal B,\mathcal D)$.\\
\spc {\bf endwhile} \\
\spc {\bf Case 1:} $J(y) < J(x)$ \\
\spc \spc $\mathcal A \leftarrow \mathcal B,\ \mathcal C \leftarrow \mathcal D$.\\
\spc {\bf Case 2:} $J(y) \geq J(x)$  and $|\mathcal A|+|\mathcal C|=1$.\\
\spc \spc  Let $(\mathcal A_{opt},\mathcal C_{opt})$ be the optimal
pair for \eqref{pd} with the upper respectively lower \\
\spc \spc bound on $\mathcal A \cup \mathcal C =\{ j \}$ removed.
$(\mathcal A_{opt},\mathcal C_{opt})$ is optimal for \eqref{pd}, hence stop. \\
\spc {\bf Case 3:} $J(y) \geq J(x)$ and $|\mathcal A| + |\mathcal C| > 1$ \\
\spc \spc  Choose $\mathcal A_0 \subseteq \mathcal A$, $\mathcal C_0 \subseteq
\mathcal C$ with $\mathcal A_0 \cup \mathcal C_0 \not =
\emptyset$ such that $x$ is feasible but not \\
\spc \spc optimal for \eqref{pd1}, for details see \cite{HungerlaenderRendl15}.\\
\spc \spc Let $(\mathcal B_{0},\mathcal D_0)$ be the optimal pair for \eqref{pd1}.\\
\spc \spc $\mathcal A \leftarrow \mathcal A_0 \cup \mathcal B_{0},\ \mathcal C \leftarrow \mathcal C_0 \cup \mathcal D_0$.\\
\spc $[\alpha,\gamma]= $KKT$(\mathcal A,\mathcal C)$\\
{\bf endwhile}
\vskip 2mm
\hrule
\caption{Algorithmic description of the feasible active set method
  from \cite{HungerlaenderRendl15}.}
\label{modified-alg_ts}
\end{table}

\section{Numerical tests}\label{sec:num}

We performed test computations in a Matlab implementation for the three examples from Section \ref{sec:appex} in $\tilde{\Omega}=\Omega=(-1,1)^2\subseteq\R^2$ using just one observation $I=1$ (thus skipping subscrips $i$ in the following) and a the piecewise constant function
\[
\mbox{test 1: }
{b_f}_{ex}(x,y)=c_{ex}(x,y)=a_{ex}(x,y)=1+10\cdot {1\!\!{\rm I}}_{B}(x,y)\,,
\]
where $B=\{(x,y)\in\R^2\, : \, (x+0.4)^2+(y+0.3)^2\leq0.04\}$
cf. Figure \ref{fig:exsol_spots}, in all these examples and
correspondingly setting $\ul{b}=\ul{c}=\ul{a}=1$,
$\ol{b}=\ol{c}=\ol{a}=11$. The boundary conditions were chosen as
$j=0$ on $\Gamma=(-1,1)\times\{-1,1\}$.

\begin{figure}[p]
\includegraphics[width=0.48\textwidth]{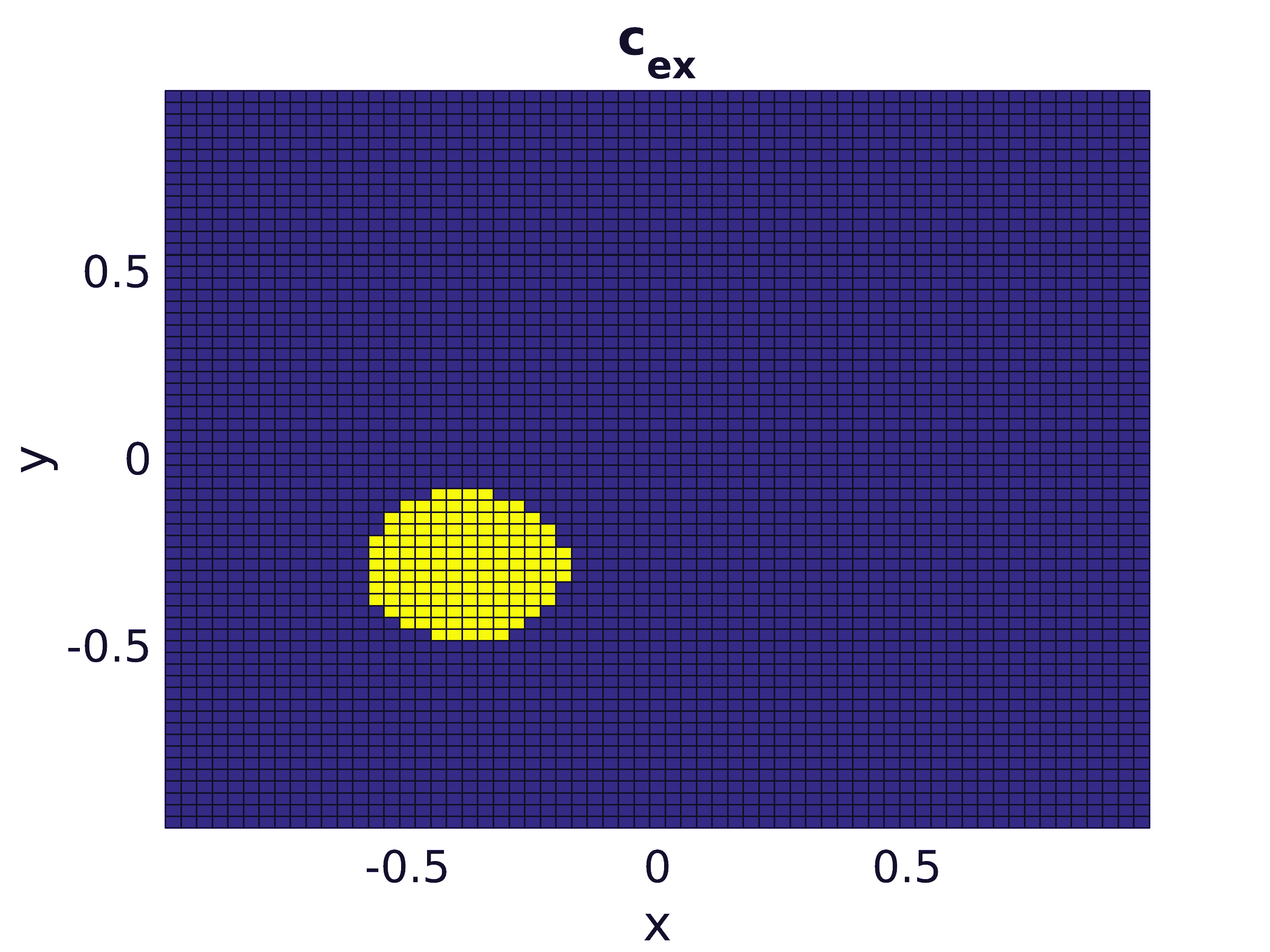}
\hspace*{0.01\textwidth}
\includegraphics[width=0.48\textwidth]{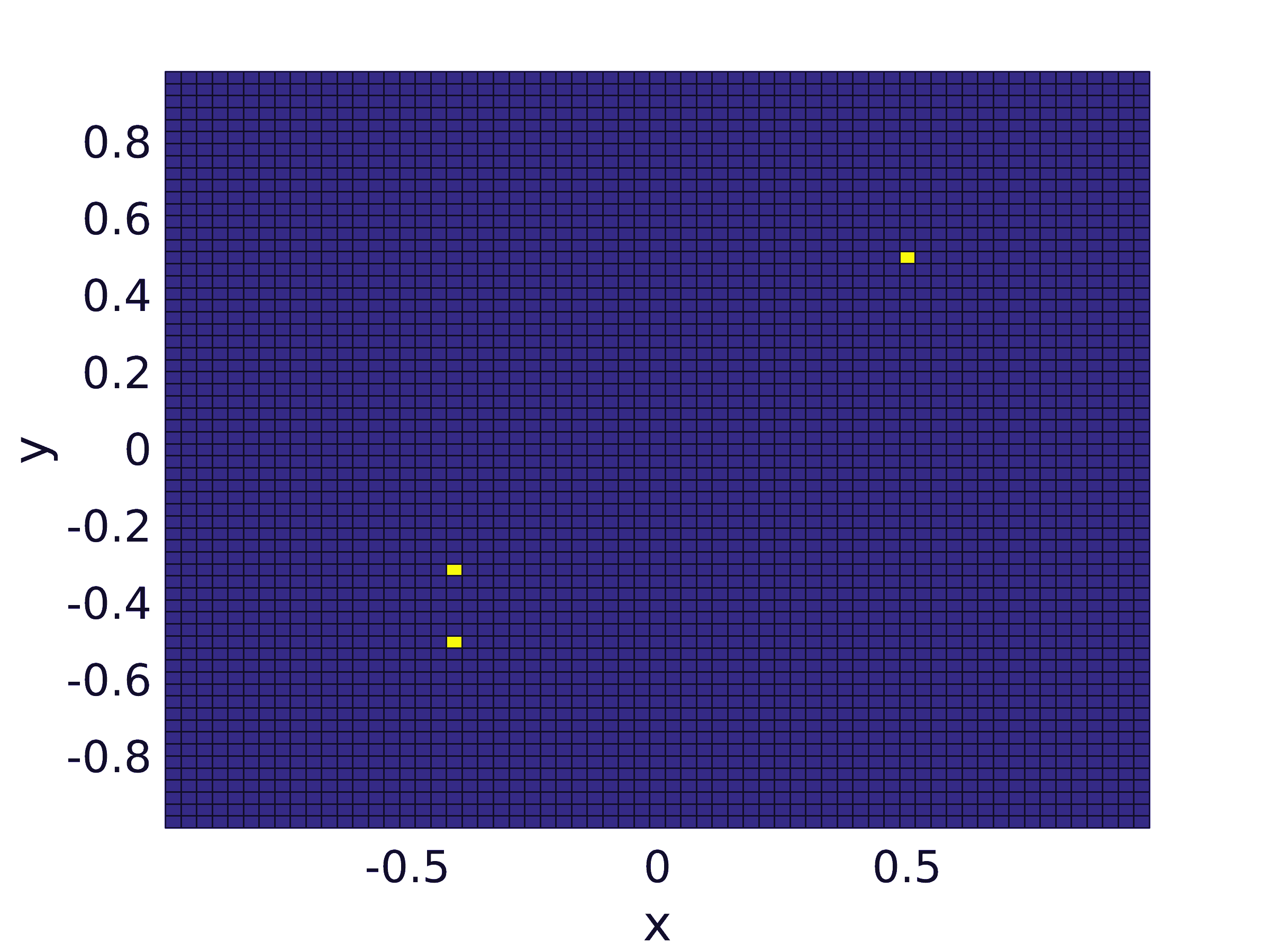}
\caption{
Left: exact coefficient ${b_f}_{ex}=c_{ex}=a_{ex}$.
Right: locations of spots for testing weak * $L^\infty$ convergence.
\label{fig:exsol_spots}}
\end{figure}

The finite element grid used in computations was determined by subdividing the unit interval into $N=32$ subintervals in both directions, leading to $(N+1)^2=1089$ gridpoints for the piecewise linear nodal discretization of $u$ and $b_f$, as well as $2*N^2=2048$ triangles for the piecewise constant element discretization $c$ and $a$. Thus in the example from Section \ref{sec_bprob} we end up with $n=2178$, in the two other examples with $n=3137$ unknowns.
Part of the experiments were also carried out on a finer grid with $N=64$ and correspondingly $n=8450$ or $n=12417$ unknowns.
In order to avoid an inverse crime, we generated the synthetic data on a finer grid for the example from Section \ref{sec_bprob}, while in the examples from Sections \ref{sec_cprob}, \ref{sec_aprob}, we additionally prescribed the exact state as $\phi_{ex}(x,y)=\cos(\frac{\pi}{2}x)\sin(\frac{\pi}{2}y)$, and computed the corresponding right hand side $f$ on the finer grid.
After projection of $\phi_{ex}$ onto the computational grid, we added uniformly distributed random noise of levels $\delta\in\{0.001, 0.01, 0.1\}$ corresponding to $0.1$, $1$ and $10$ per cent noise, to obtain synthetic data $\ydel$.
In all tests we started with the constant function with value $5.5$ for ${b_f}_0,c_0,a_0$ (i.e., the mean value between upper and lower bound) and $\phi_0\equiv\ydel$. Moreover, we always set $\tau=1.1$.

First of all, we provide a comparison of the methods described in
Section \ref{sec:QPbox}, in their Matlab implementations
\verb+Feas_AS+ \cite{HungerlaenderRendl15} and
\verb+Infeas_AS+ \cite{HungerlaenderRendl17} with \verb+quadprog+
which is the standard Matlab function for solving \eqref{pd}. The trust-region-reflective
algorithm \verb+quadprog+ is a subspace trust-region method based on the
interior-reflective Newton method from \cite{ColemanLi}.
To facilitate transparency of our numerical tests we provide the
Matlab code for generating
  and solving all instances with the various methods discussed in this
  paper under \url{http://philipphungerlaender.com/qp-code/}.

Table \ref{tab_comparison} shows the CPU times measured with the Matlab function \verb+cputime+, the $L^1$ errors as well as the errors in certain spots within the two homogeneous regions and on their interface,
\[
\mbox{spot}_1=(0.5,0.5)\,, \quad
\mbox{spot}_2=(-0.4,-0.3)\,, \quad
\mbox{spot}_3=(-0.4,-0.5)\,,
\]
cf. Figure \ref{fig:exsol_spots}, more precisely, on $\frac{1}{N}\times\frac{1}{N}$ squares located at these spots, corresponding to the piecewise constant $L^1$ functions with these supports in order to exemplarily test weak * convergence. Moreover, we provide the number $k$ of Gauss Newton SQP iterations (which is always one in the linear inverse source problem, of course) and the relative residual, i.e., the cost function ratio at the final iterate $\frac{J(x_k^\delta,u_k^\delta)}{J(x_0,u_0)}$.
We do so for the three examples from Sections \ref{sec_bprob}, \ref{sec_cprob}, and \ref{sec_aprob}, shortly referred to as  ``source'', ``potential'', and ``diffusion'', respectively, using the lowest noise level $\delta=0.001$ and, besides the discretization with $N=32$, also one with $N=64$ subintervals in both directions.
\begin{table}[p]
  \begin{center}
    \small
    \hspace*{-0.8cm}
\begin{tabular}
{|l||l|l|l||l|l|l|}
\hline
&\multicolumn{6}{|c|}{source}
\\ \hline
&\multicolumn{3}{|c||}{N=32 (n=2178)}
&\multicolumn{3}{|c|}{N=64 (n=8450)}
\\ \hline
&\verb+quadprog+&\verb+Infeas_AS+&\verb+Feas_AS+
&\verb+quadprog+&\verb+Infeas_AS+&\verb+Feas_AS+
\\ \hline
$k$&
1&1&1&1&1&1
\\ \hline
$\frac{J(x_k^\delta,u_k^\delta)}{J(x_0,u_0)}$&
6.6901e-06&6.6154e-06&6.6154e-06&
3.2887e-05&3.2532e-05&3.2744e-05
\\ \hline
$\mbox{err}_{spot_1}$&
2.4023e-06&0&0&
6.7780e-07&0&0
\\ \hline
$\mbox{err}_{spot_2}$&
7.3665e-06&0&0&
0.0462&0&0
\\ \hline
$\mbox{err}_{spot_3}$&
1.0890e-05&0&0&
1.3992&1.3917&1.3917
\\ \hline
$\mbox{err}_{L^1(\Omega)}$&
0.0462&0.0465&0.0465&
0.0832&0.0834&0.0834
\\ \hline
CPU&
1.68&1.16&1.09&
29.62&112.97&22.85
\\ \hline
\hline
&\multicolumn{6}{|c|}{potential}
\\ \hline
&\multicolumn{3}{|c||}{N=32 (n=3137)}
&\multicolumn{3}{|c|}{N=64 (n=12417)}
\\ \hline
&\verb+quadprog+&\verb+Infeas_AS+&\verb+Feas_AS+
&\verb+quadprog+&\verb+Infeas_AS+&\verb+Feas_AS+
\\ \hline
$k$&
4&4&6&
3&3&3
\\ \hline
$\frac{J(x_k^\delta,u_k^\delta)}{J(x_0,u_0)}$&
6.2155e-06&6.0569e-06&8.0316e-07&
1.6318e-04&1.6286e-04&1.6286e-04
\\ \hline
$\mbox{err}_{spot_1}$&
6.9611e-10&0&0&
1.2184e-10&0&0
\\ \hline
$\mbox{err}_{spot_2}$&
1.7502e-06&0&0&
0.7183&0.7145&0.7145
\\ \hline
$\mbox{err}_{spot_3}$&
1.5392&1.6093&1.4363&
3.2625&3.2629&3.2629
\\ \hline
$\mbox{err}_{L^1(\Omega)}$&
0.1051&0.1044&0.0938&
0.1460&0.1444&0.1444
\\ \hline
CPU&
21.88&24.44&18.34&
432.06&368.67&276.17
\\ \hline
\hline
&\multicolumn{6}{|c|}{diffusion}
\\ \hline
&\multicolumn{3}{|c||}{N=32 (n=3137)}
&\multicolumn{3}{|c|}{N=64 (n=12417)}
\\ \hline
&\verb+quadprog+&\verb+Infeas_AS+&\verb+Feas_AS+
&\verb+quadprog+&\verb+Infeas_AS+&\verb+Feas_AS+
\\ \hline
$k$&
4&4&4&
8&8&8
\\ \hline
$\frac{J(x_k^\delta,u_k^\delta)}{J(x_0,u_0)}$&
0.0931&0.0931&0.0931&
0.0543&0.0543&0.0543
\\ \hline
$\mbox{err}_{spot_1}$&
3.7303e-14&0&0&
3.6415e-14&0&0
\\ \hline
$\mbox{err}_{spot_2}$&
4.4418&4.4418&4.4418&
7.8278e-05&0&0
\\ \hline
$\mbox{err}_{spot_3}$&
0.3200&0.3199&0.3199&
4.3077e-13&0&0
\\ \hline
$\mbox{err}_{L^1(\Omega)}$&
0.3259&0.3259&0.3259&
0.3799&0.3799&0.3799
\\ \hline
CPU&
32.01&6.30&5.25&
1193.00&532.57&463.89
\\ \hline
\end{tabular}
\\[2ex]
\caption{Comparison of different solvers for QPs with box constraints.
  \label{tab_comparison}}
\end{center}
\end{table}
We observe that \verb+Infeas_AS+ and \verb+Feas_AS+ obviously reach the same solution, which is almost bang-bang, as the vanishing errors in almost all spots show, whereas \verb+quadprog+, being an interior point method, exhibits small deviations from the bounds.
In all cases \verb+Feas_AS+ outperforms the other two methods as far as CPU times are concerned. Therefore the following computations were done using \verb+Feas_AS+.

To provide an illustration of convergence as the noise level tends to
zero, we performed five runs on each noise level for each example and
list the average errors in Table \ref{tab_deltas}. In the diffusion
example, where we need to regularize, we set
$\alpha=1.e-4\cdot\delta$.
Correspondingly we provide an illustration of the convergence history
for the inverse potential problem from Section \ref{sec_cprob} for two
different noise levels $\delta=0.1$ and $\delta=0.01$ in Figures
\ref{fig:convdel01}, \ref{fig:convdel001}.

\begin{table}[p]
  \begin{center}
    \small
    \hspace*{-0.8cm}
\begin{tabular}
{|l||l|l|l||l|l|l||l|l|l|}
\hline
&\multicolumn{3}{|c|| }{source}
&\multicolumn{3}{|c||}{potential}
&\multicolumn{3}{|c|}{diffusion}
\\ \hline
$\delta$&0.001&0.01&0.1&0.001&0.01&0.1&0.001&0.01&0.1
\\ \hline
$\mbox{err}_{spot_1}$&
         0&
    0.2000&
    0.2000&
     0&
     0&
     0&
     0&
     0&
     0
\\ \hline
$\mbox{err}_{spot_2}$&
         0&
    2.7488&
    4.0702&
         0&
    0.7960&
    4.8689&
         0&
    8.4141&
    9.8436
\\ \hline
$\mbox{err}_{spot_3}$&
         0&
    0.5678&
    1.9445&
    1.0840&
    2.1512&
    2.5862&
    0.6572&
         0&
         0
\\ \hline
$\mbox{err}_{L^1(\Omega)}$&
    0.0472&
    0.5288&
    0.5721&
    0.1472&
    0.2136&
    0.3671&
    0.7200&
    0.3783&
    0.3745
\\ \hline
\end{tabular}
\\[2ex]
\caption{Convergence as $\delta\to0$: averaged errors of five test runs with uniform noise.
  \label{tab_deltas}}
\end{center}
\end{table}

\begin{figure}[p]
\includegraphics[width=0.31\textwidth]{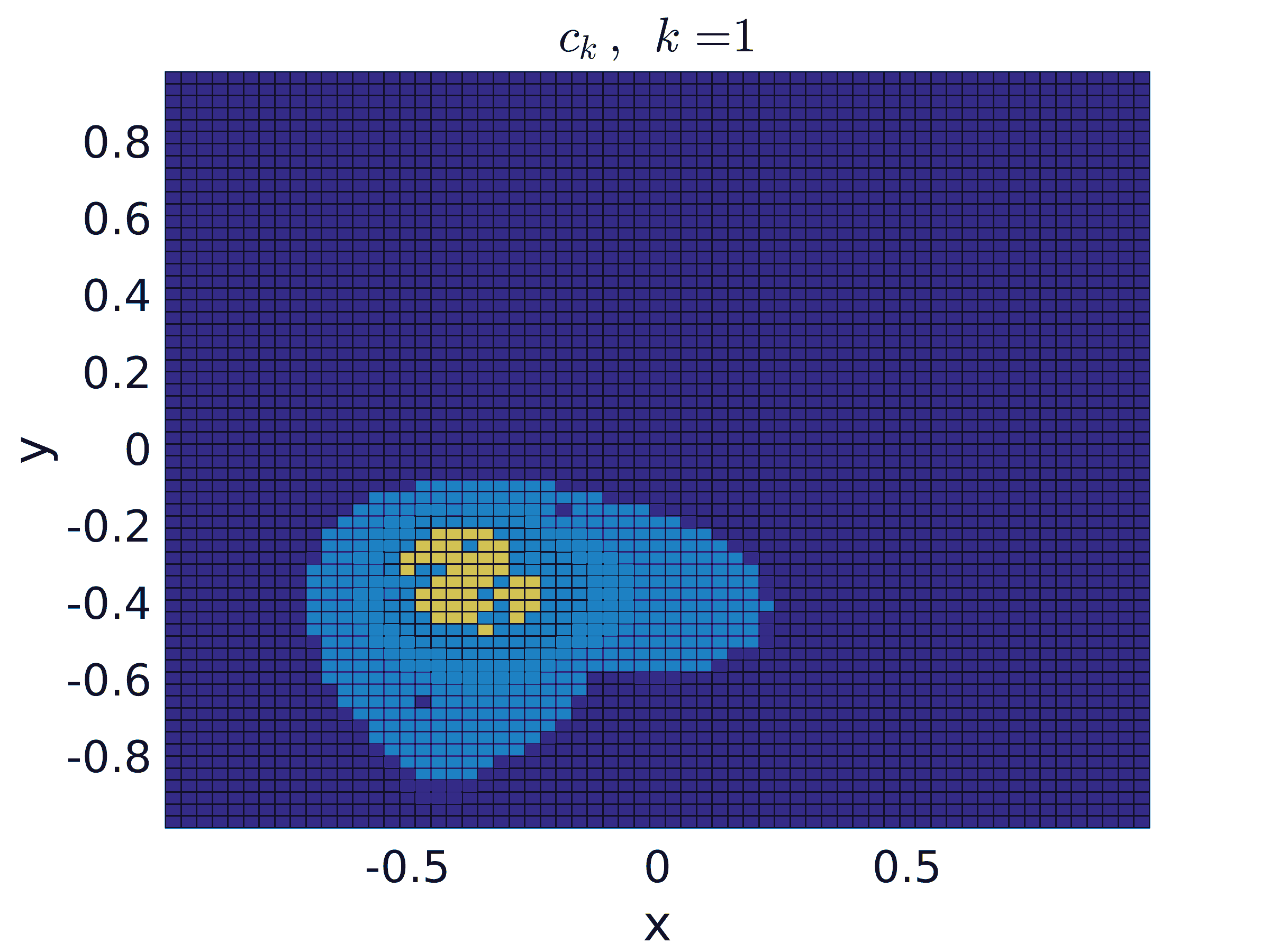}
\hspace*{0.01\textwidth}
\includegraphics[width=0.31\textwidth]{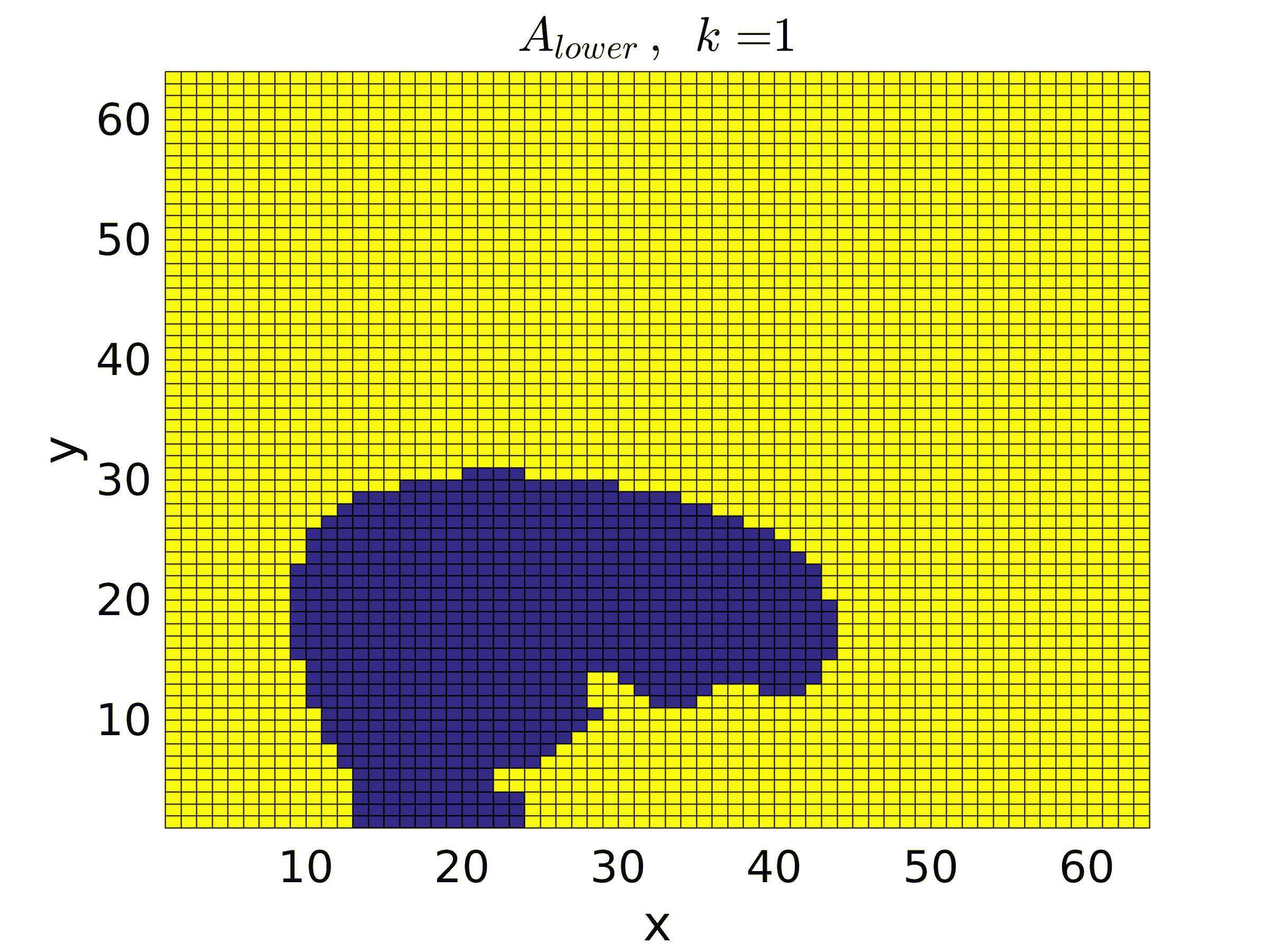}
\hspace*{0.01\textwidth}
\includegraphics[width=0.31\textwidth]{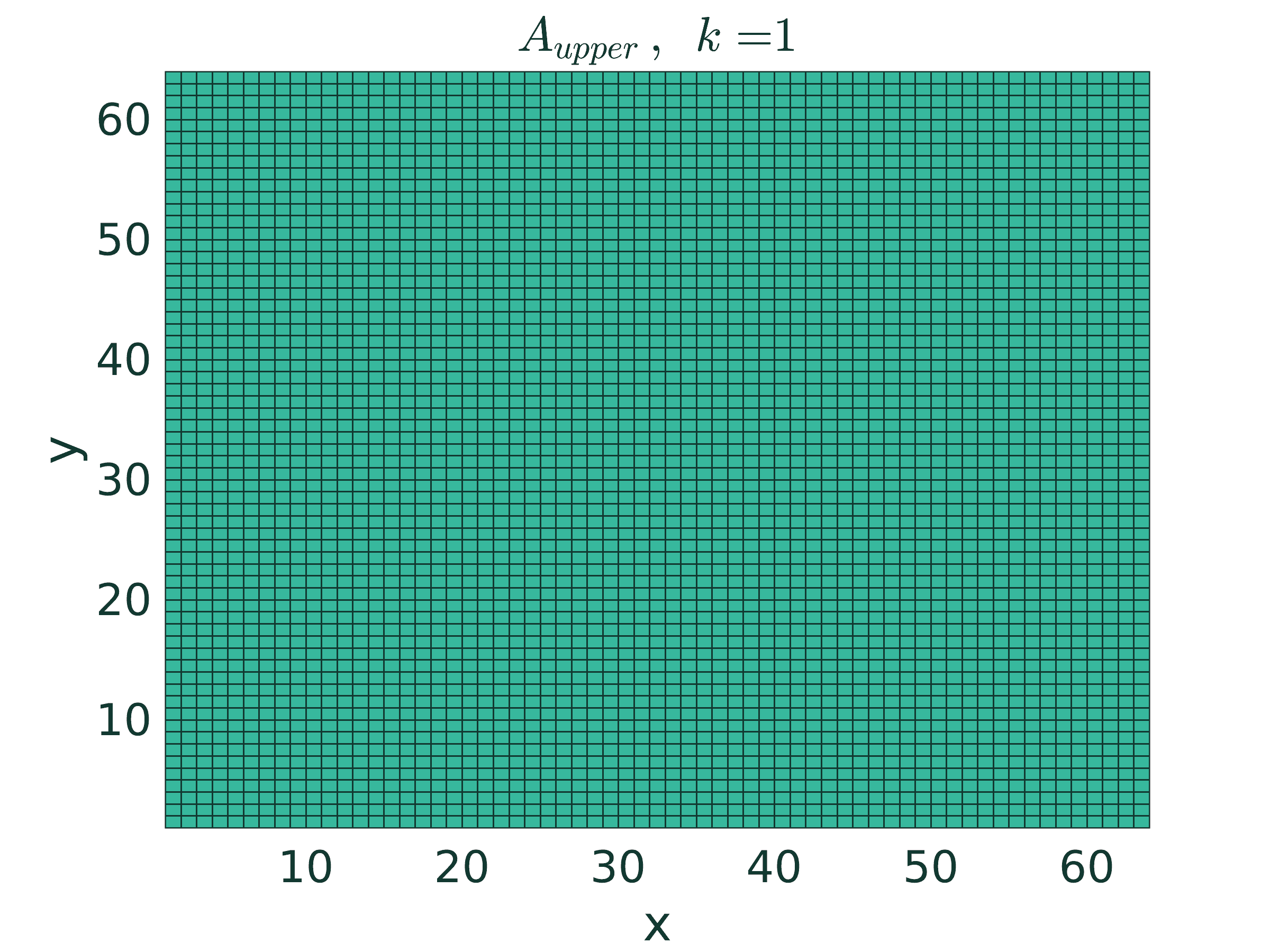}\\
\includegraphics[width=0.31\textwidth]{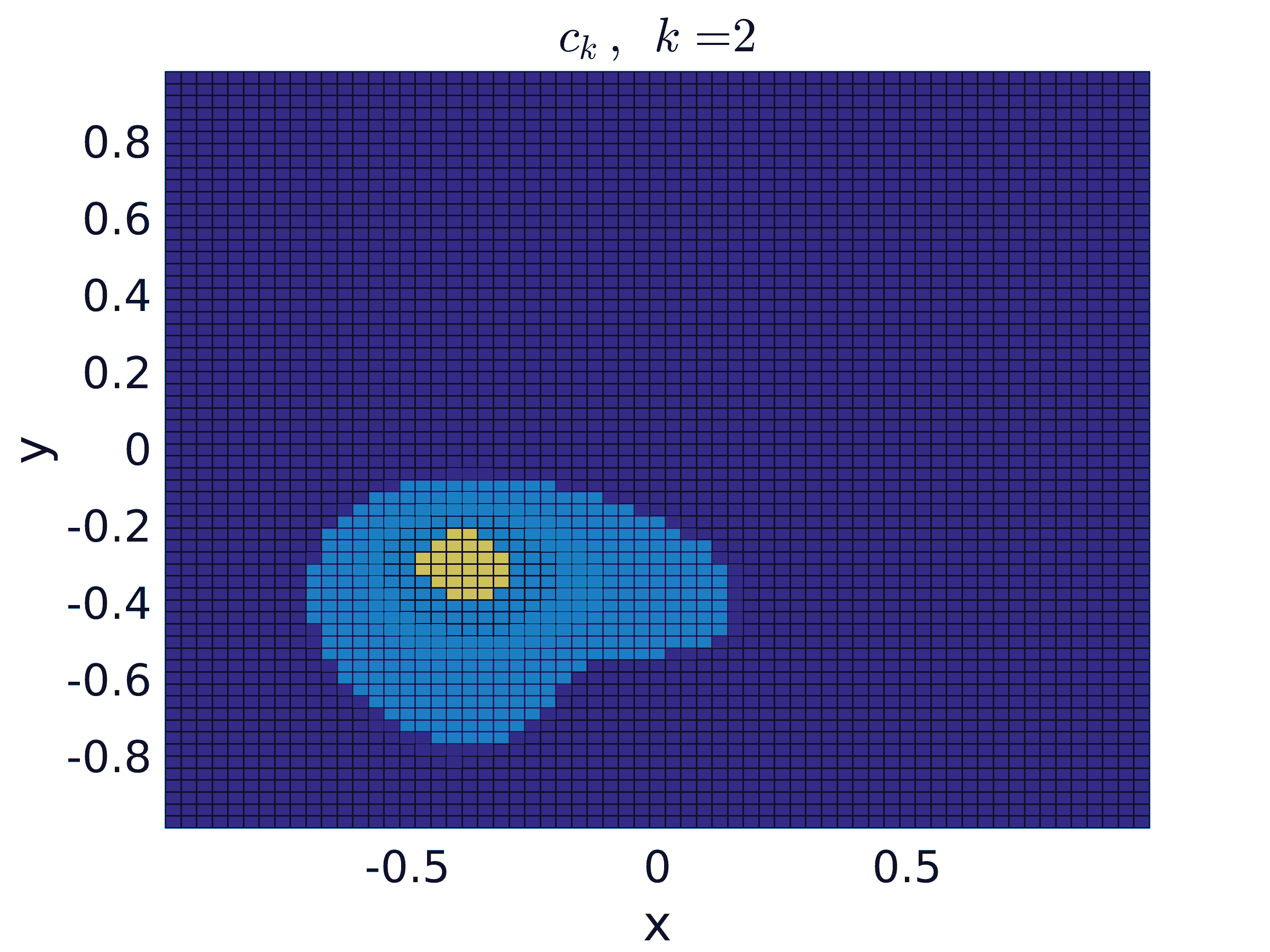}
\hspace*{0.01\textwidth}
\includegraphics[width=0.31\textwidth]{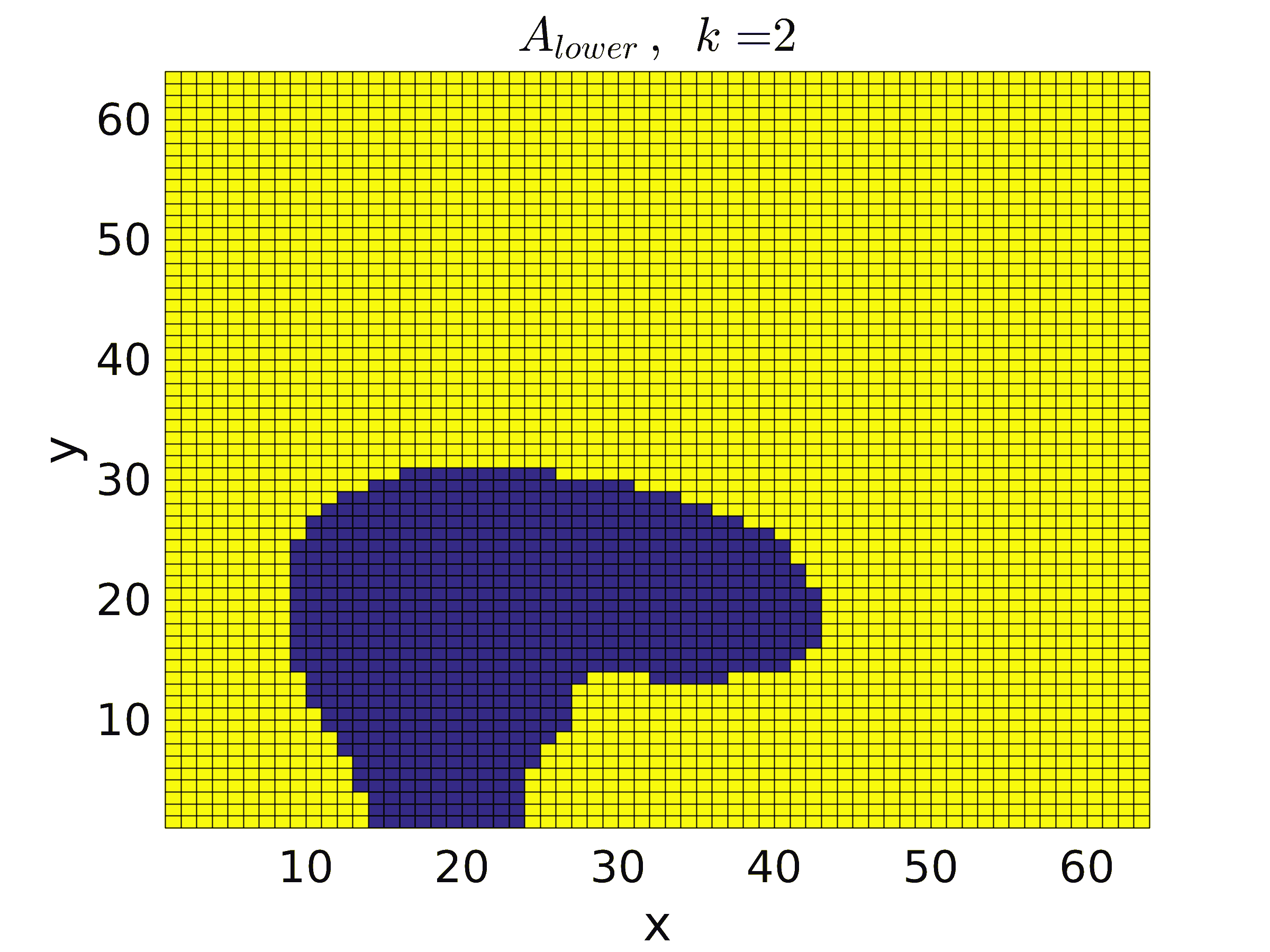}
\hspace*{0.01\textwidth}
\includegraphics[width=0.31\textwidth]{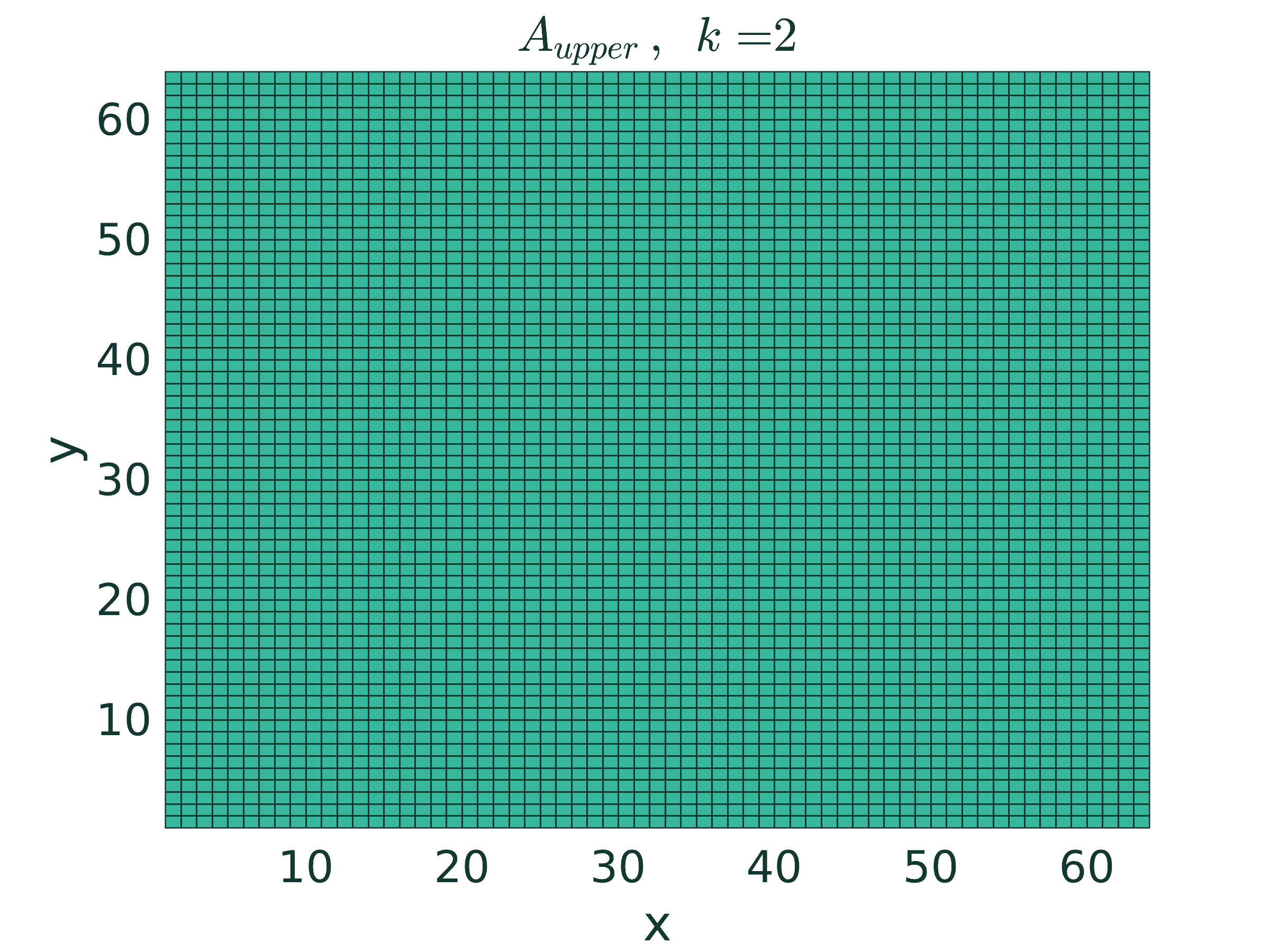}\\
\includegraphics[width=0.31\textwidth]{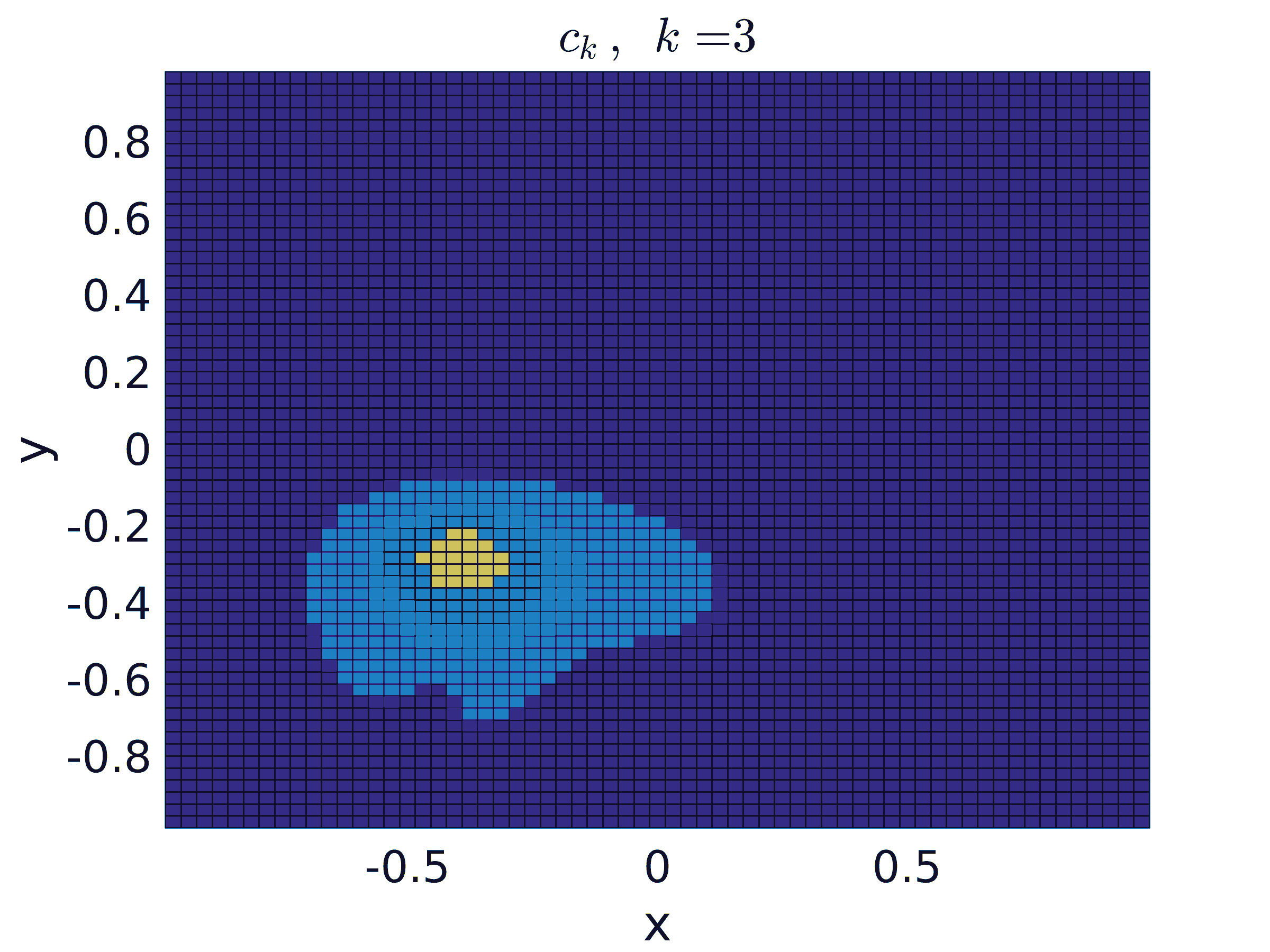}
\hspace*{0.01\textwidth}
\includegraphics[width=0.31\textwidth]{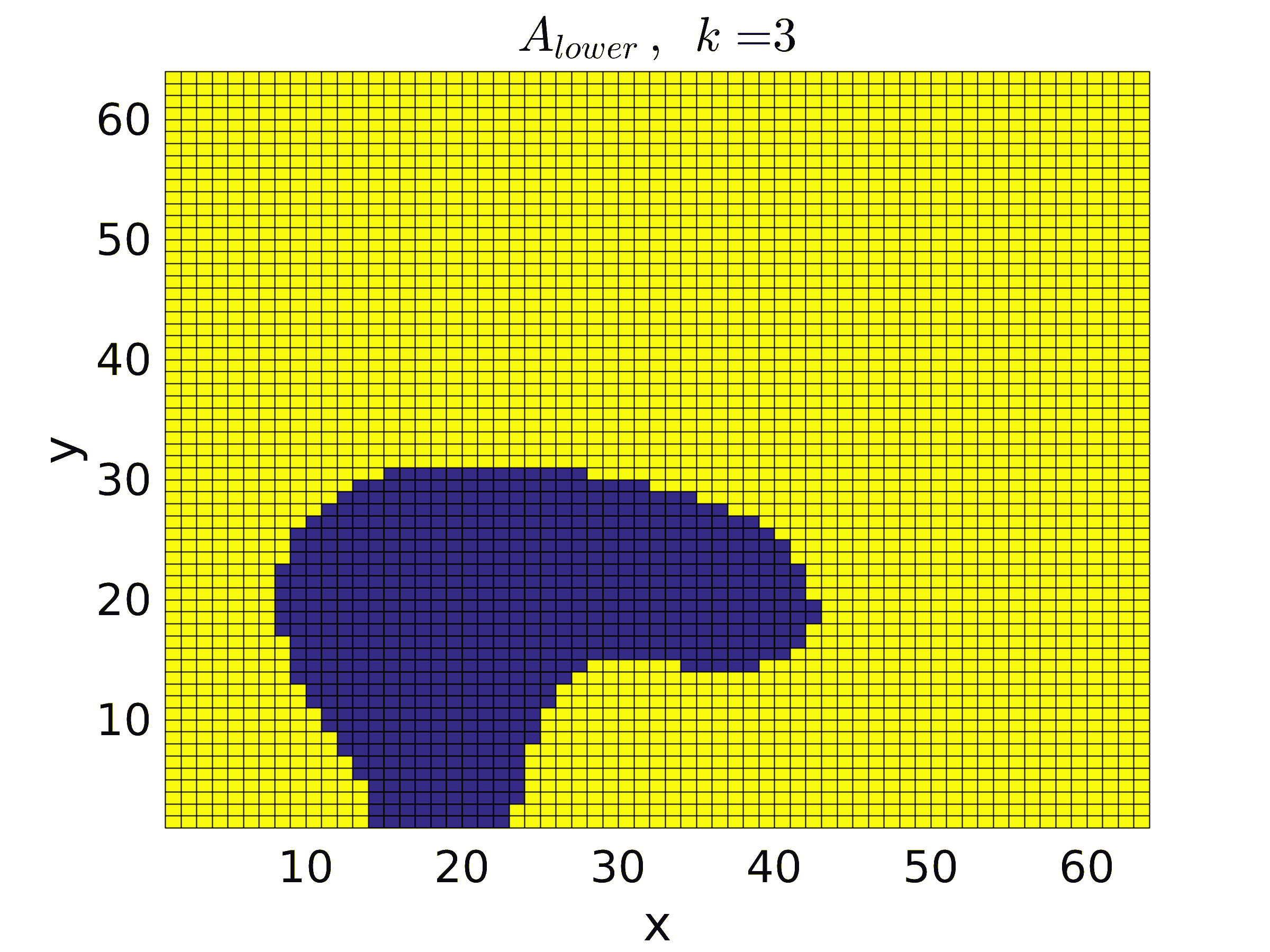}
\hspace*{0.01\textwidth}
\includegraphics[width=0.31\textwidth]{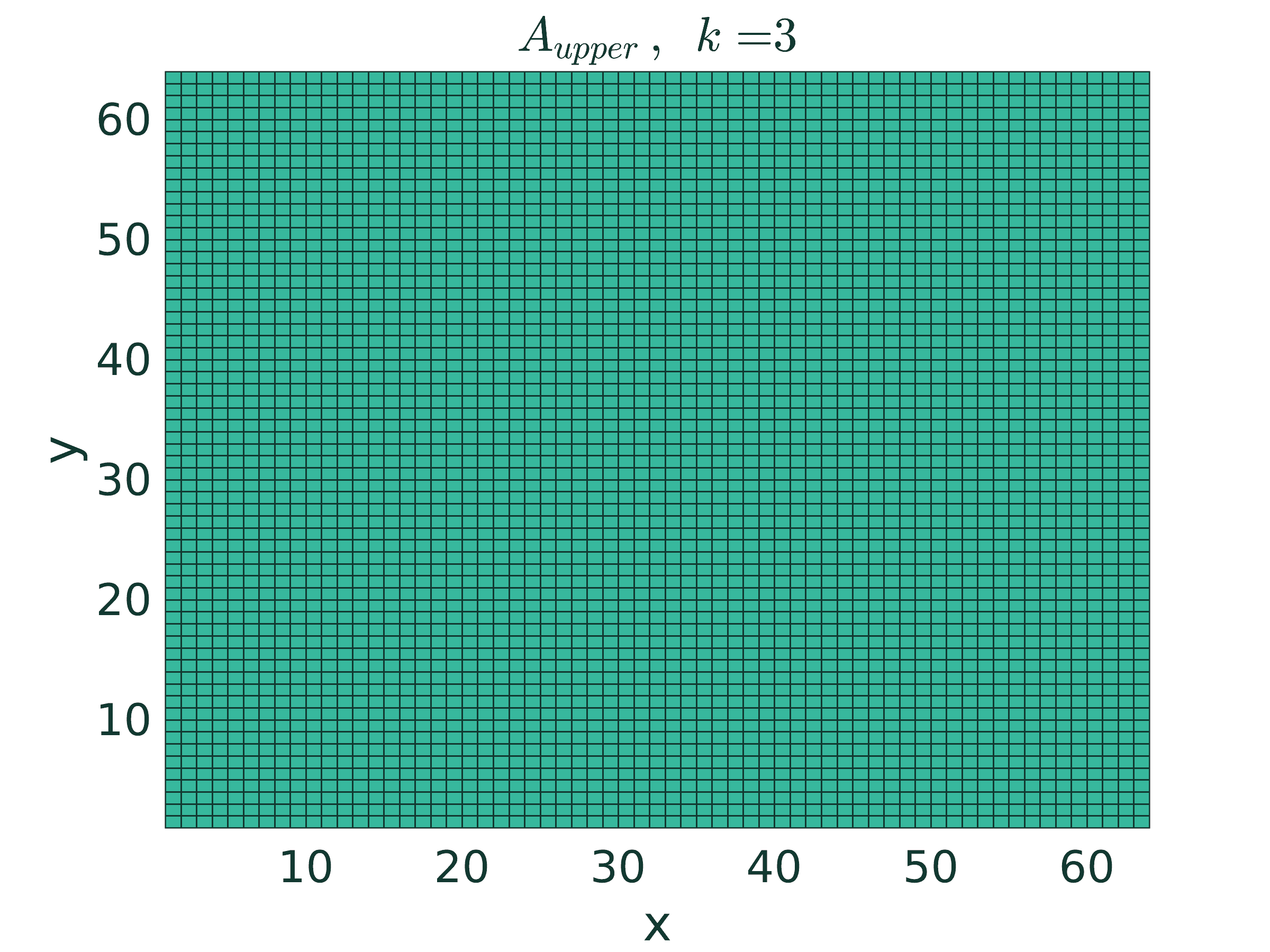}\\
\includegraphics[width=0.31\textwidth]{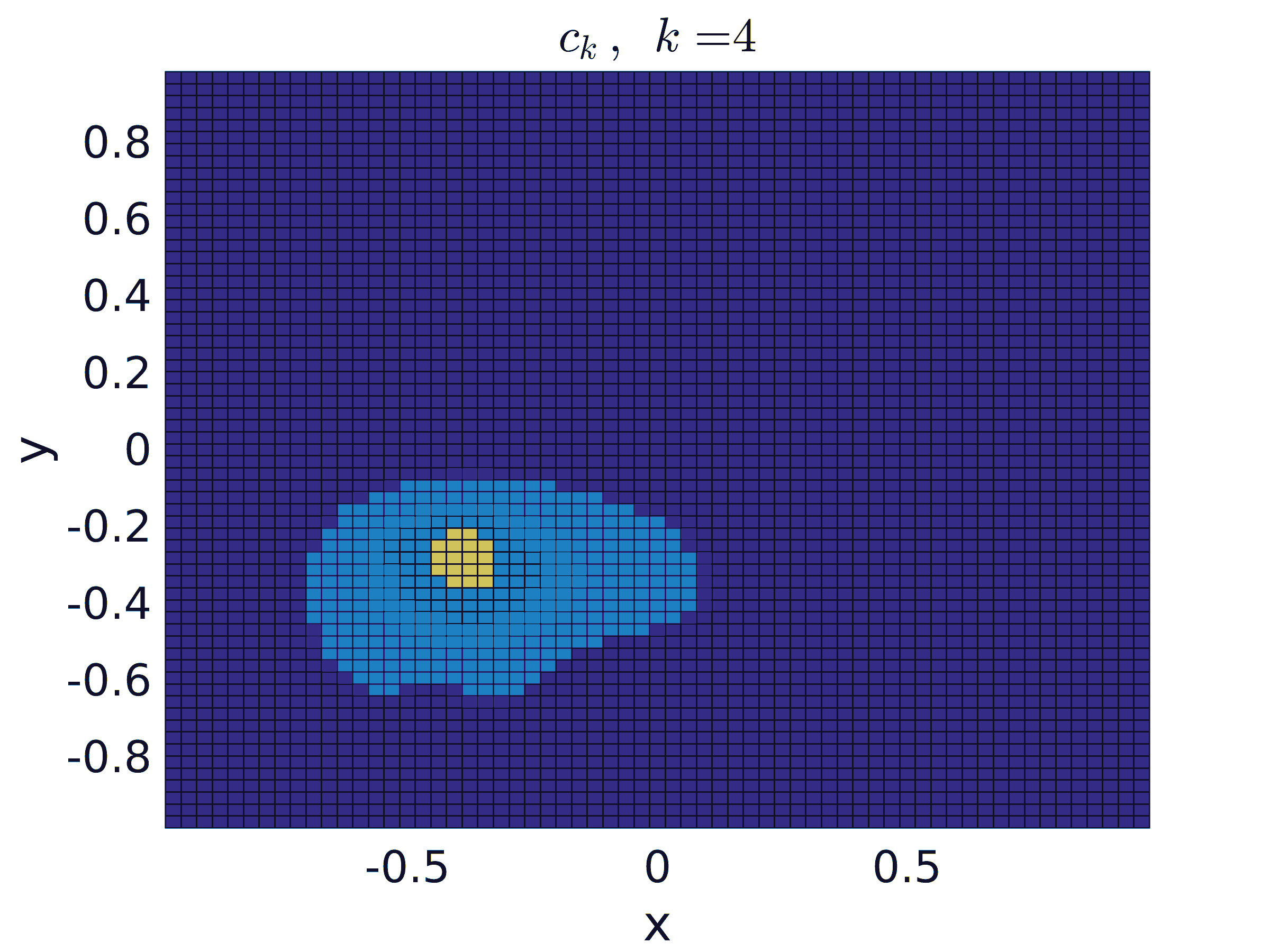}
\hspace*{0.01\textwidth}
\includegraphics[width=0.31\textwidth]{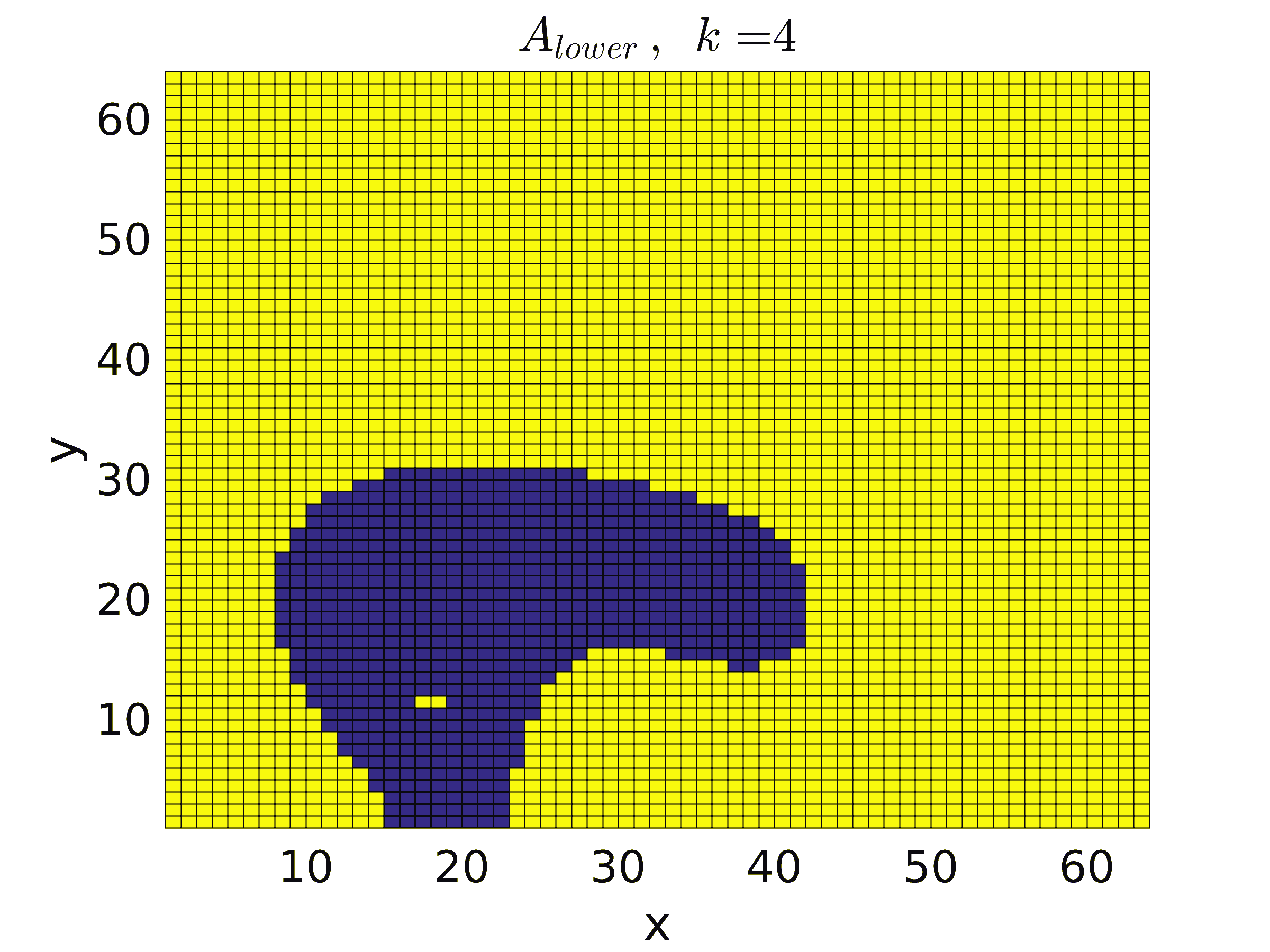}
\hspace*{0.01\textwidth}
\includegraphics[width=0.31\textwidth]{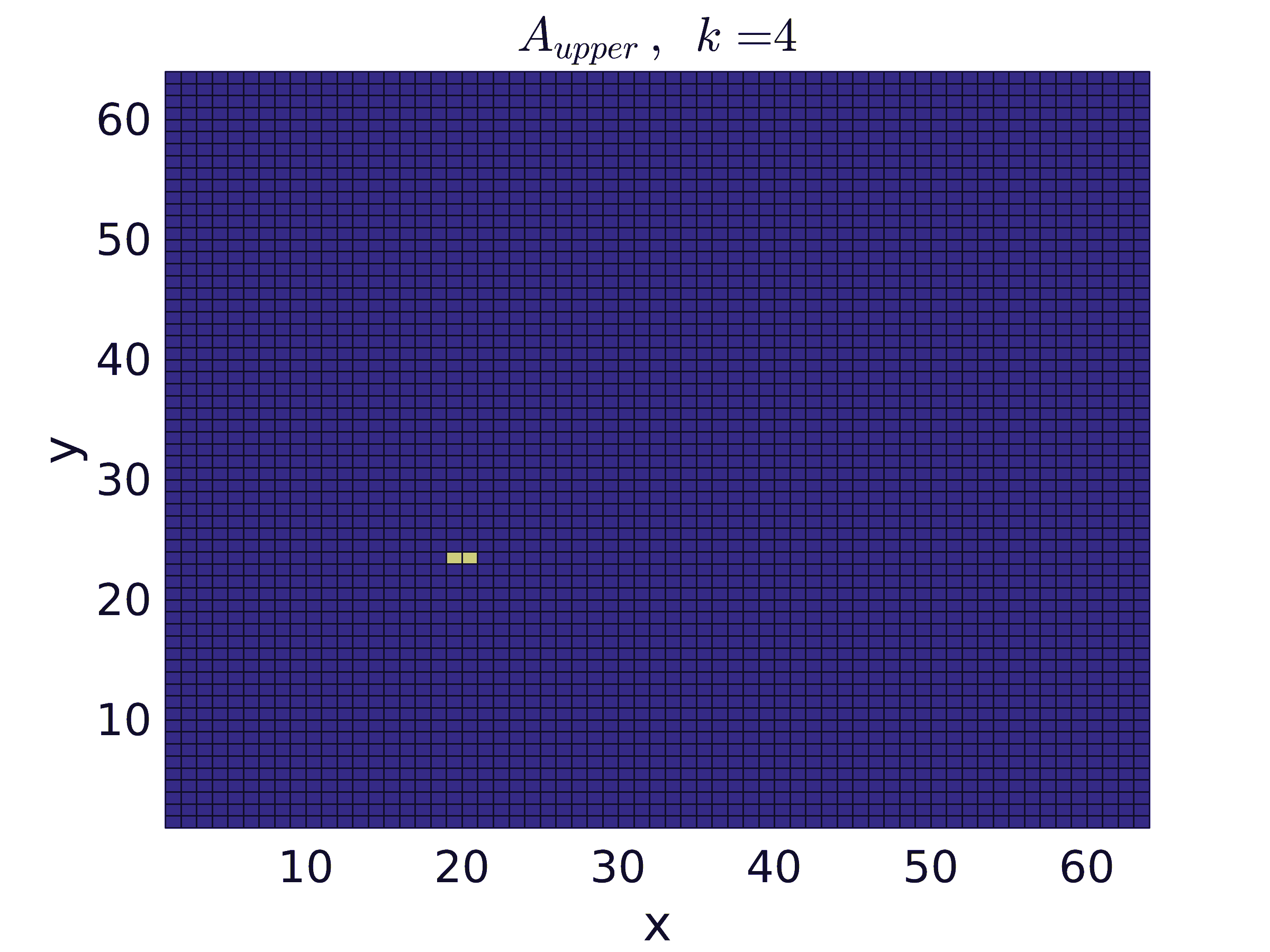}\\
\caption{
Left: reconstructed coefficient $c_k$;
Middle: active set for lower bound;
Right: active set for upper bound;
For $k=1,2,3,4$ (top to bottom) and $\delta=0.1$.
\label{fig:convdel01}}
\end{figure}
\begin{figure}[p]
\includegraphics[width=0.31\textwidth]{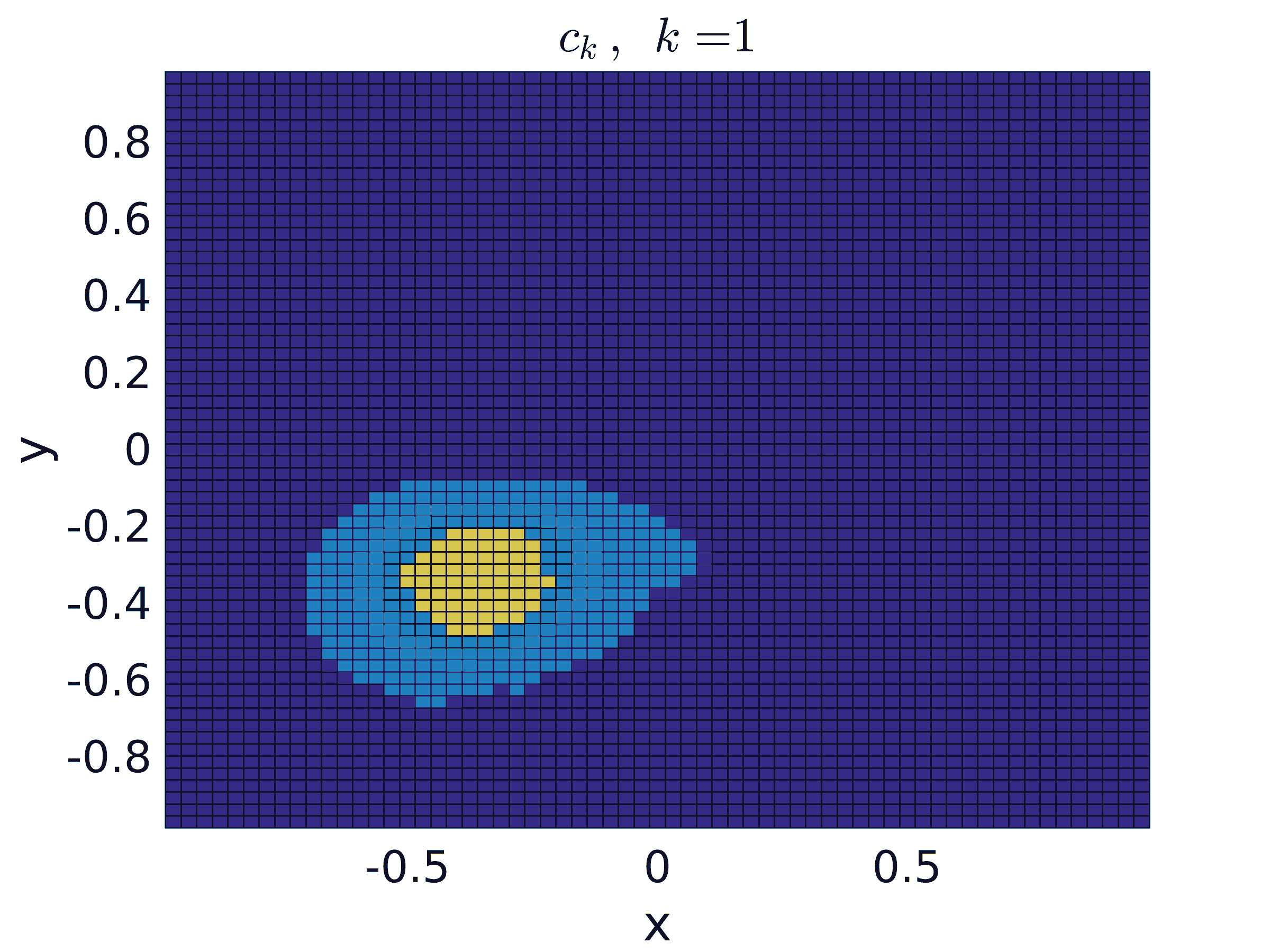}
\hspace*{0.01\textwidth}
\includegraphics[width=0.31\textwidth]{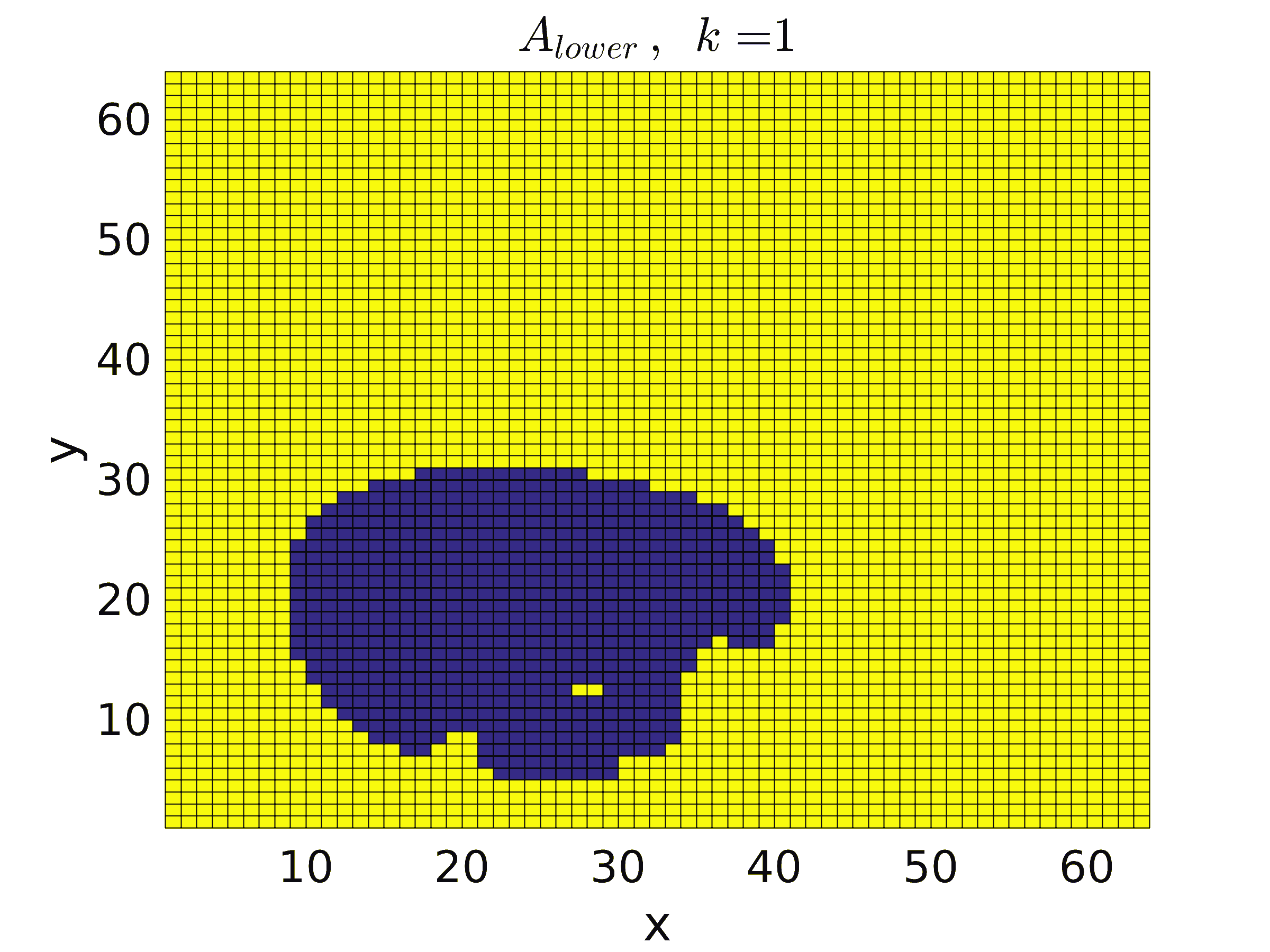}
\hspace*{0.01\textwidth}
\includegraphics[width=0.31\textwidth]{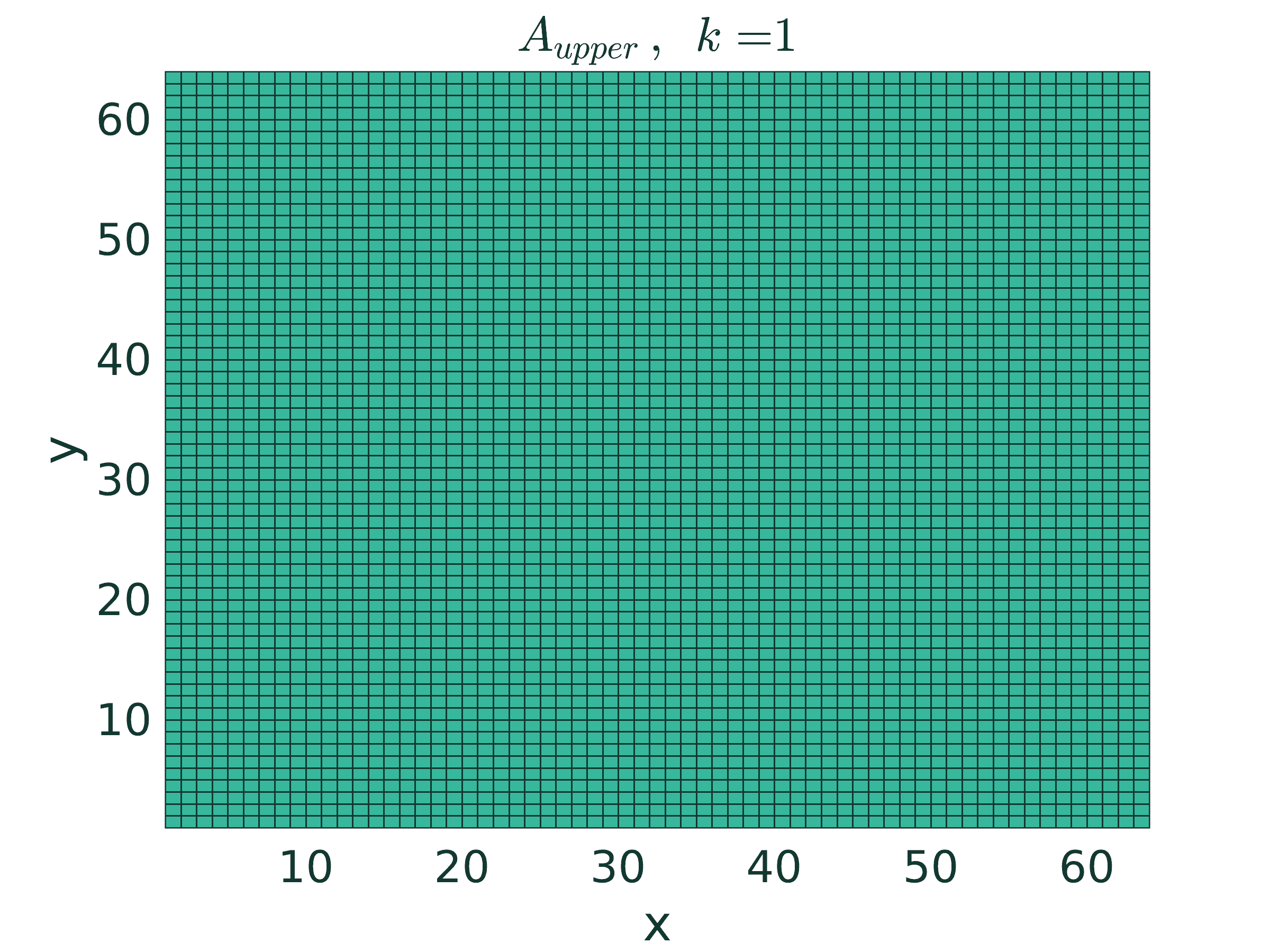}\\
\includegraphics[width=0.31\textwidth]{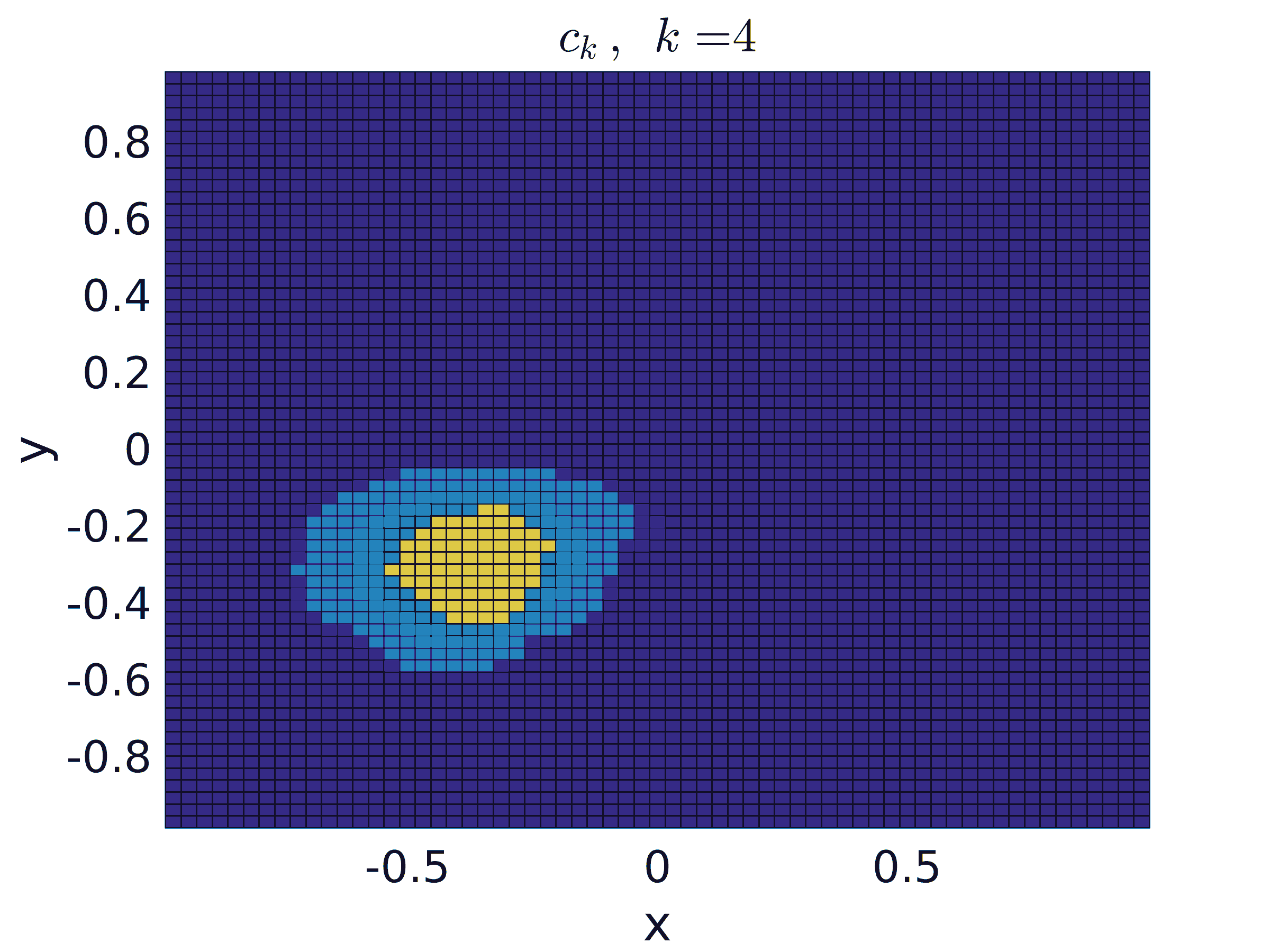}
\hspace*{0.01\textwidth}
\includegraphics[width=0.31\textwidth]{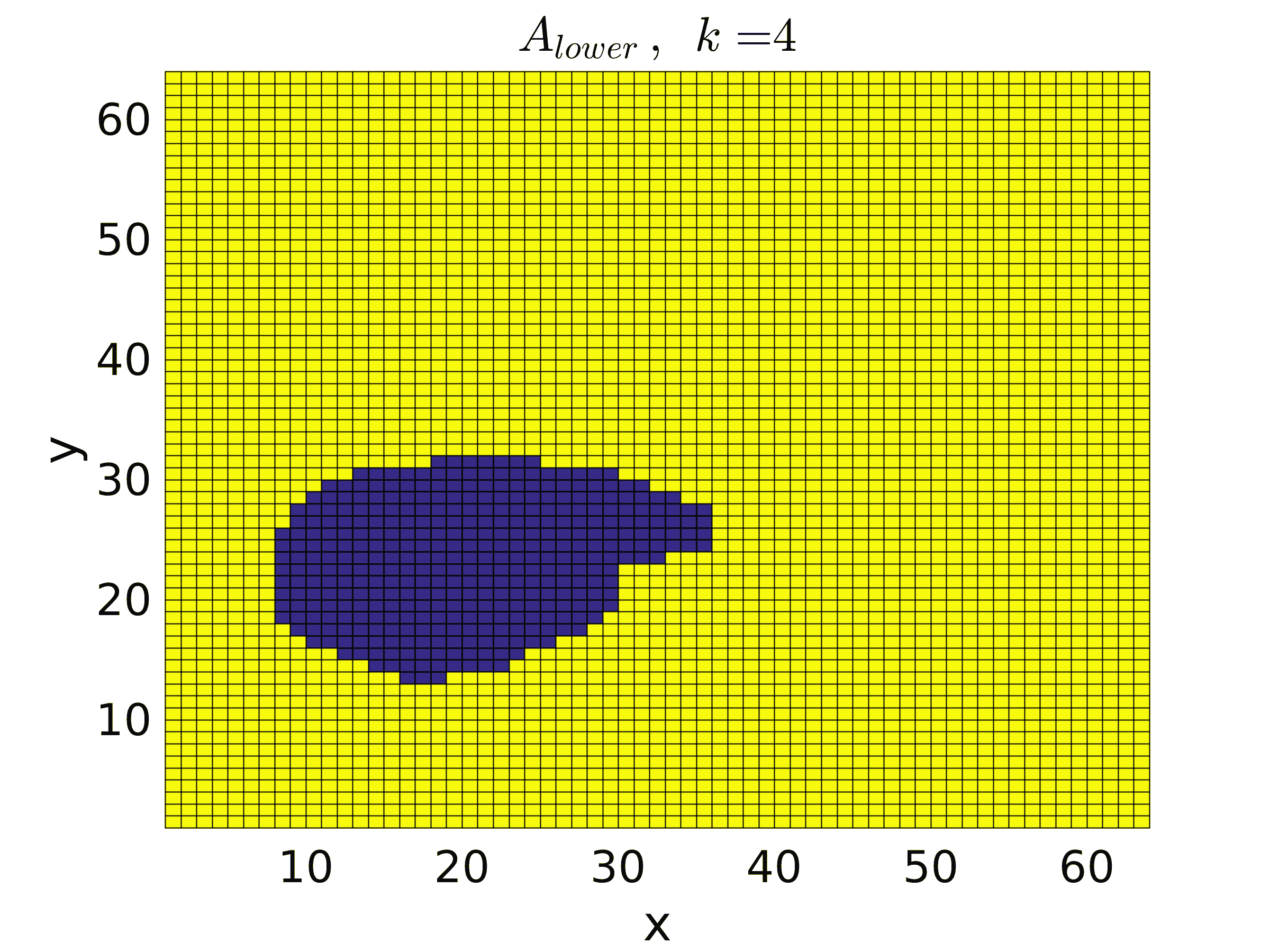}
\hspace*{0.01\textwidth}
\includegraphics[width=0.31\textwidth]{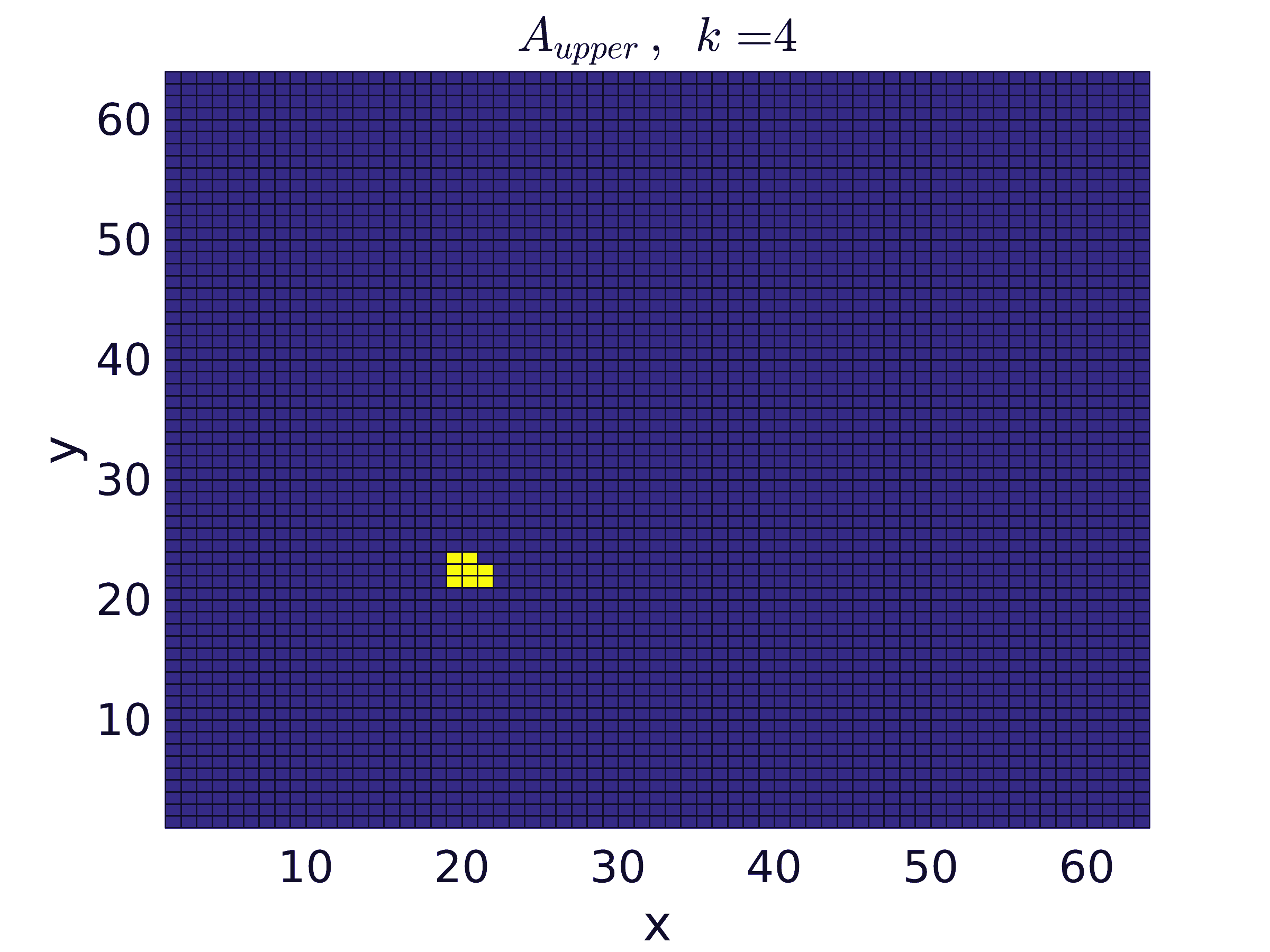}\\
\includegraphics[width=0.31\textwidth]{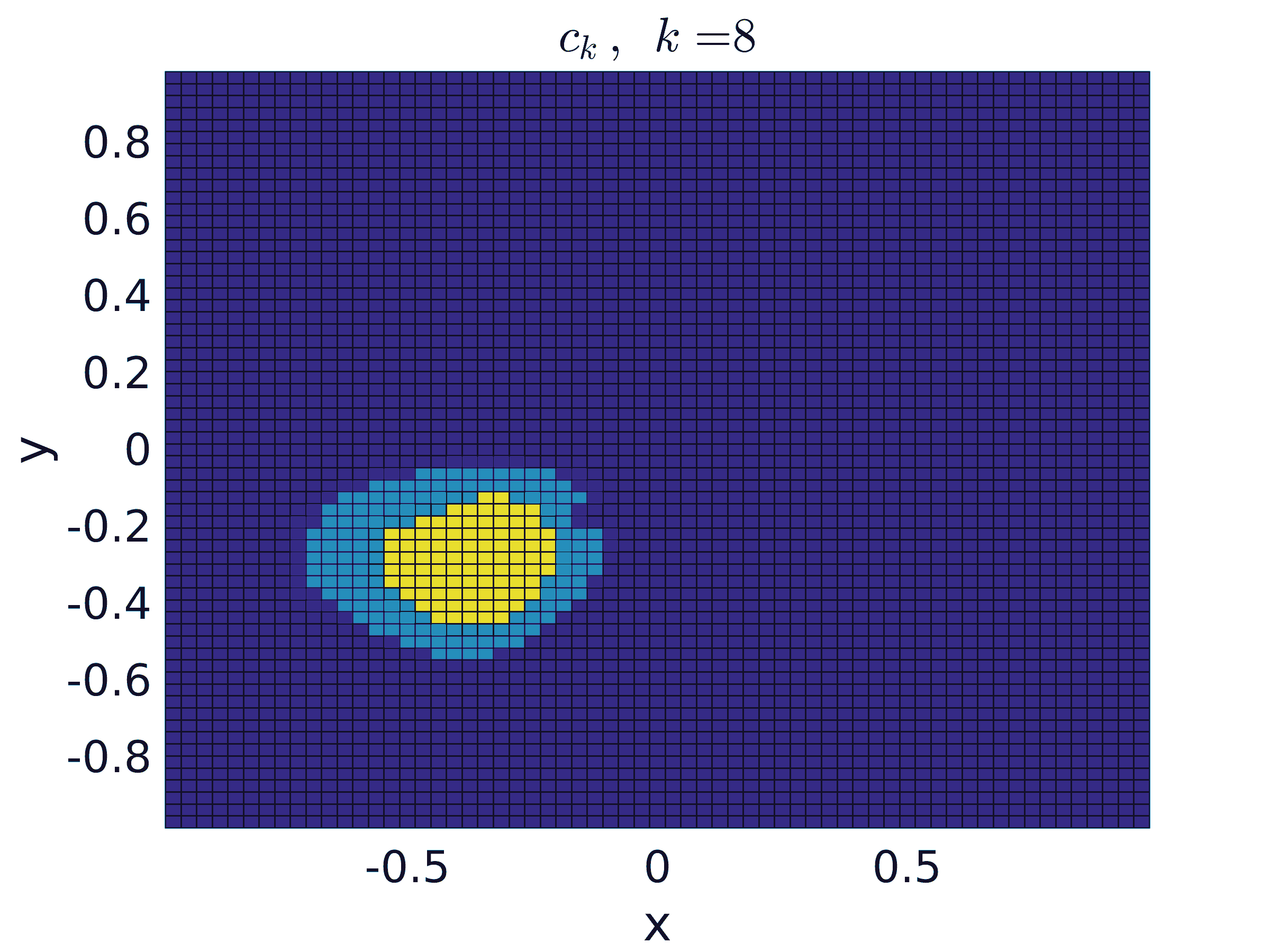}
\hspace*{0.01\textwidth}
\includegraphics[width=0.31\textwidth]{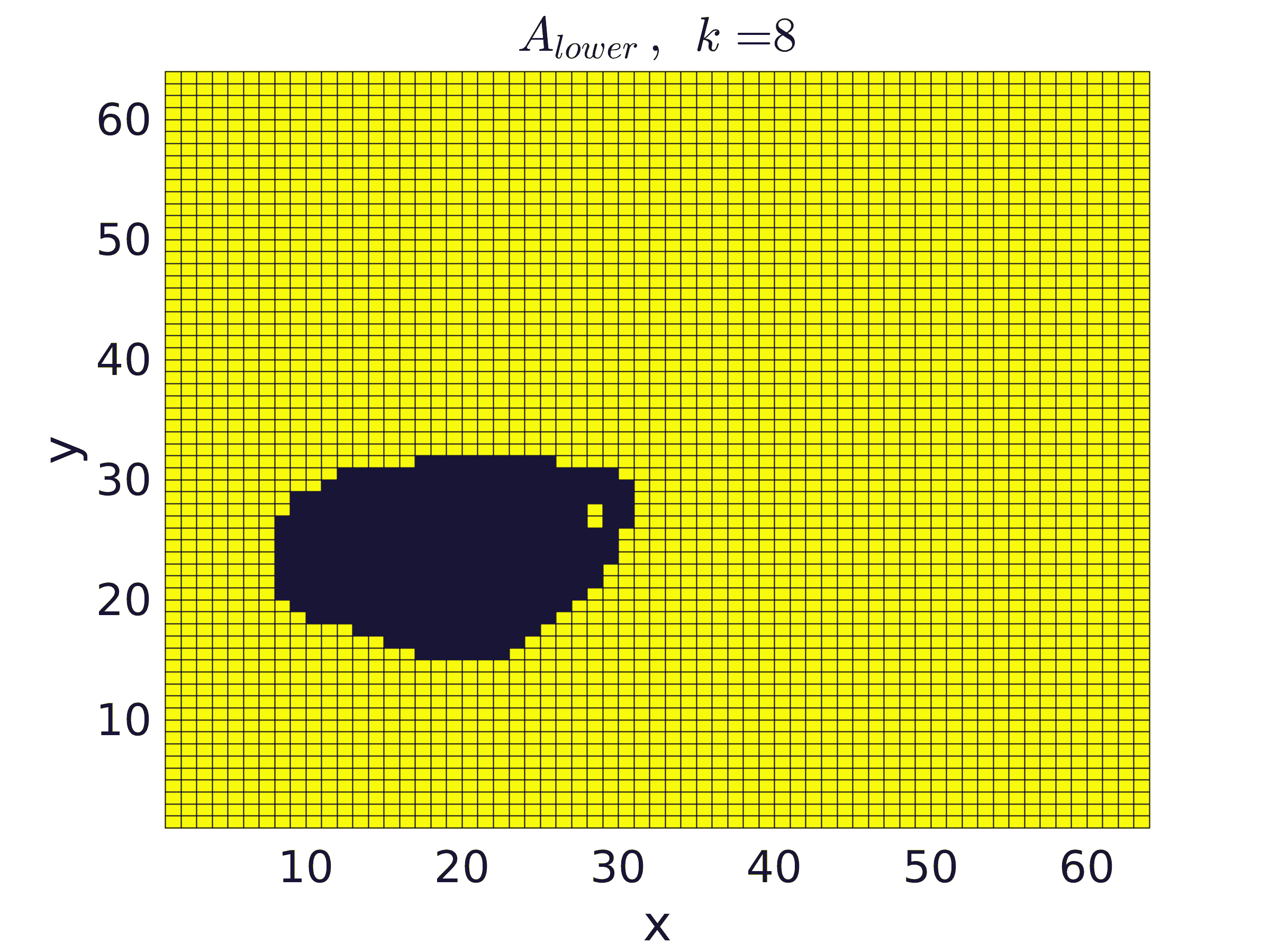}
\hspace*{0.01\textwidth}
\includegraphics[width=0.31\textwidth]{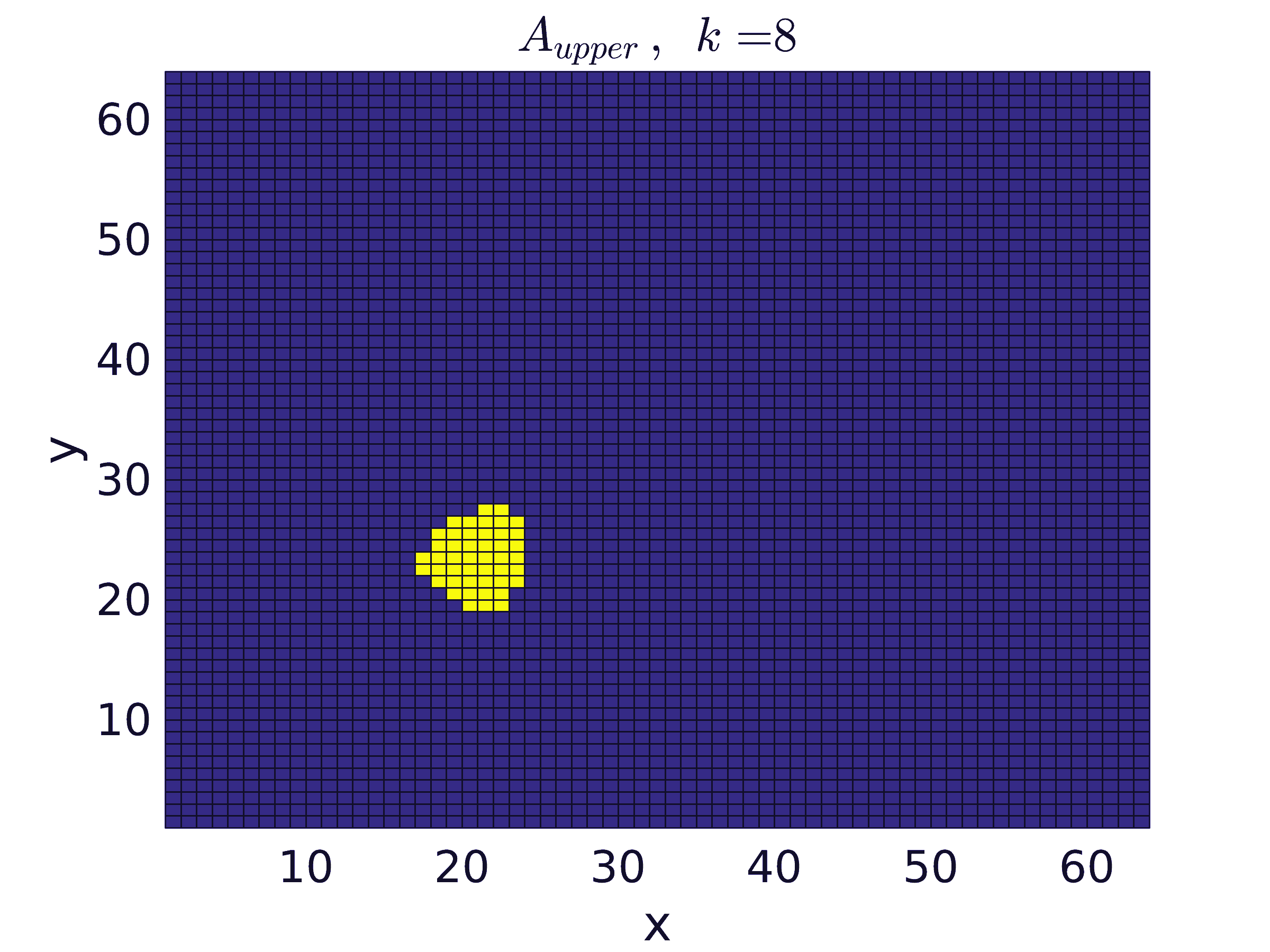}\\
\includegraphics[width=0.31\textwidth]{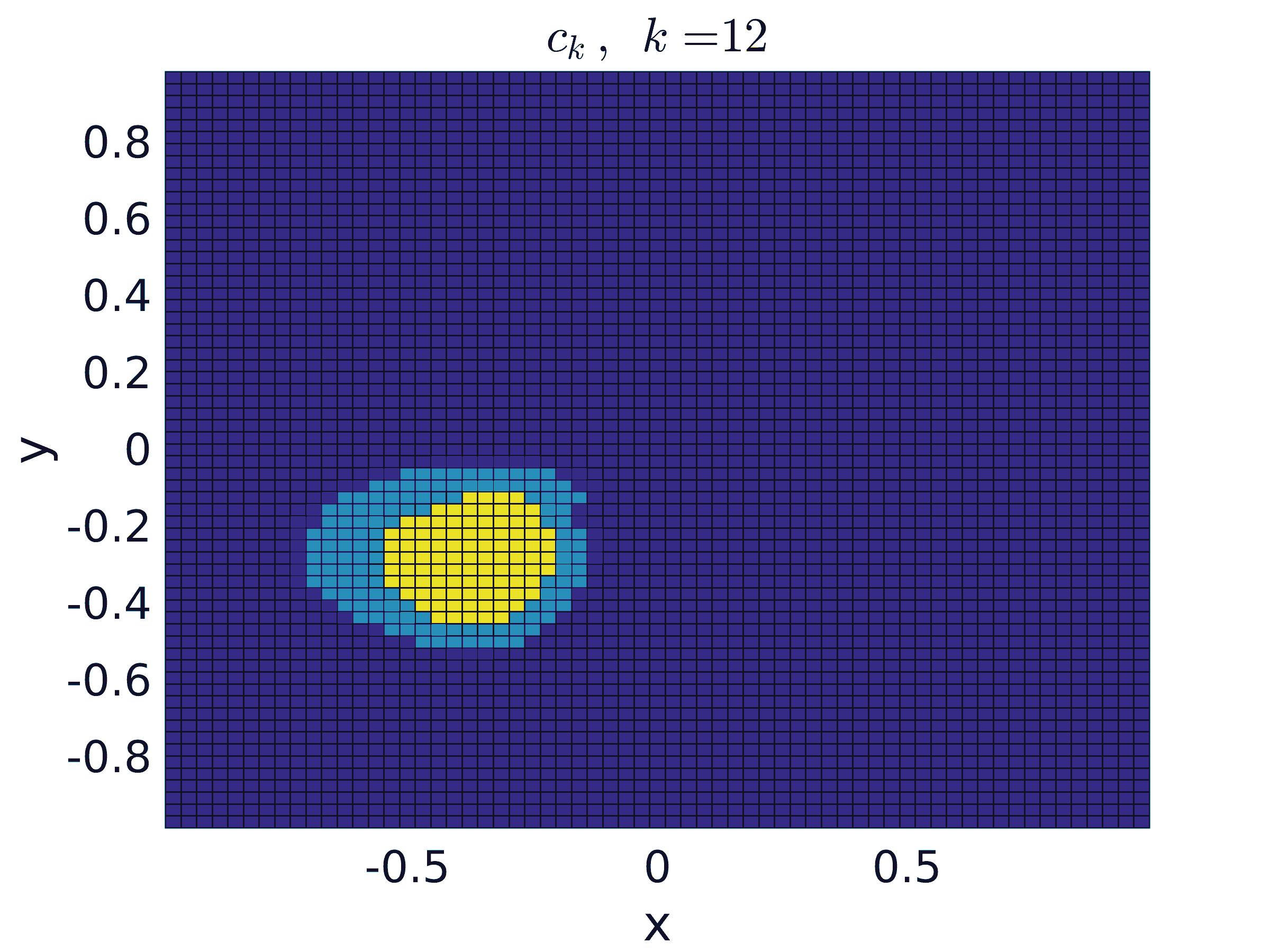}
\hspace*{0.01\textwidth}
\includegraphics[width=0.31\textwidth]{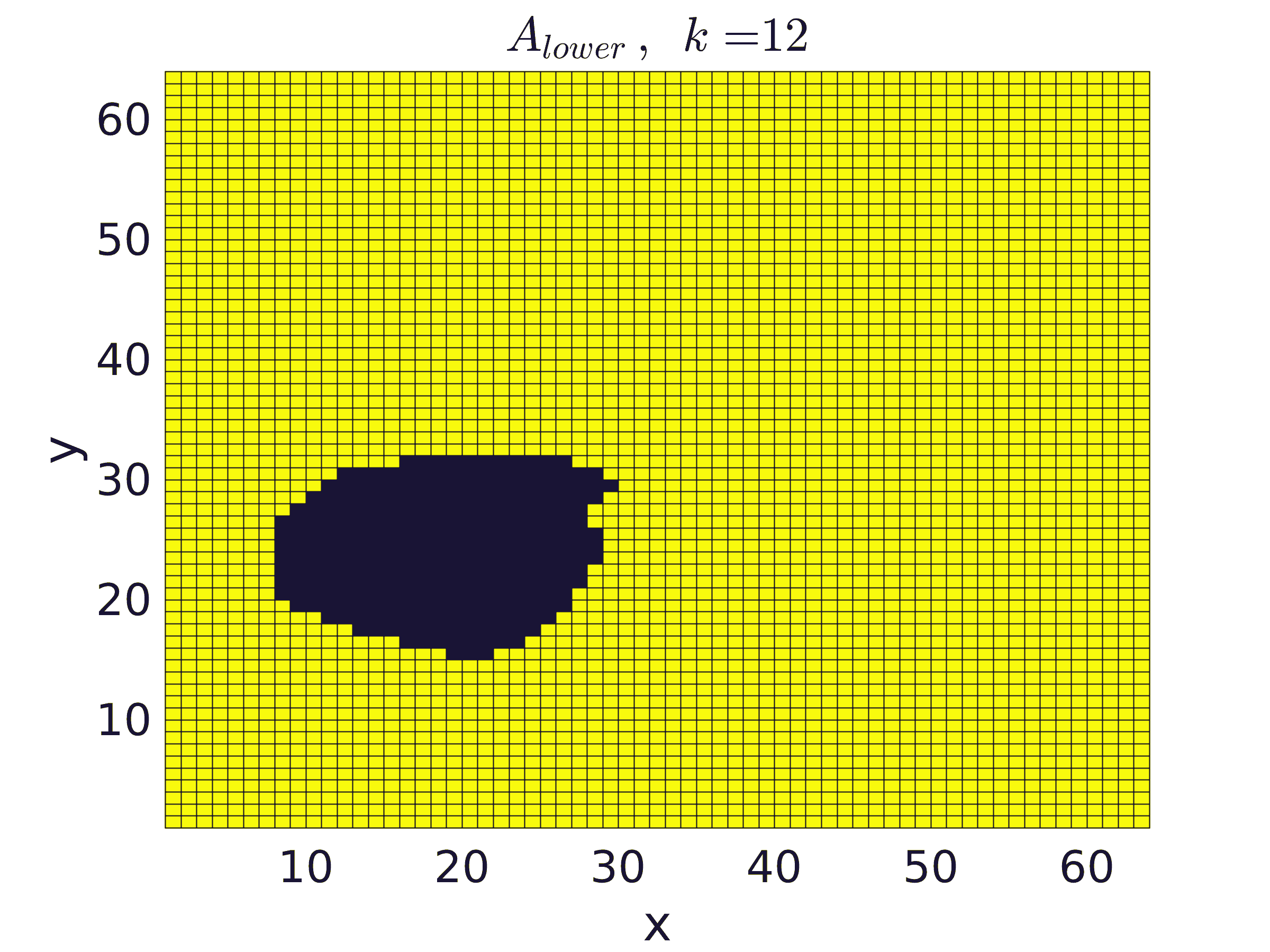}
\hspace*{0.01\textwidth}
\includegraphics[width=0.31\textwidth]{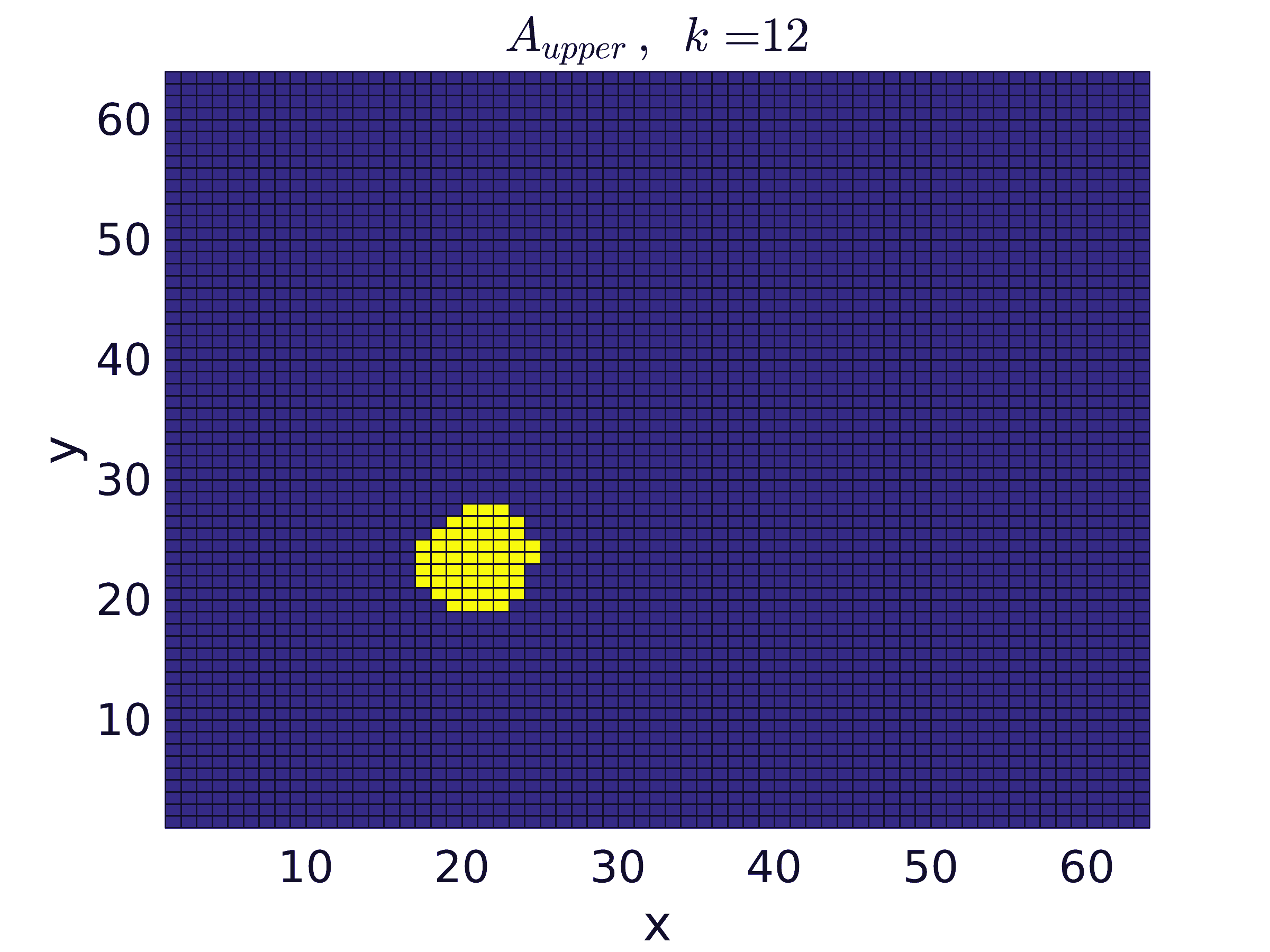}\\
\caption{
Left: reconstructed coefficient $c_k$;
Middle: active set for lower bound;
Right: active set for upper bound;
For $k=1,4,8,12$ (top to bottom) and $\delta=0.01$.
\label{fig:convdel001}}
\end{figure}

Next, we consider a fixed noise level of $\delta=0.001$ and illustrate
the gain in computational effort obtained by
using the active set from the previous Newton step as an initial guess
(warm start) as compared to starting each Newton step with an empty
active set $\mathcal A = \emptyset$ for the upper bound and a full active
set $\mathcal C = \mathcal N$ for the lower bound (cold start), see
Figure \ref{fig:warmcold}. We do so for the inverse potential problem from Section \ref{sec_cprob}. The CPU
times were 8.60 (351.20) seconds with warm start and 9.15 (525.99)
seconds with cold start to achieve an $L^1$ error of 0.0959 (0.0827)
for N=32 (N=64).
\begin{figure}[p]
\includegraphics[width=0.48\textwidth]{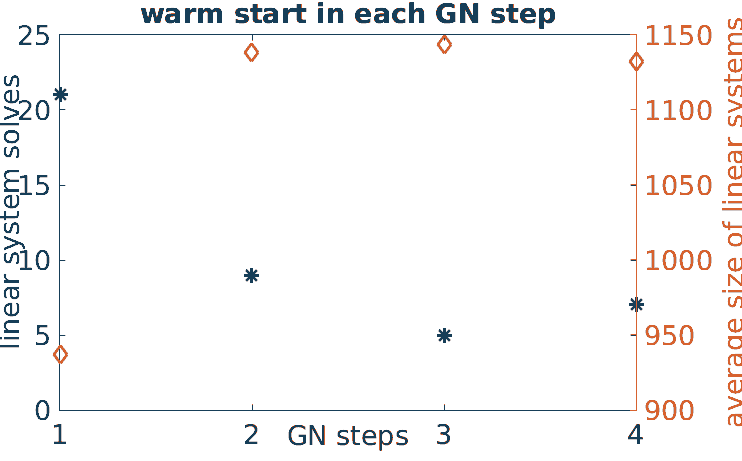}
\hspace*{0.01\textwidth}
\includegraphics[width=0.48\textwidth]{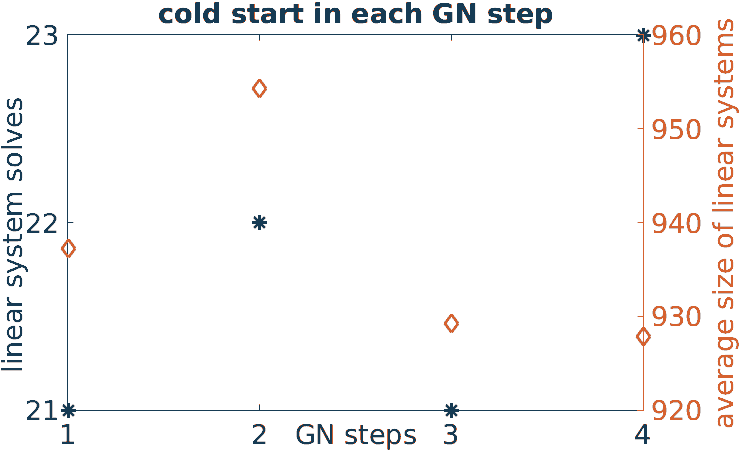}\\
\includegraphics[width=0.48\textwidth]{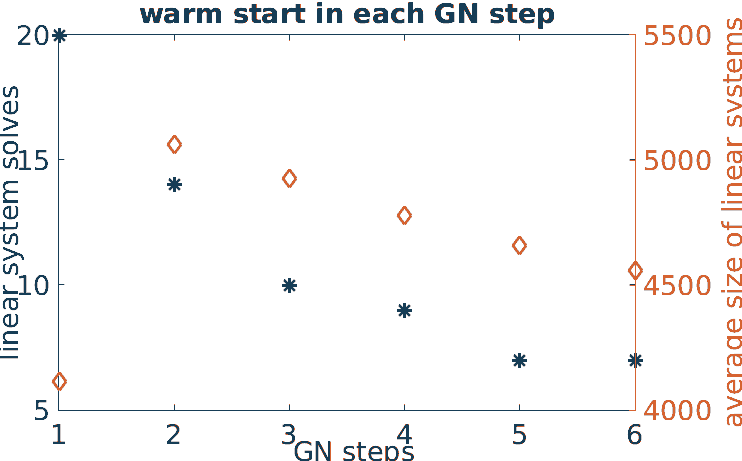}
\hspace*{0.01\textwidth}
\includegraphics[width=0.48\textwidth]{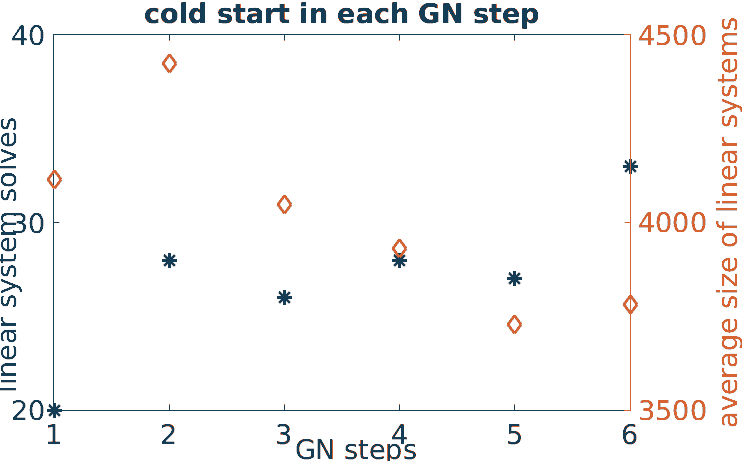}
\caption{
Linear system solves (stars) and average size of linear systems (diamonds);
Left: warm start;
Right: cold start;
Top: N=32;
Bottom: N=64.
\label{fig:warmcold}}
\end{figure}

Finally, in order to demonstrate the ability of the method to also deal with coefficients that would not allow for a well-defined parameter-to-state map, we consider the inverse potential problem from Section \ref{sec_cprob} with the test cases
\[
\begin{aligned}
&\mbox{test 2: }c_{ex}(x,y)=1-10\cdot {1\!\!{\rm I}}_{B_1}+5\cdot {1\!\!{\rm I}}_{B_2}
&& \ul{c}=-9,&& \ol{c}=6\,,\\
&\mbox{test 3: }c_{ex}(x,y)=-10\cdot {1\!\!{\rm I}}_{B_1}-5\cdot {1\!\!{\rm I}}_{B_2}
&& \ul{c}=-10,&& \ol{c}=0\,,
\end{aligned}
\]
where $B_1= B_{0.2}(-0.4,-0.3)$, $B_2= B_{0.1}(0.5,0.5)$,
see Figures \ref{fig:exsol_spots_neg}, \ref{fig:convdel001_neg}, \ref{fig:exsol_spots_nneg}, \ref{fig:convdel001_nneg}.

\begin{figure}[p]
\includegraphics[width=0.48\textwidth]{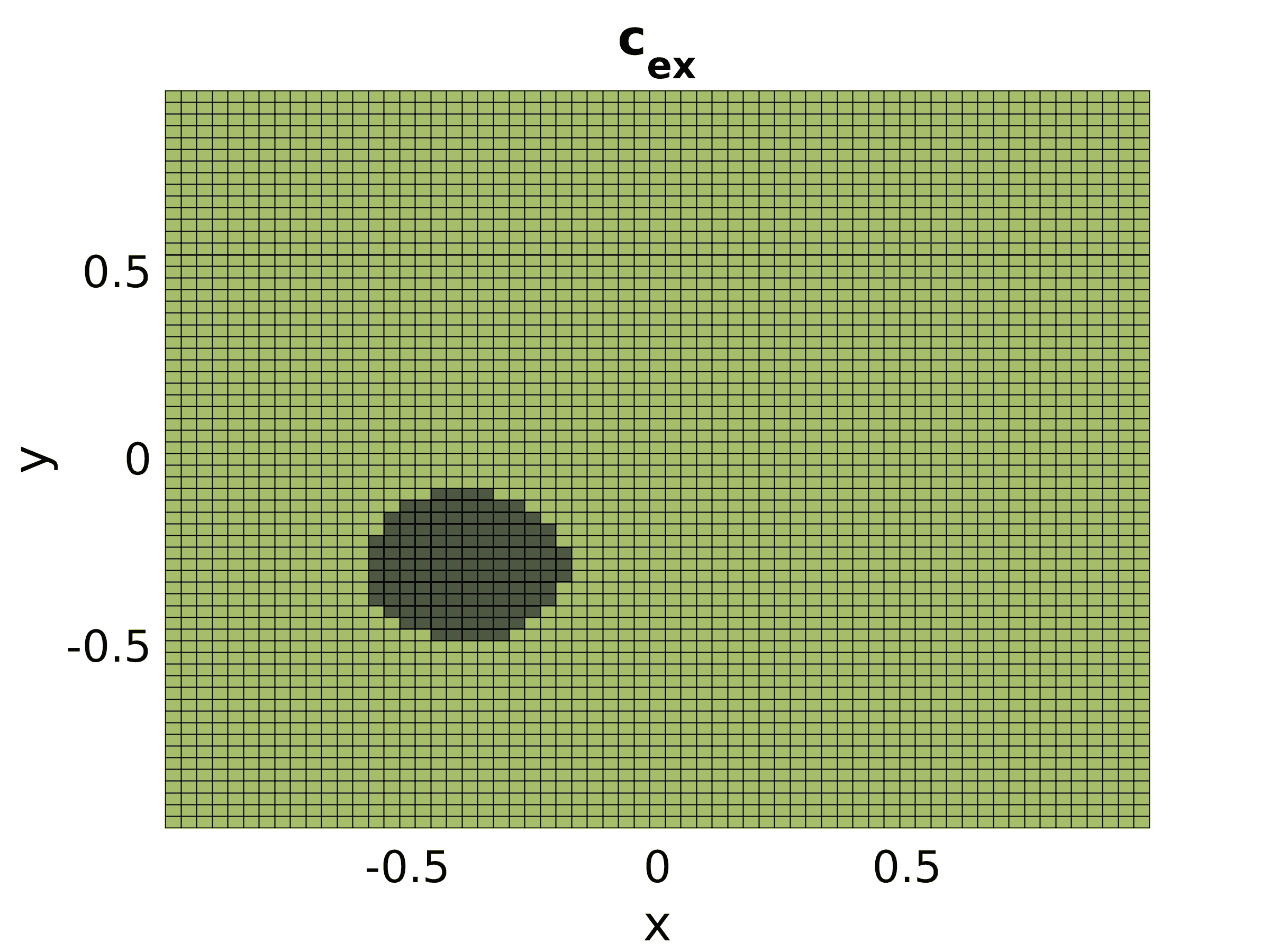}
\hspace*{0.01\textwidth}
\includegraphics[width=0.48\textwidth]{markspots64.png}
\caption{
Test 2: left: exact coefficient $c_{ex}$; $\ul{c}=-9$, $\ol{c}=6$;
right: locations of spots for testing weak * $L^\infty$ convergence
\label{fig:exsol_spots_neg}}
\end{figure}

\begin{figure}[p]
\includegraphics[width=0.31\textwidth]{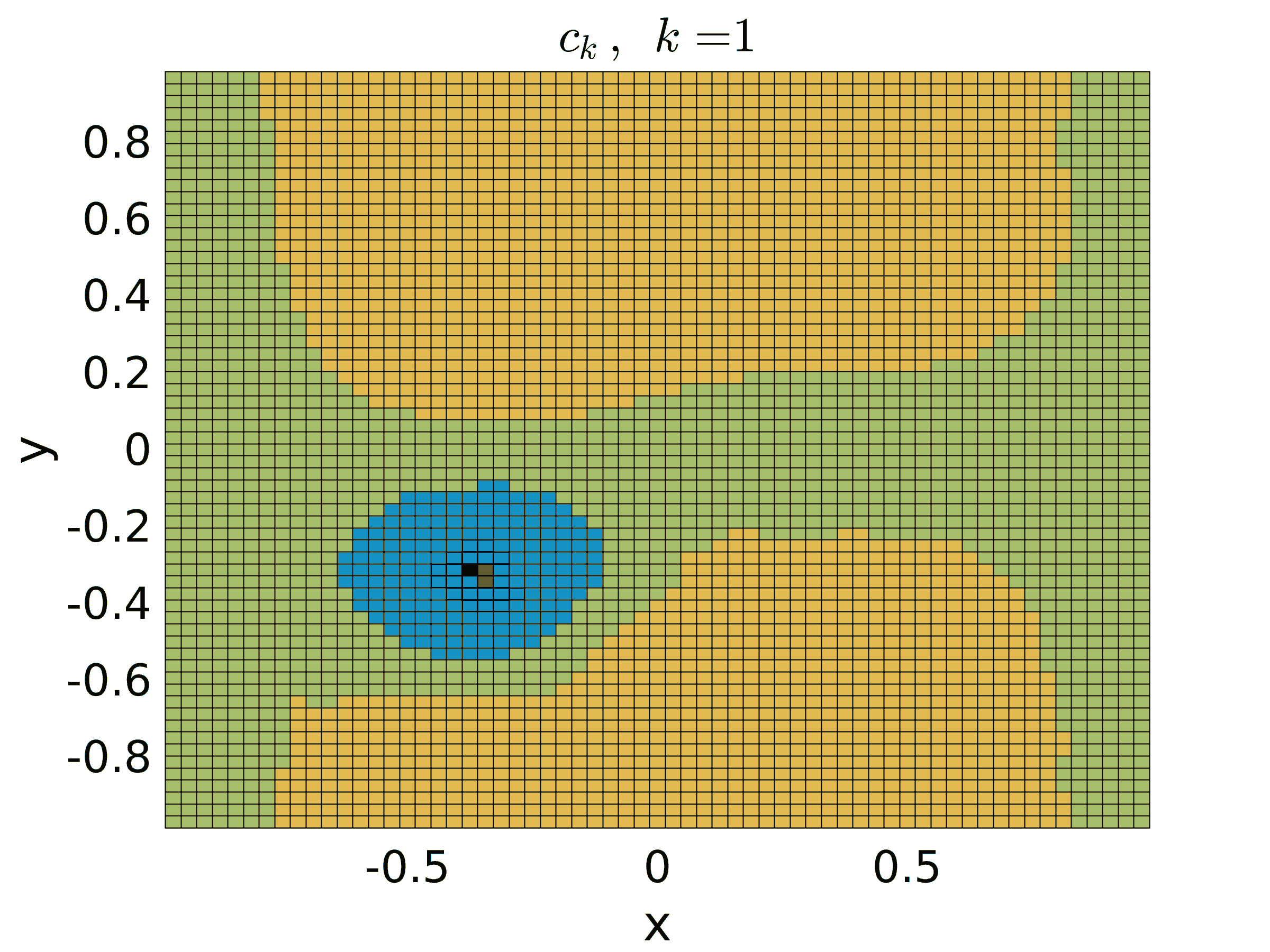}
\hspace*{0.01\textwidth}
\includegraphics[width=0.31\textwidth]{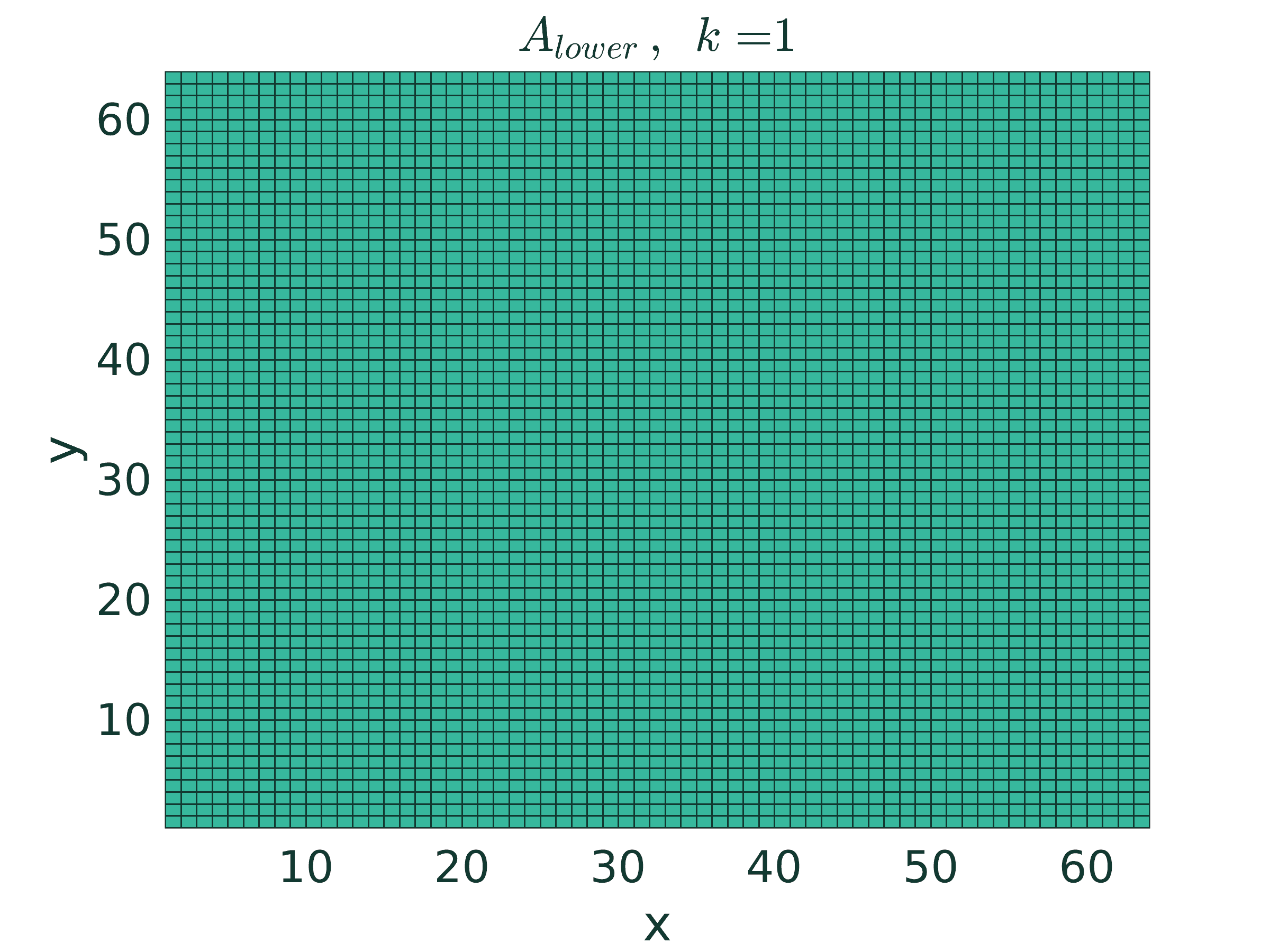}
\hspace*{0.01\textwidth}
\includegraphics[width=0.31\textwidth]{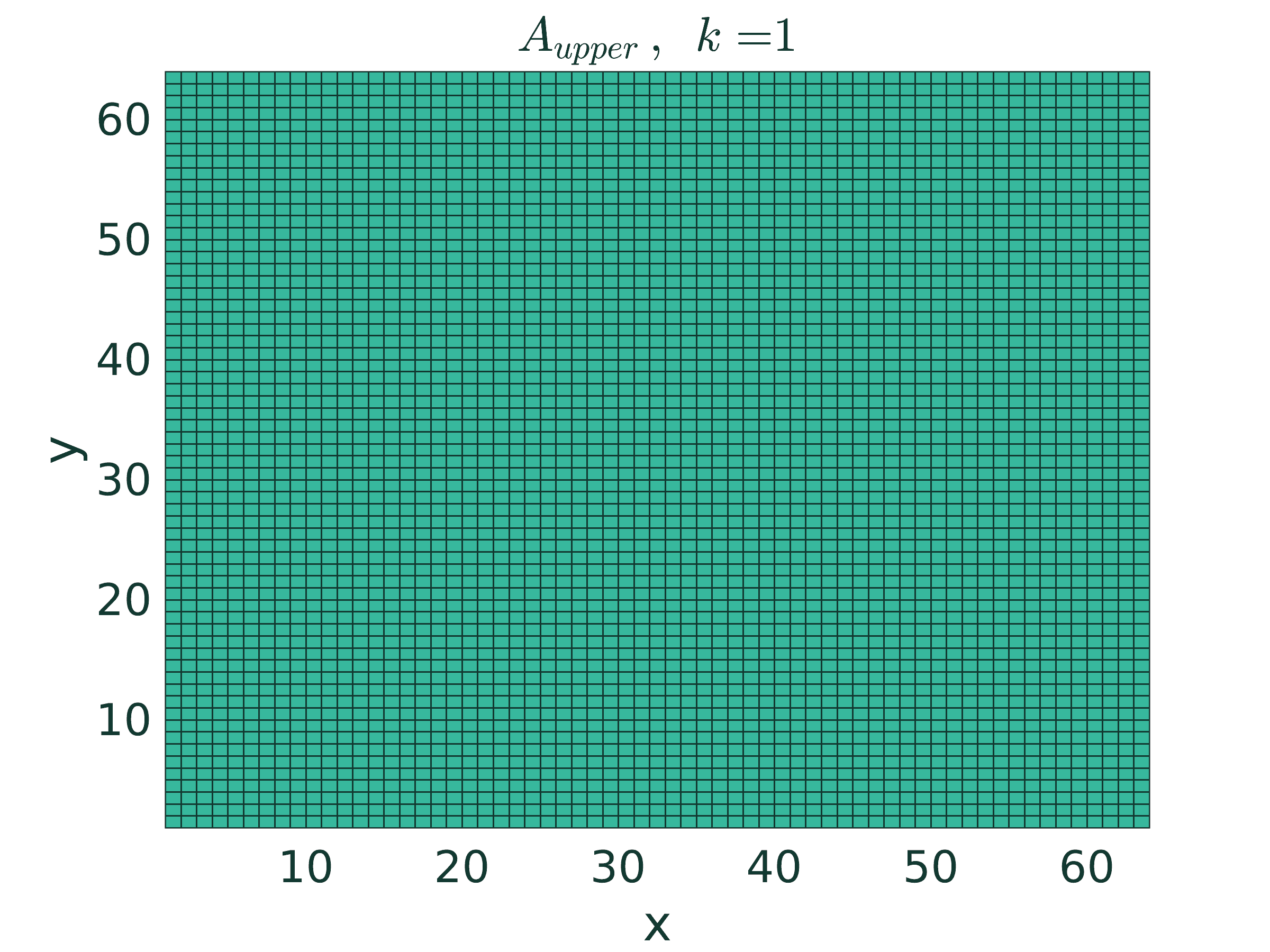}\\
\includegraphics[width=0.31\textwidth]{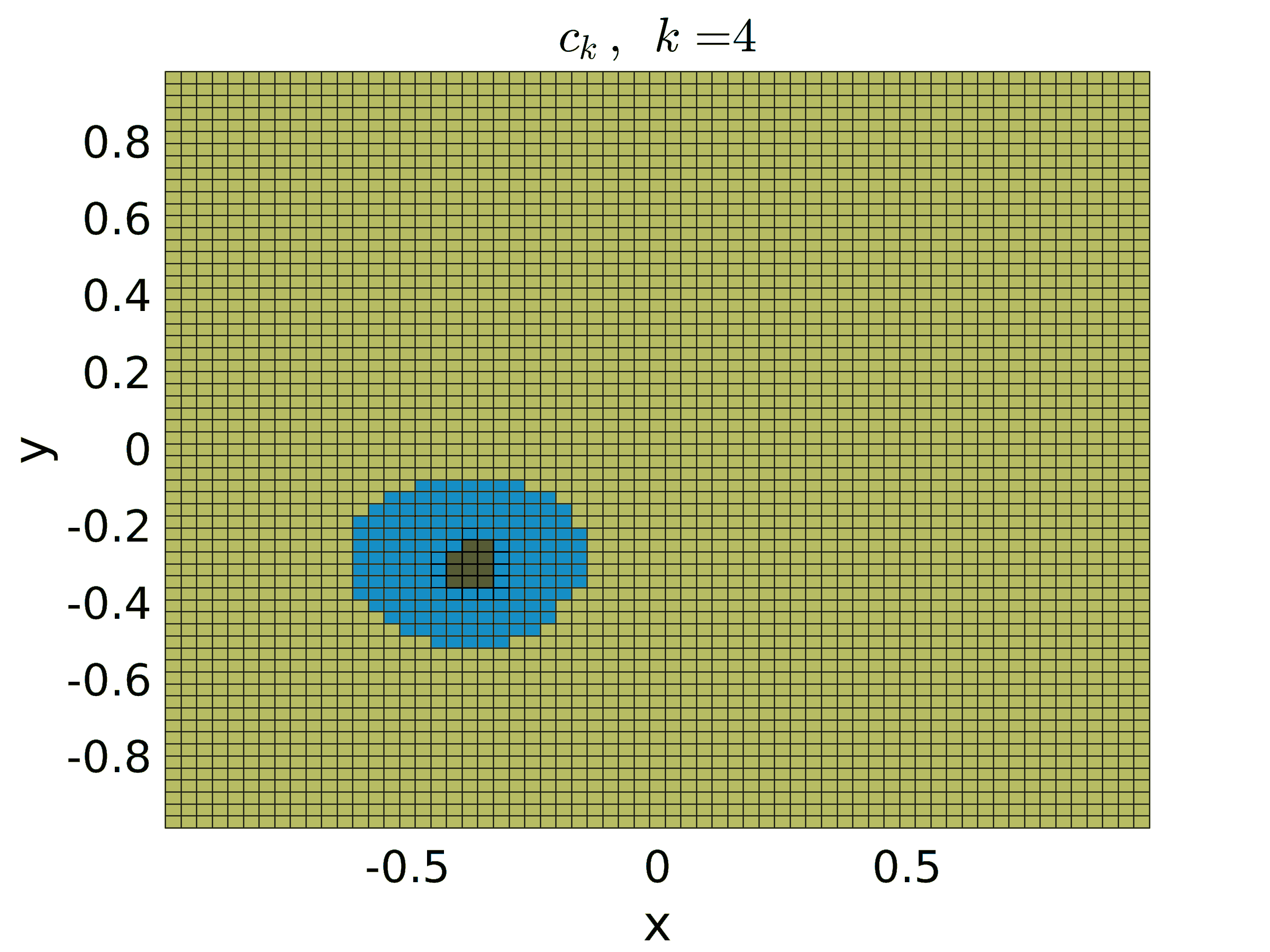}
\hspace*{0.01\textwidth}
\includegraphics[width=0.31\textwidth]{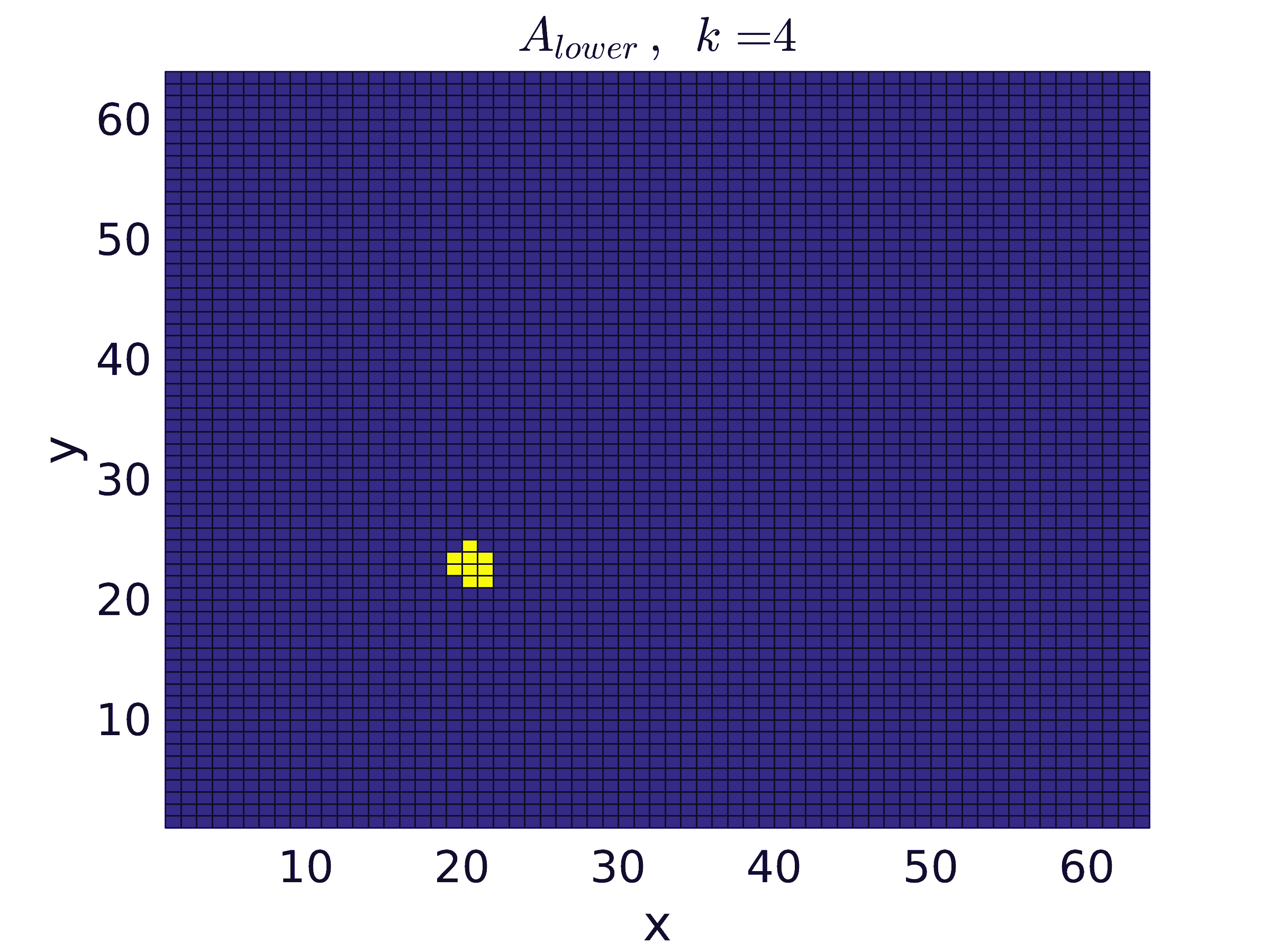}
\hspace*{0.01\textwidth}
\includegraphics[width=0.31\textwidth]{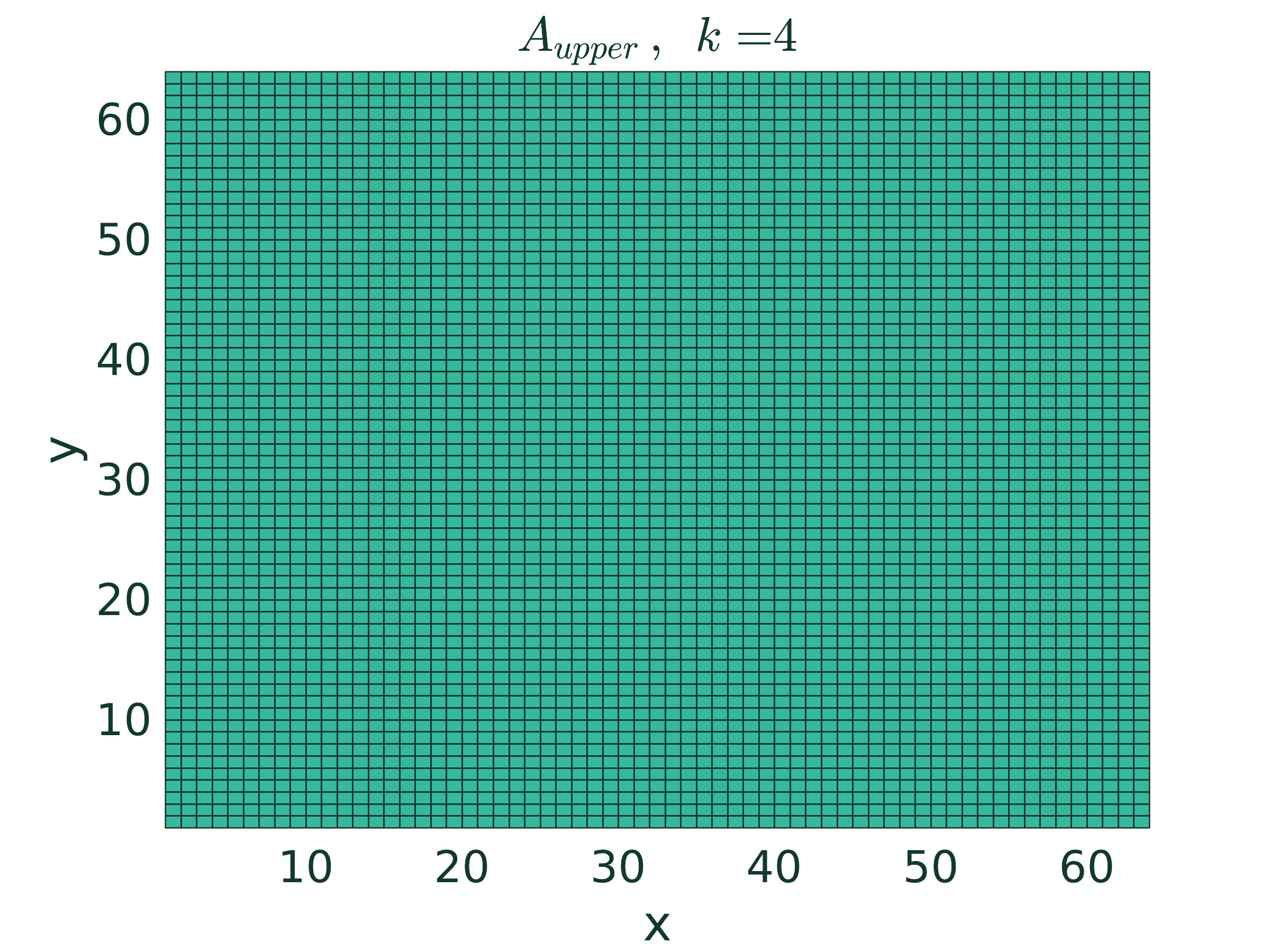}\\
\includegraphics[width=0.31\textwidth]{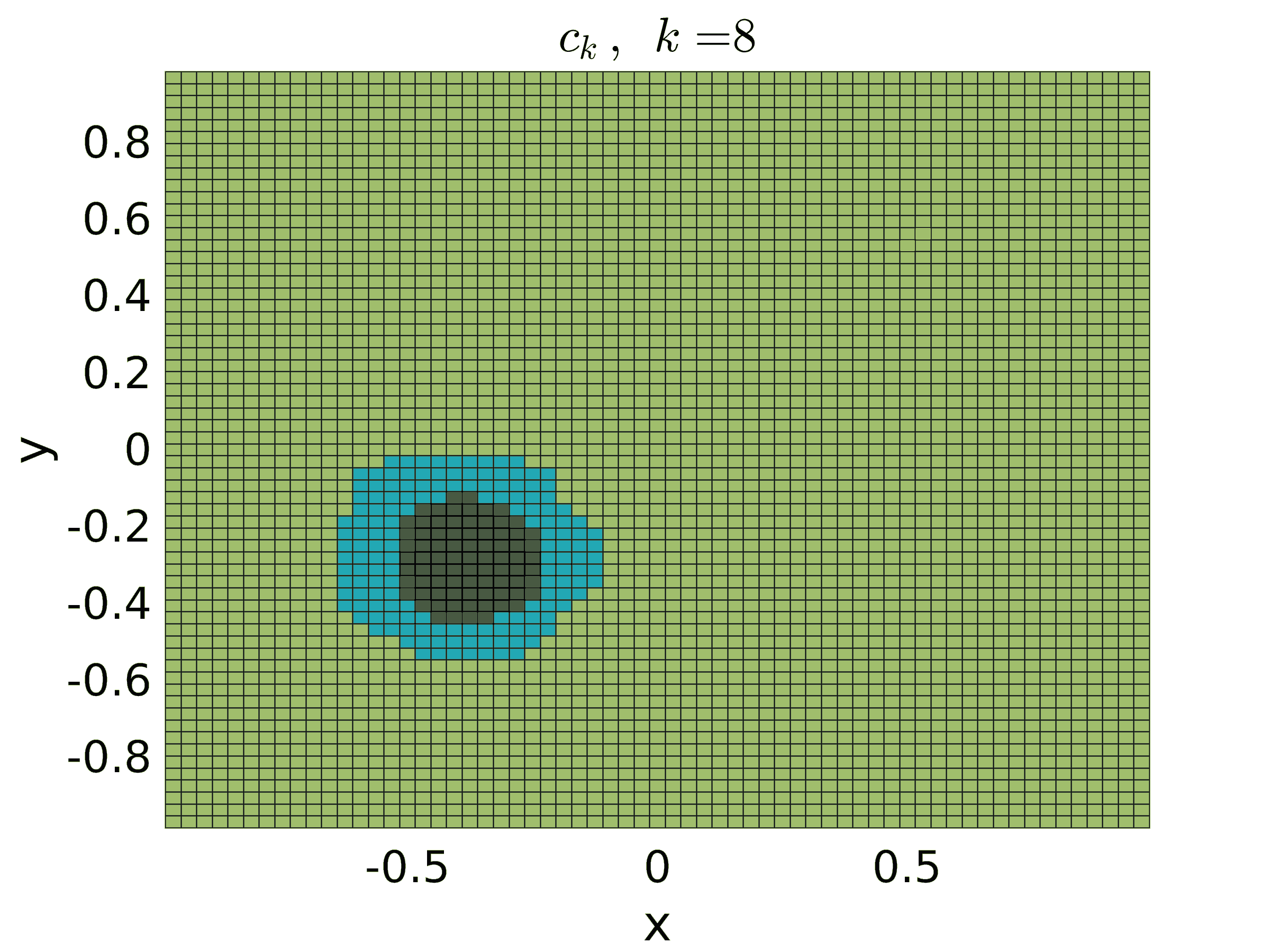}
\hspace*{0.01\textwidth}
\includegraphics[width=0.31\textwidth]{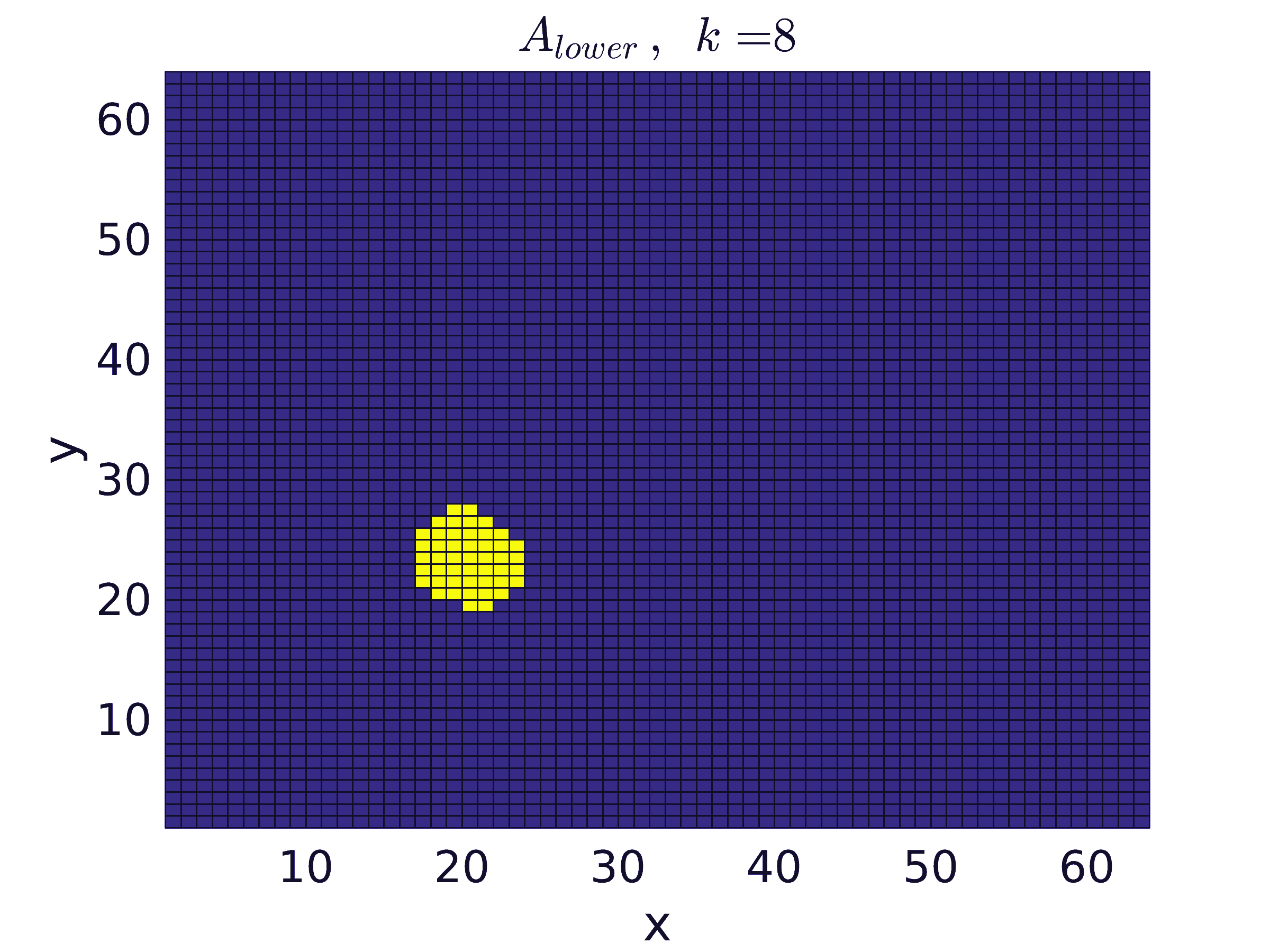}
\hspace*{0.01\textwidth}
\includegraphics[width=0.31\textwidth]{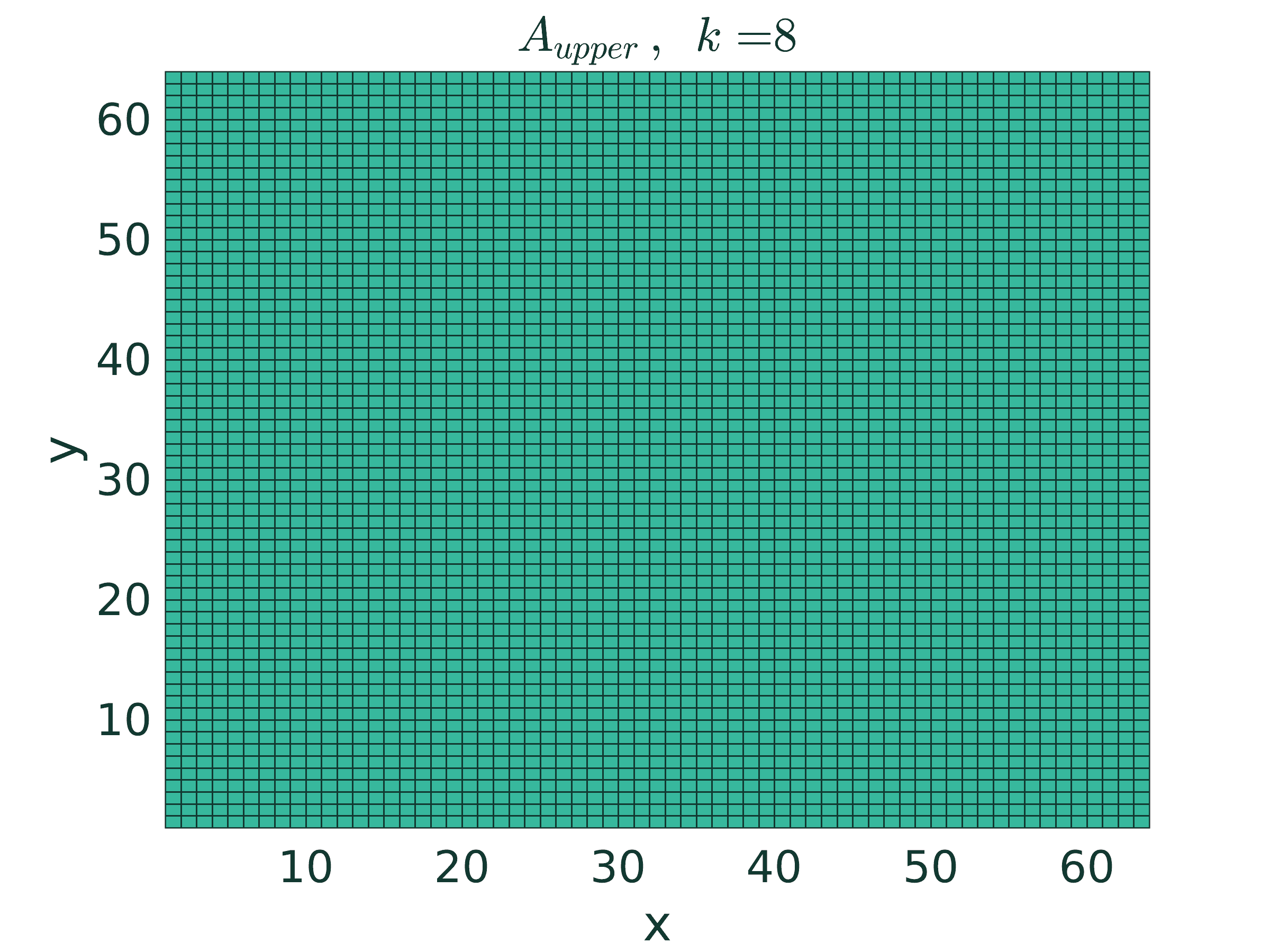}\\
\includegraphics[width=0.31\textwidth]{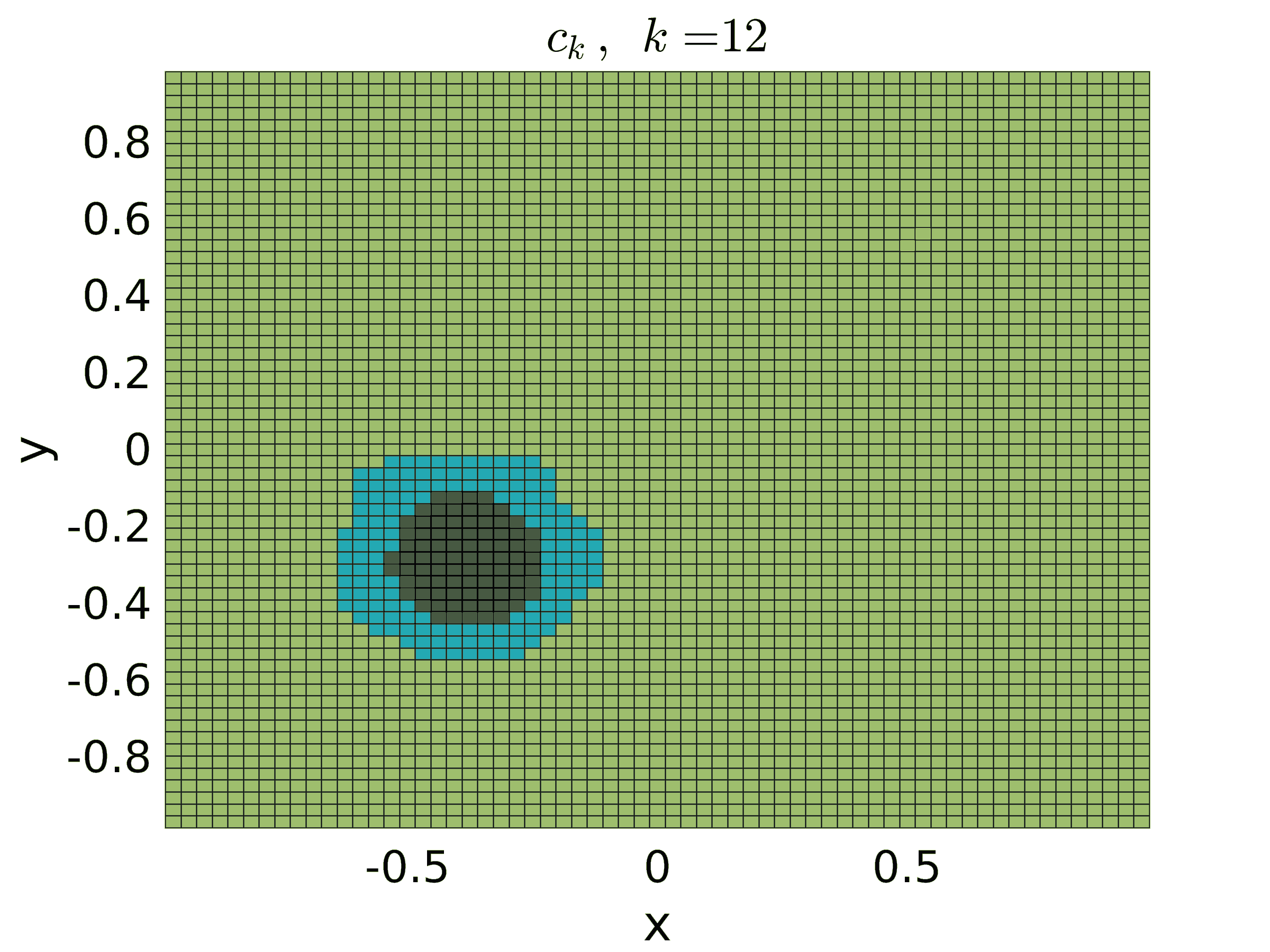}
\hspace*{0.01\textwidth}
\includegraphics[width=0.31\textwidth]{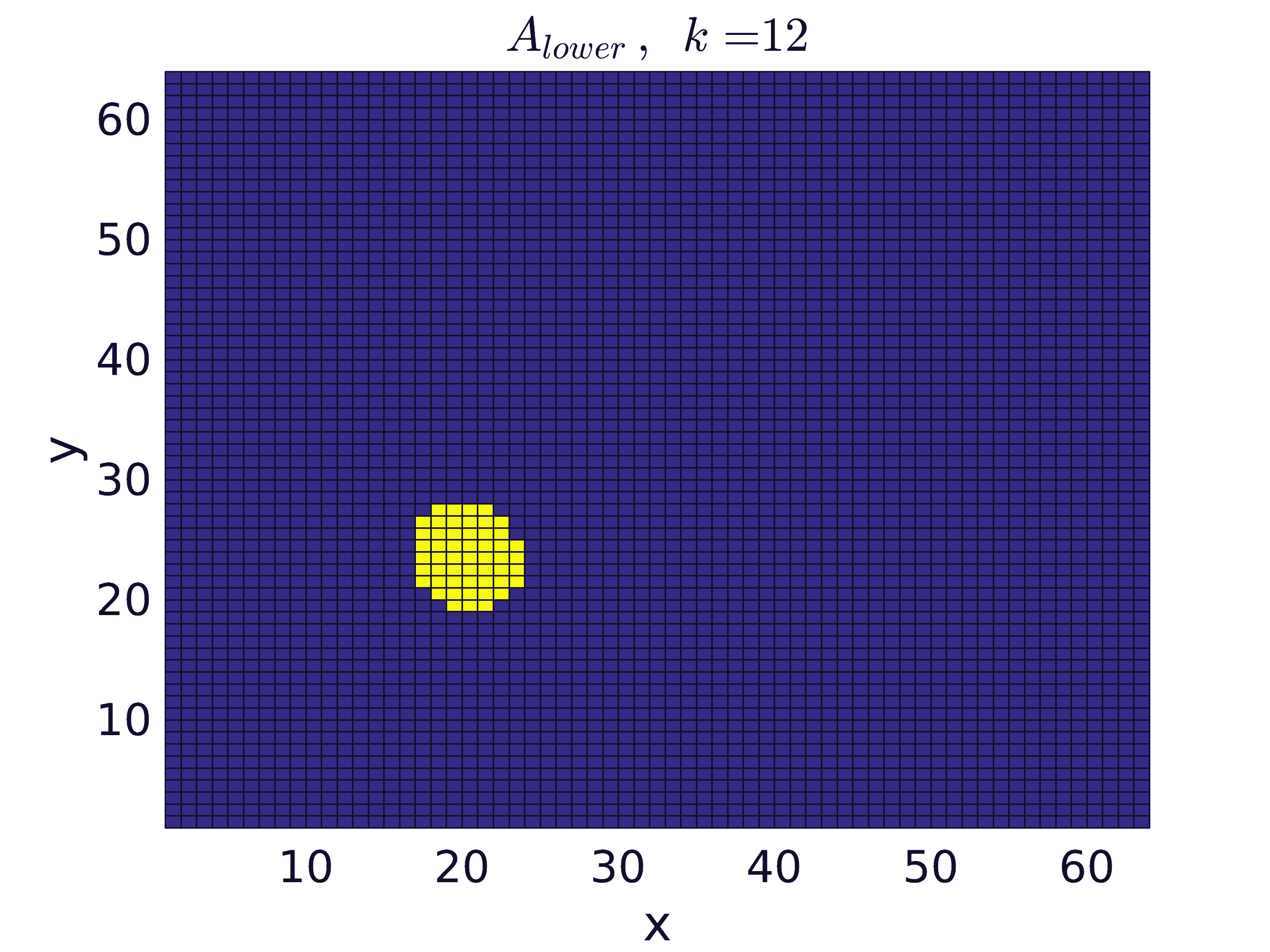}
\hspace*{0.01\textwidth}
\includegraphics[width=0.31\textwidth]{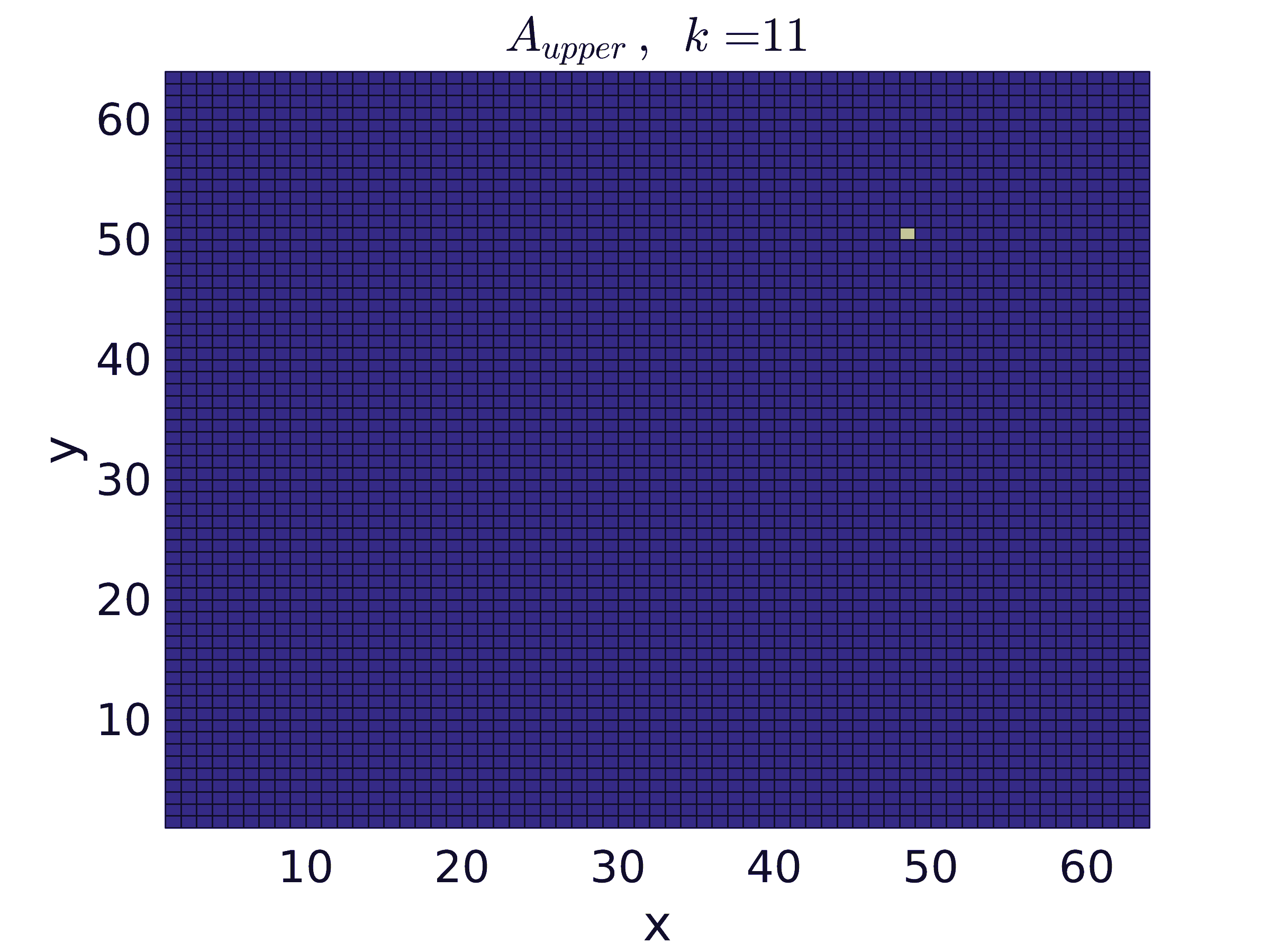}\\
\caption{
Test 2: Left: reconstructed coefficient $c_k$;
Middle: active set for lower bound;
Right: active set for upper bound;
For $k=1,4,8,12$ (top to bottom) and $\delta=0.01$.
\label{fig:convdel001_neg}}
\end{figure}

\begin{figure}[p]
\includegraphics[width=0.48\textwidth]{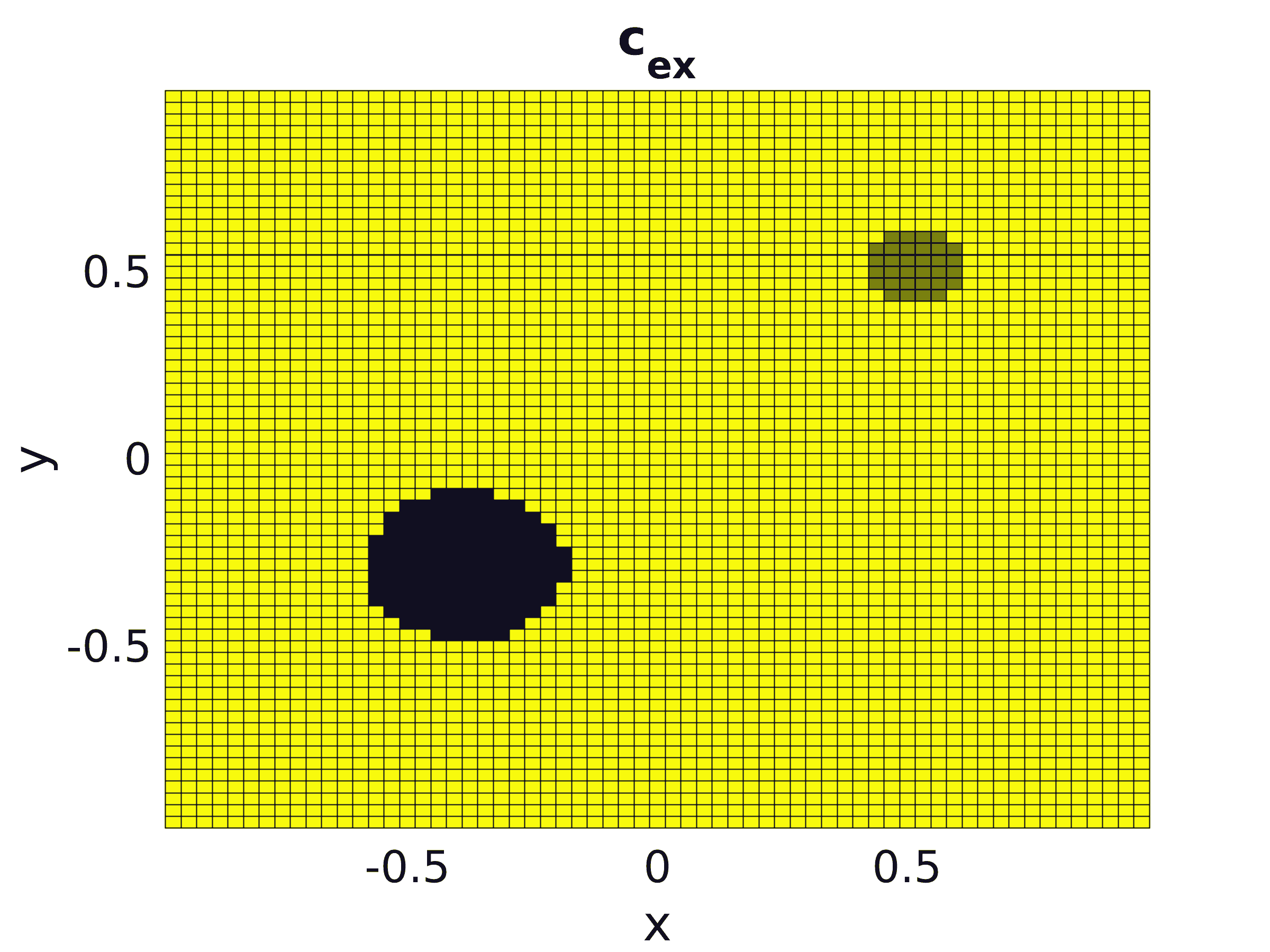}
\hspace*{0.01\textwidth}
\includegraphics[width=0.48\textwidth]{markspots64.png}
\caption{
Test 3: left: exact coefficient $c_{ex}$; $\ul{c}=-10$, $\ol{c}=0$;
right: locations of spots for testing weak * $L^\infty$ convergence
\label{fig:exsol_spots_nneg}}
\end{figure}

\begin{figure}[p]
\includegraphics[width=0.31\textwidth]{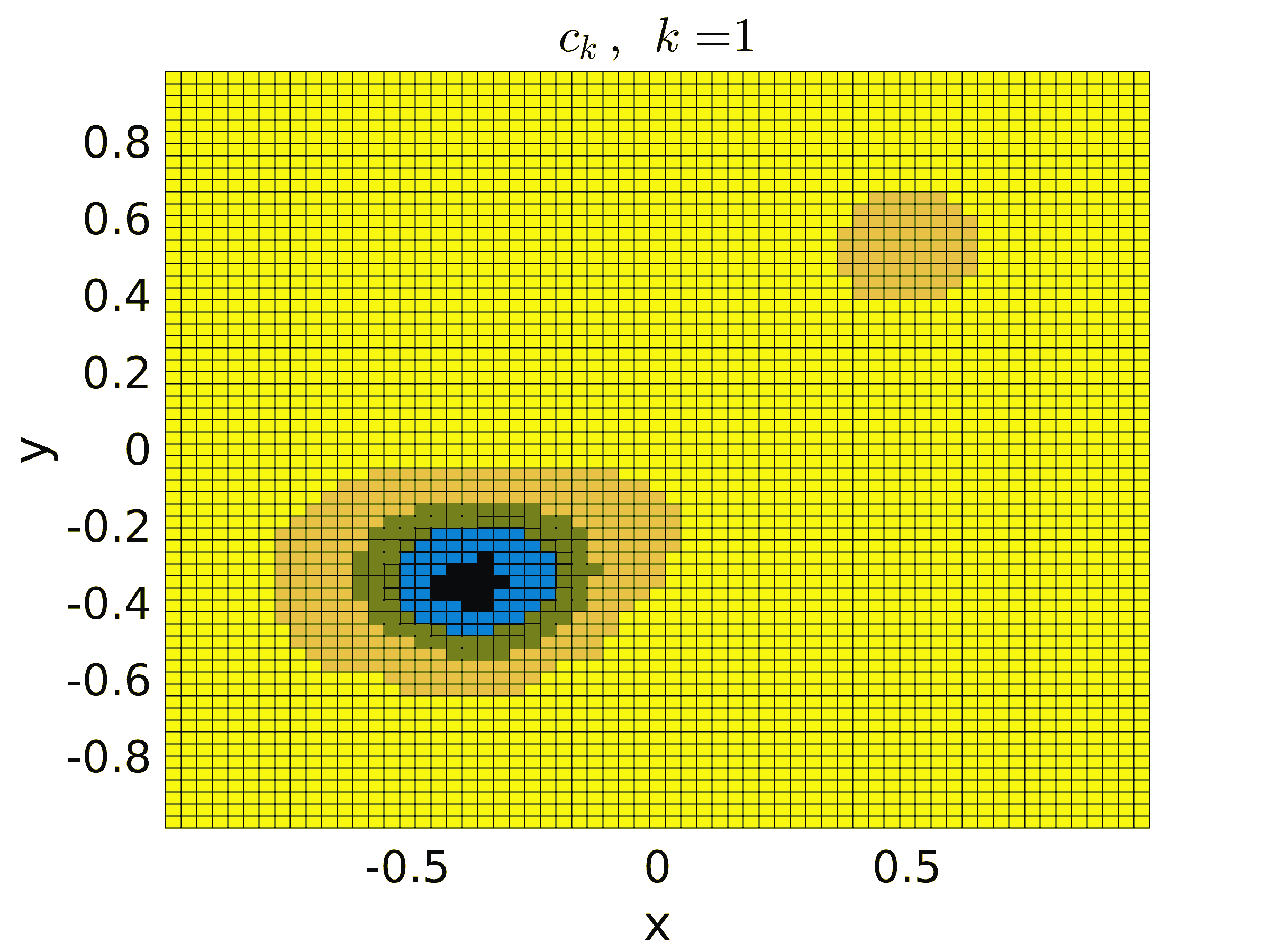}
\hspace*{0.01\textwidth}
\includegraphics[width=0.31\textwidth]{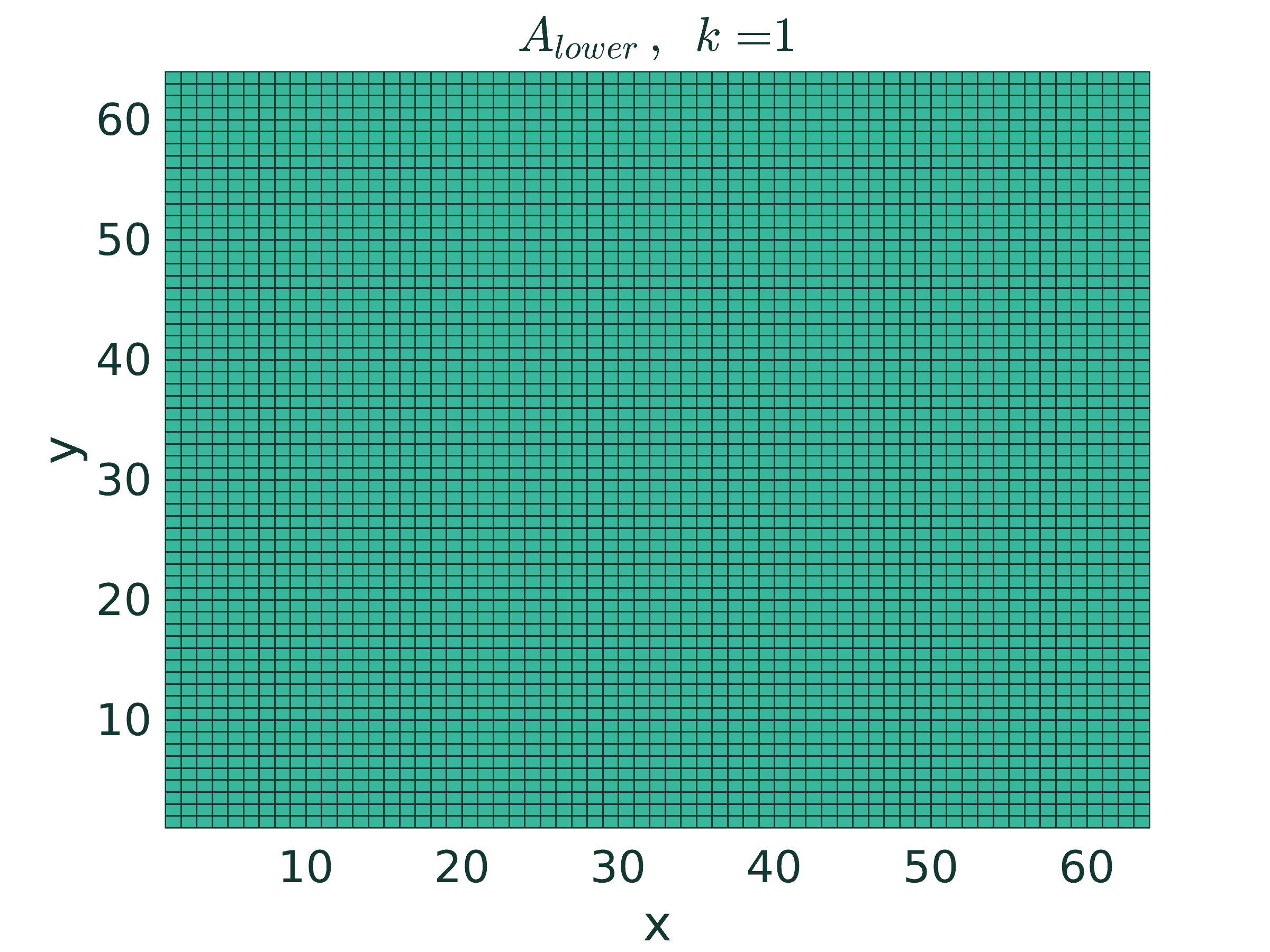}
\hspace*{0.01\textwidth}
\includegraphics[width=0.31\textwidth]{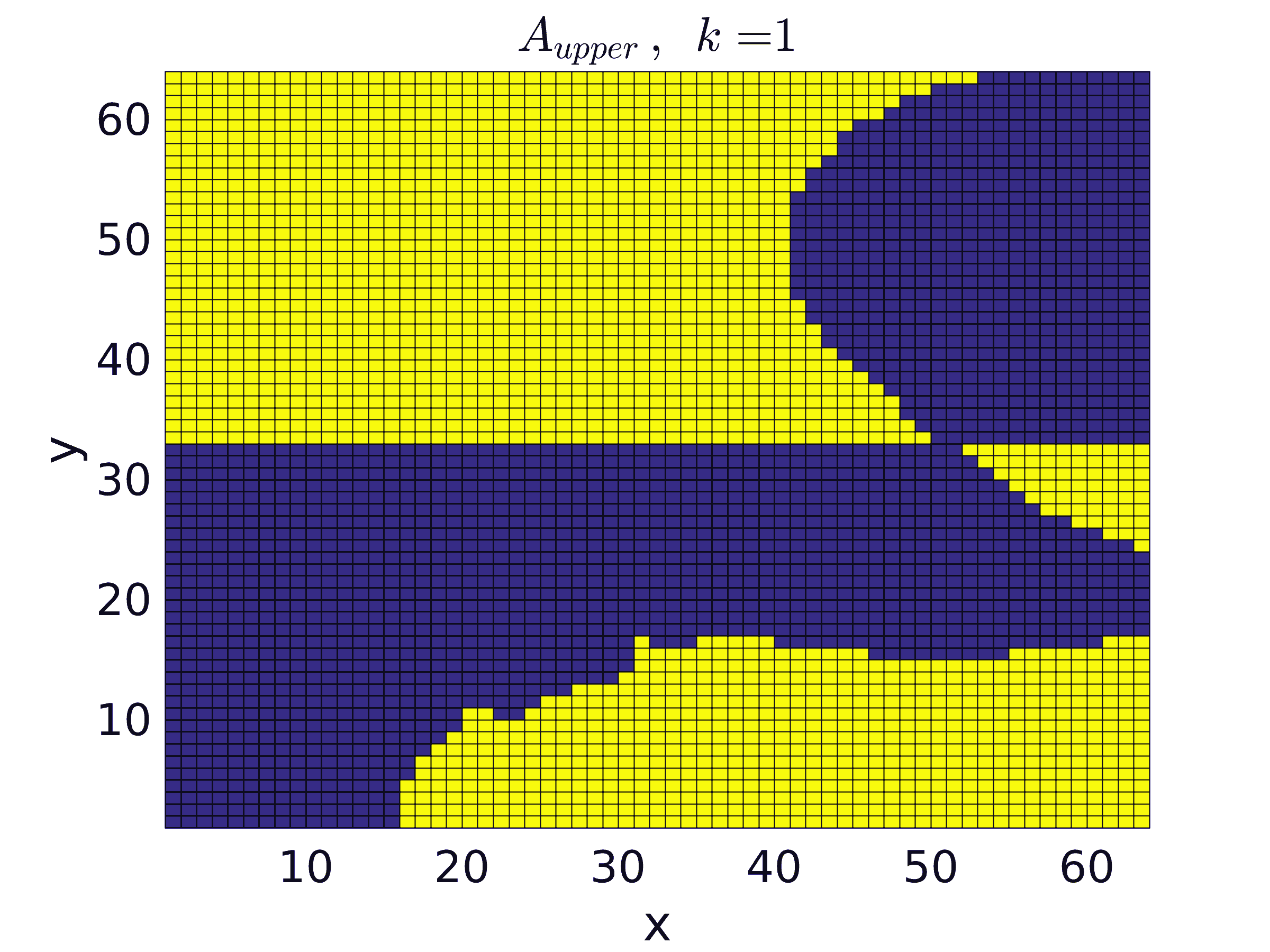}\\
\includegraphics[width=0.31\textwidth]{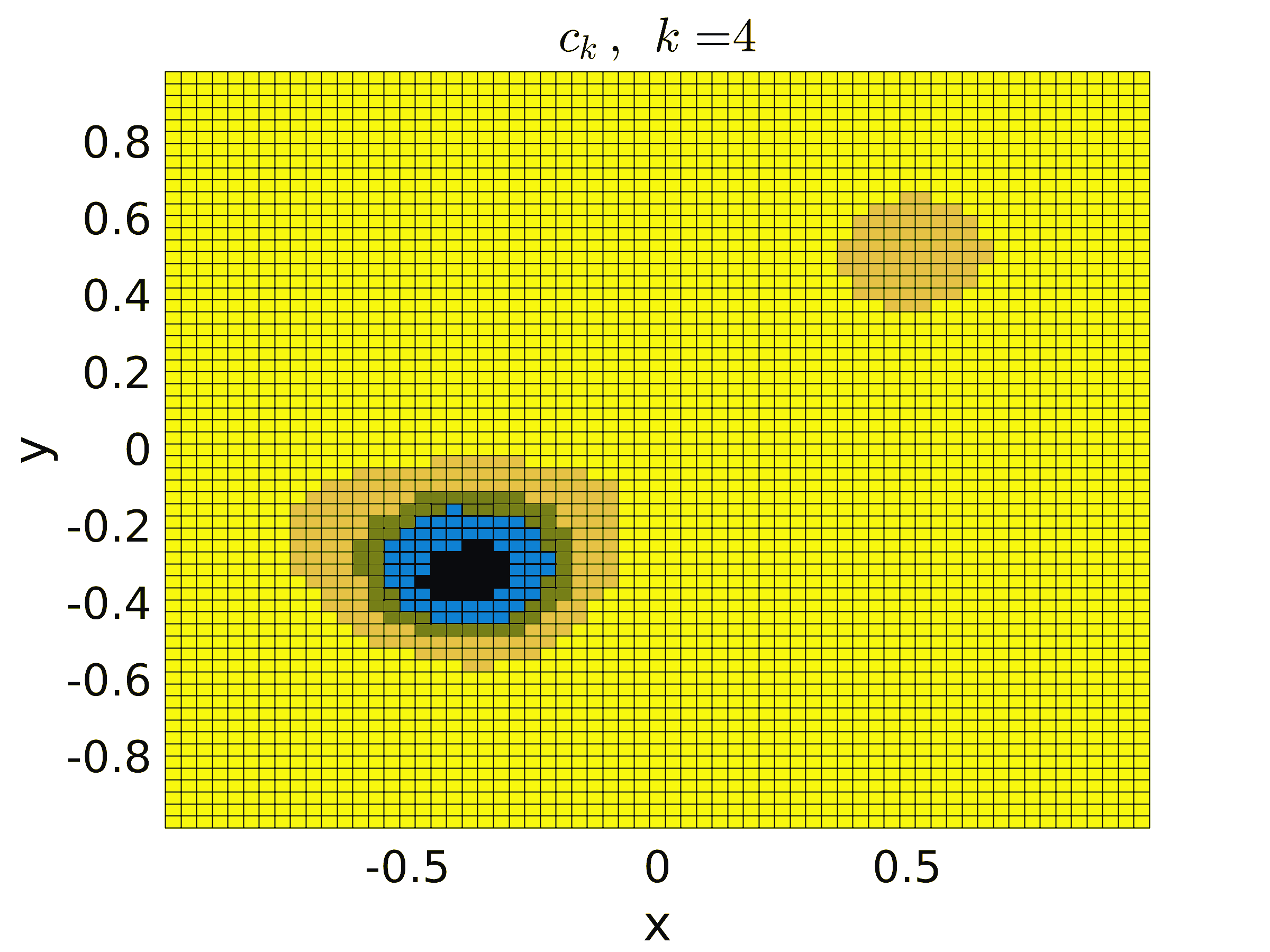}
\hspace*{0.01\textwidth}
\includegraphics[width=0.31\textwidth]{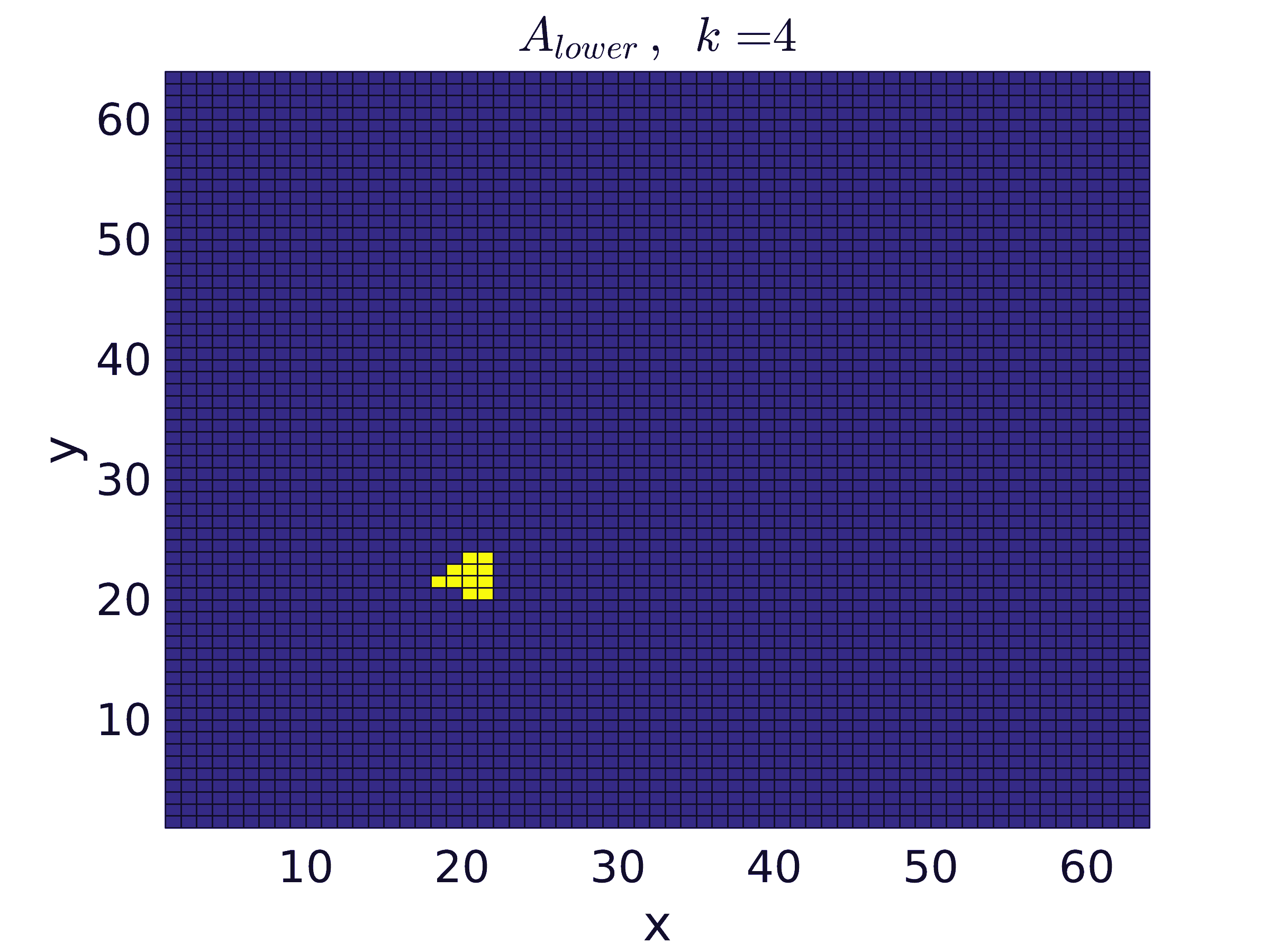}
\hspace*{0.01\textwidth}
\includegraphics[width=0.31\textwidth]{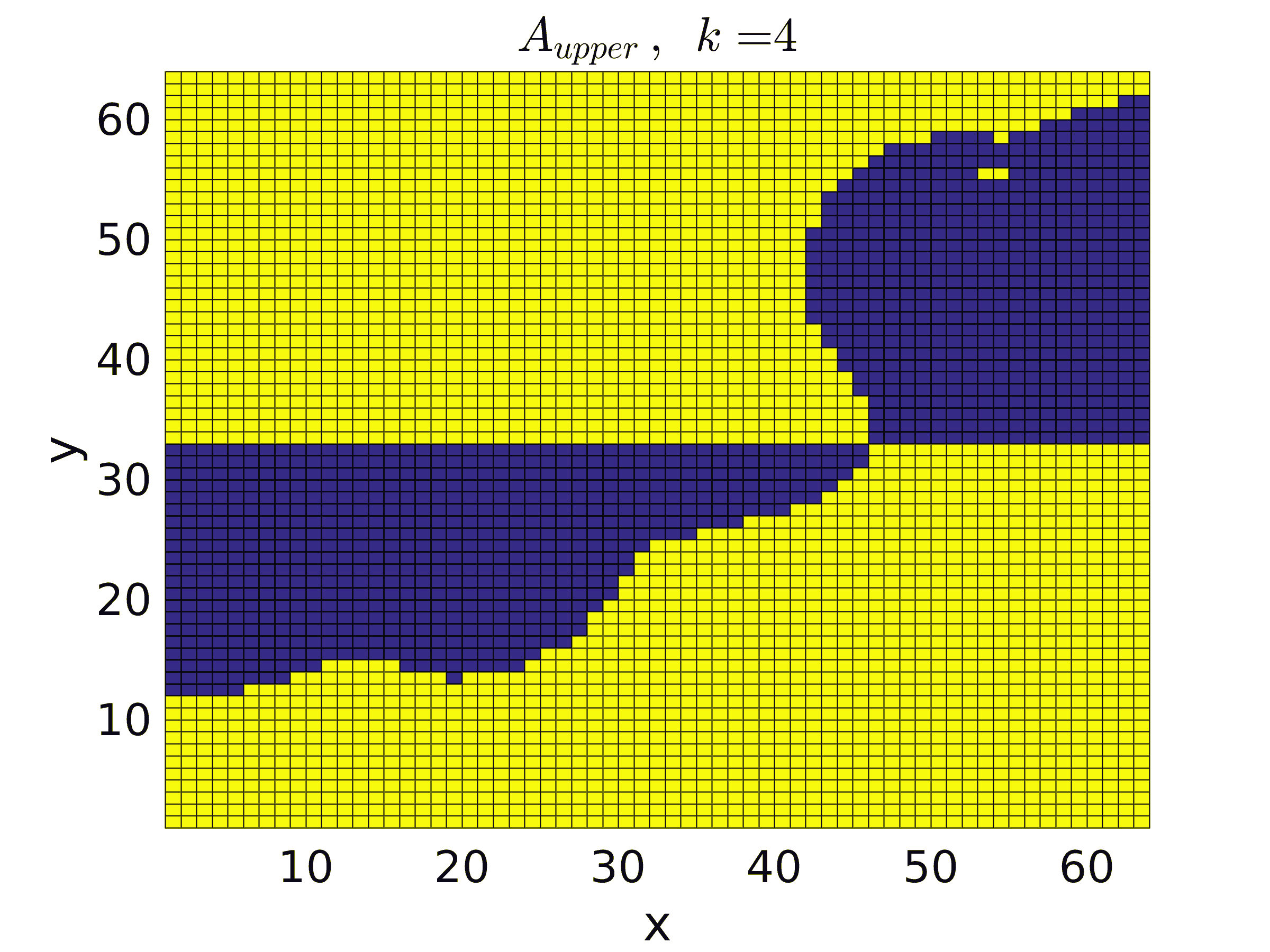}\\
\includegraphics[width=0.31\textwidth]{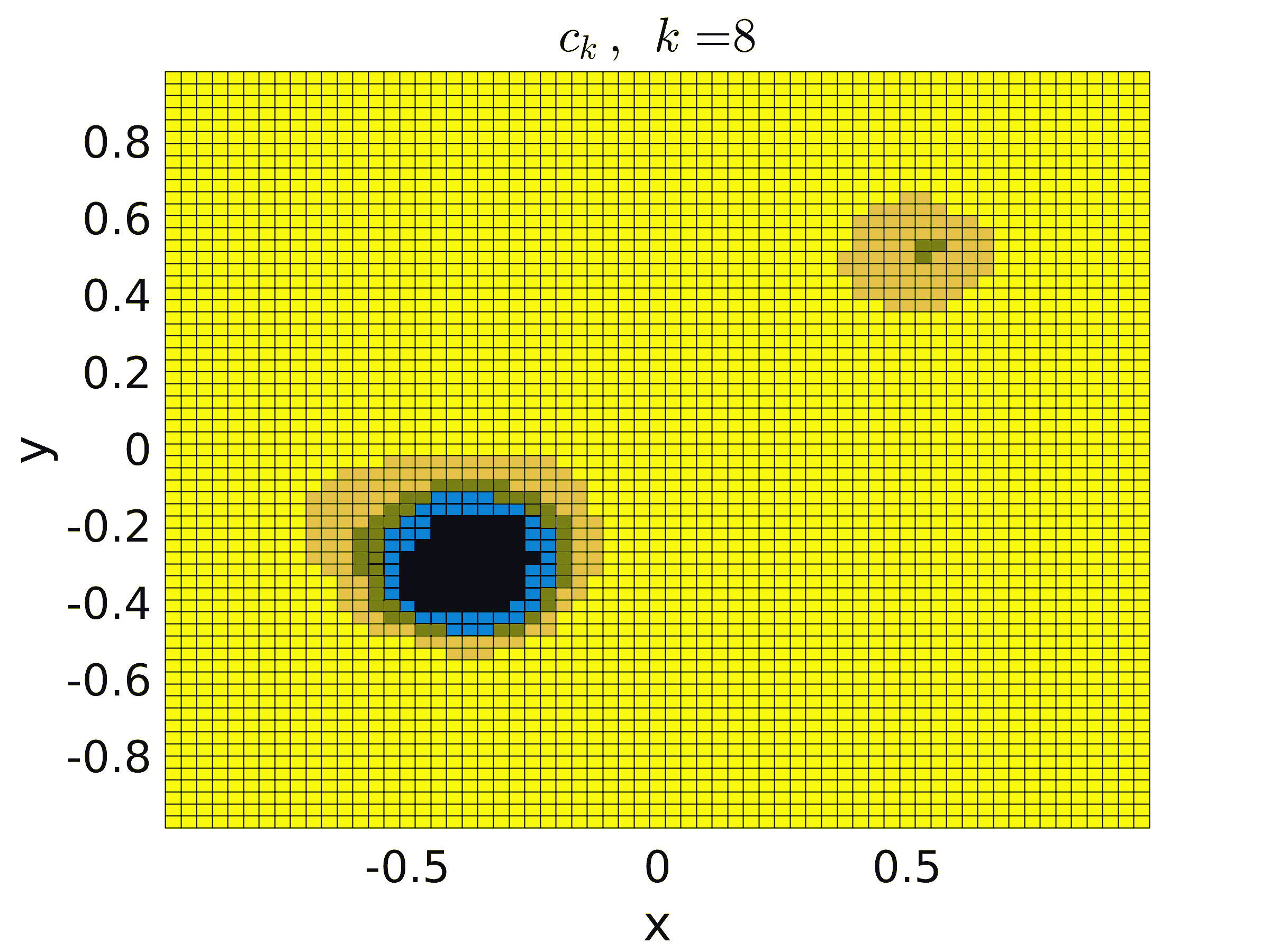}
\hspace*{0.01\textwidth}
\includegraphics[width=0.31\textwidth]{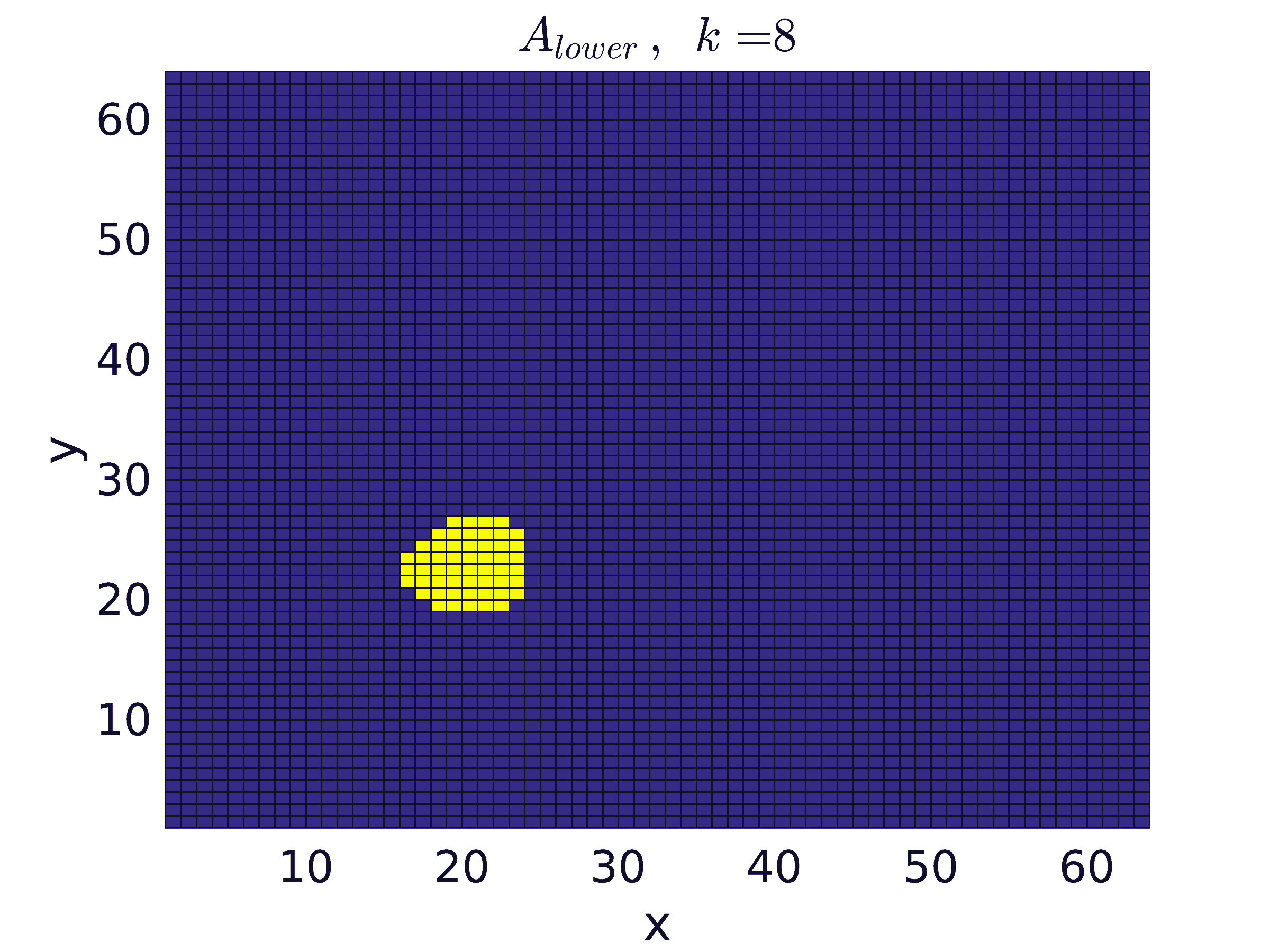}
\hspace*{0.01\textwidth}
\includegraphics[width=0.31\textwidth]{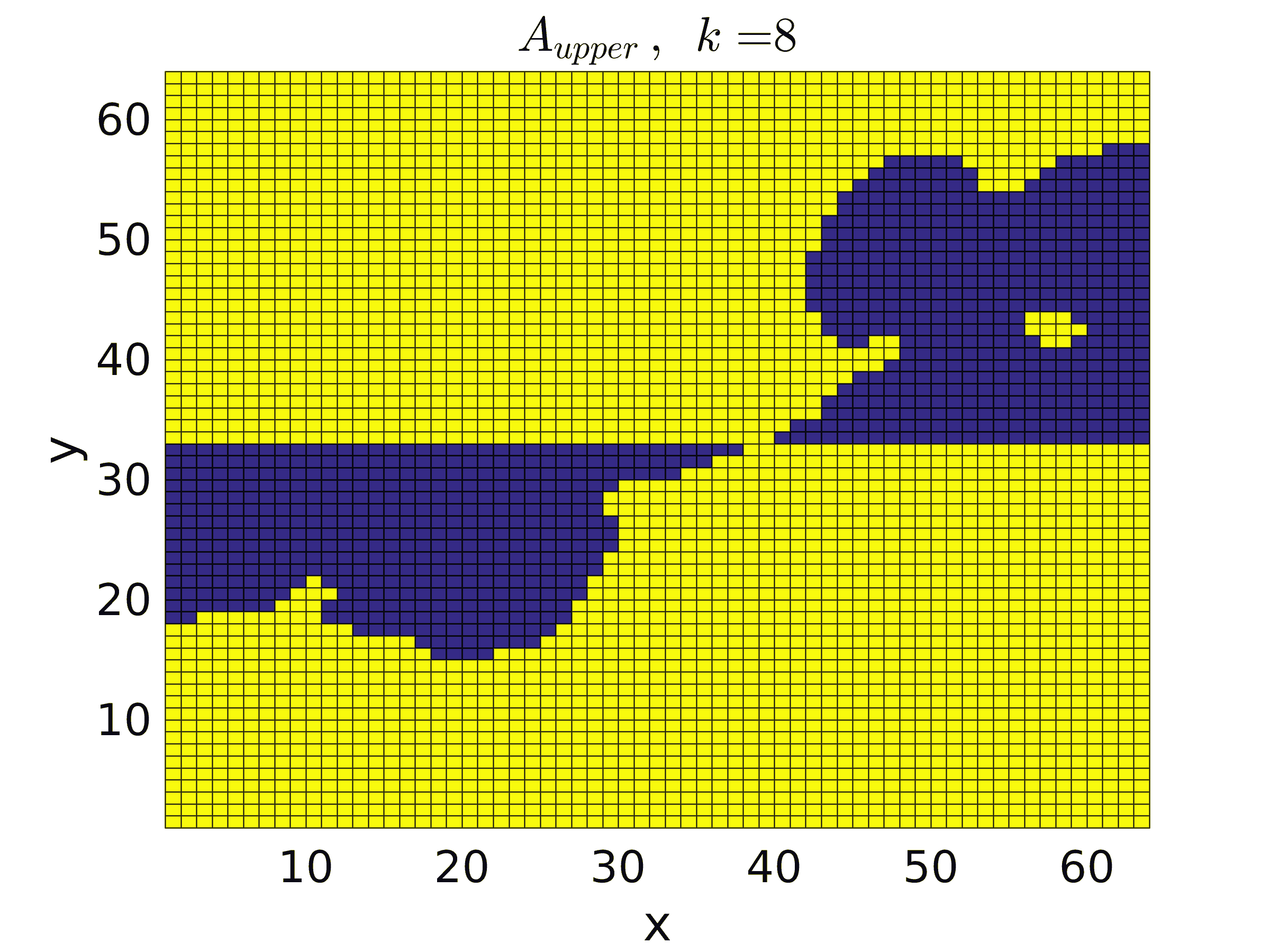}\\
\includegraphics[width=0.31\textwidth]{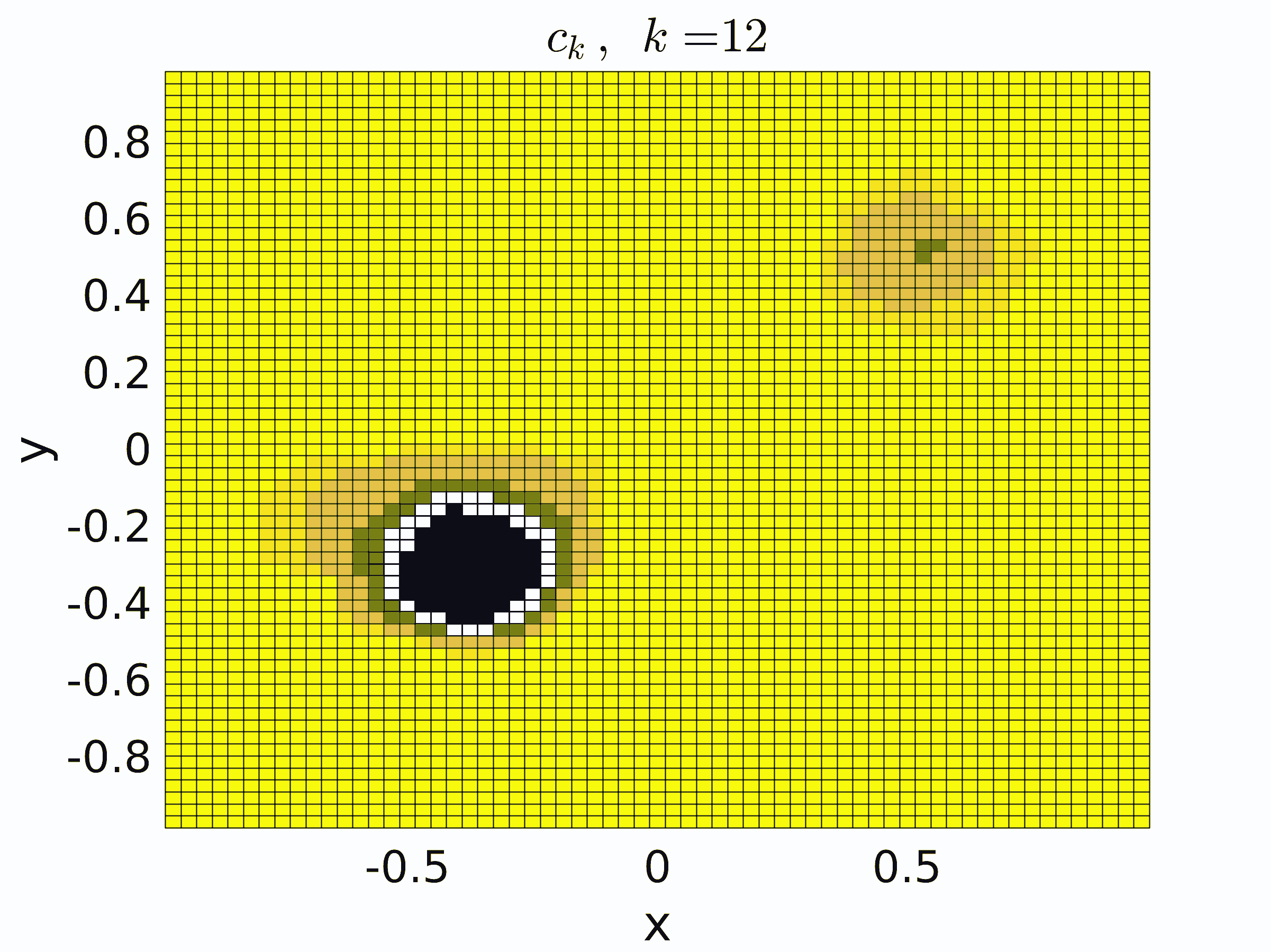}
\hspace*{0.01\textwidth}
\includegraphics[width=0.31\textwidth]{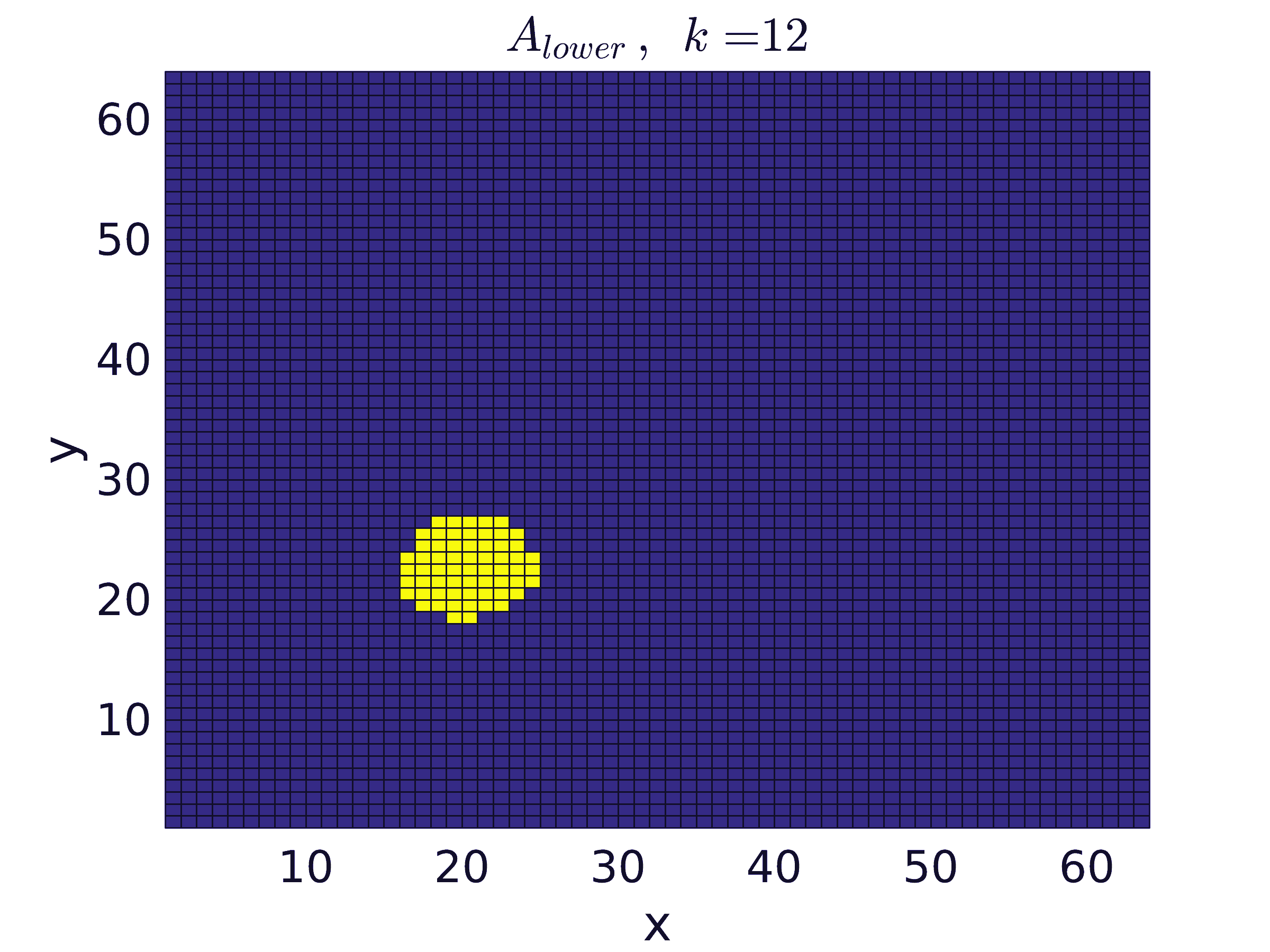}
\hspace*{0.01\textwidth}
\includegraphics[width=0.31\textwidth]{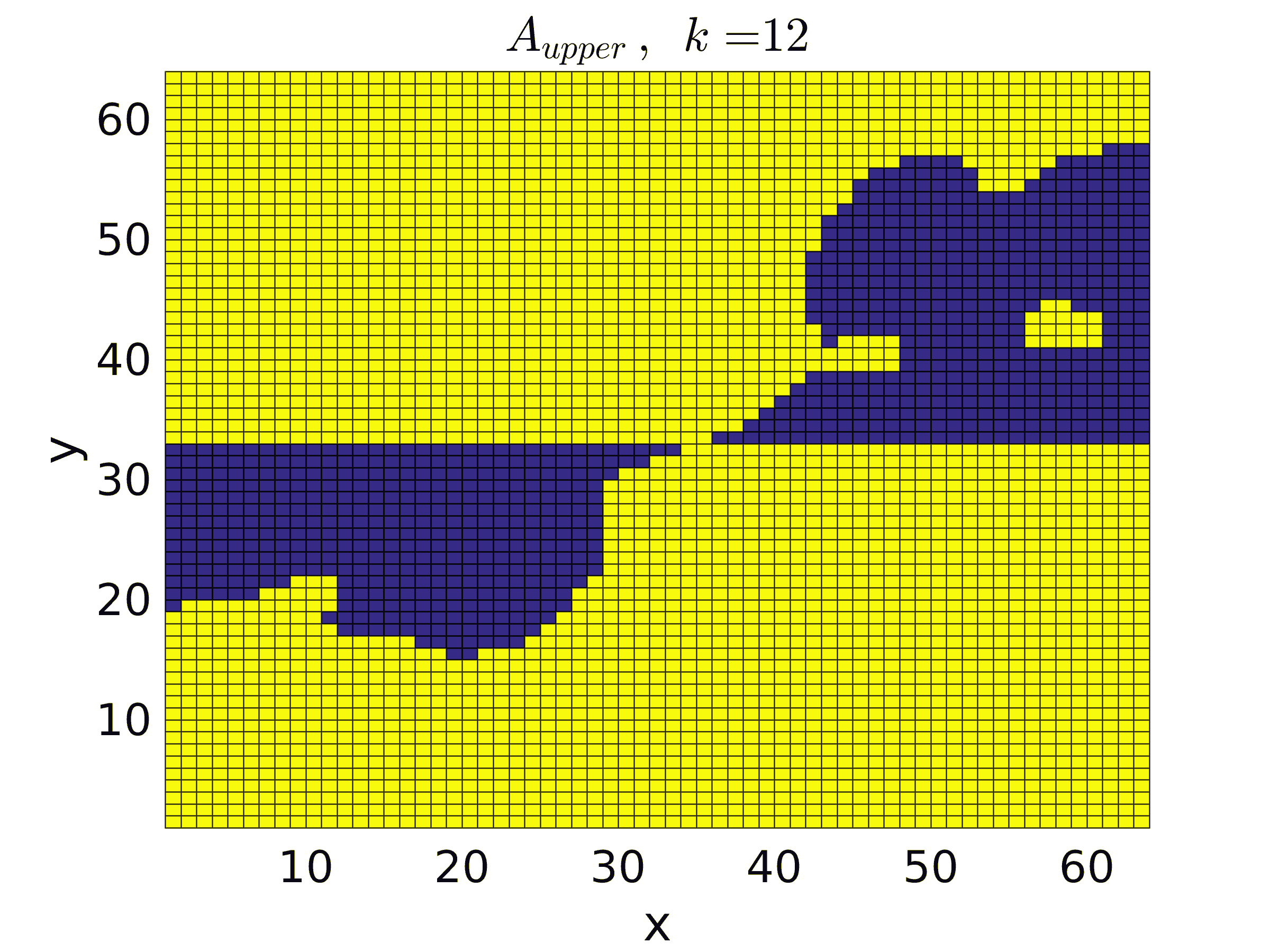}\\
\caption{
Test 3: Left: reconstructed coefficient $c_k$;
Middle: active set for lower bound;
Right: active set for upper bound;
For $k=1,4,8,12$ (top to bottom) and $\delta=0.01$.
\label{fig:convdel001_nneg}}
\end{figure}

\section{Conclusions and Remarks}\label{sec:con}
Imposing bounds both on the searched for parameter and on the data misfit is shown to provide an efficient tool for stabilizing inverse problems. The box constrained minimization problems resulting after discretization can be efficiently solved by a Gauss-Newton type SQP approach using recently developed active set methods for box constrained strictly convex quadratic programs. This is demonstrated by three examples of coefficient idenitification in elliptic PDEs.
\\
Future research in this direction might be concerned with strategies dealing with the potential nonconvexity arising due to stronger nonlinearity. (Note that the examples considered here so far can be shown to satisfy the so-called tangential cone condition and are therefore only mildly nonlinear.)
\\
Also investigations on convergence of iterative regularization methods
based on formulation \eqref{minJRaaoM} in an infinite dimensional
function space setting would be of interest. However, this requires to
work in nonreflexive spaces both in parameter and in data space, which
makes an analysis challenging.

\section*{Acknowledgment}
The authors gratefully acknowledge financial support by the Austrian
Science Fund FWF under the grants I2271 ``Regularization and
Discretization of Inverse Problems for PDEs in Banach Spaces'' and
P30054 ``Solving Inverse Problems without Forward Operators'' as well
as partial support by the Karl Popper Kolleg
``Modeling-Simulation-Optimization'', funded by the
Alpen-Adria-Universit\" at Klagenfurt and by the Carin\-thian Economic
Promotion Fund (KWF).

\end{document}